\documentclass[10pt]{article}
\usepackage[english]{babel}
\usepackage[latin1]{inputenc}
\usepackage{amsmath}
\usepackage{amsfonts}
\usepackage{amssymb}
\usepackage{amsthm}
\usepackage{mathrsfs}

\newtheorem{thm}{Theorem}
\newtheorem{defn}[thm]{Definition}
\newtheorem{rem}[thm]{Remark}
\newtheorem{cor}[thm]{Corollary}
\newtheorem{prop}[thm]{Proposition}
\newtheorem{lem}[thm]{Lemma}

\begin{document}

\title{Random walks in $(\mathbb{Z}_{+})^{2}$
with non-zero drift\\ absorbed at the axes}

\author{Irina Kurkova\footnotemark[1] 
        \and
        Kilian Raschel\footnotemark[1] }

\date{\today}

\maketitle

\footnotetext[1]{Laboratoire de Probabilit\'es et
        Mod\`eles Al\'eatoires, Universit\'e Pierre et Marie Curie,
        4 Place Jussieu 75252 Paris Cedex 05, France. 
        E-mails~: \texttt{irina.kourkova@upmc.fr},
                  \texttt{kilian.raschel@upmc.fr}.}

\begin{abstract}
  Spatially homogeneous random walks in $(\mathbb{Z}_{+})^{2}$ with
  non-zero jump probabilities at distance at
  most $1$, with non-zero drift in the interior of the quadrant
  and absorbed when reaching the axes are studied.
  Absorption probabilities generating functions are obtained
  and the asymptotic of absorption probabilities along the axes is
  made explicit. The asymptotic of the Green functions is
  computed along all different infinite paths of states, in
  particular along those approaching the axes.
\end{abstract}

\bigskip

\smallskip
\newline \noindent {\it AMS $2000$ Subject Classification : primary 60G50, 60G40 ;
secondary 30E20, 30F10.}

\section{Introduction}\label{Intro}
Random walks in angles of $\mathbb{Z}^{d}$ conditioned in the sense
of Doob's $h$-transform never to reach the boundary
nowadays arouse enough interest in the mathematical community as they appear in
several distinct domains.

 An important class of such walks is
 the so-called ``non-colliding'' random
 walks. These walks are the processes $(Z_{1}(n), \ldots, Z_{k}(n))_{n\geq 0}$
 composed of $k$ independent and identically distributed
 random walks that never leave the Weyl chamber
 $W=\{z\in \mathbb{R}^{k} : z_{1}<\cdots <z_{k} \}$.
 The distances between these random walks $U(n)=
 (Z_{2}(n)-Z_{1}(n),\ldots, Z_{k}(n)-Z_{k-1}(n))$
 give a $k-1$ dimensional random process whose
 components are positive.
 These processes  appear in the eigenvalue description of important
 matrix-valued stochastic processes~: see~\cite{DY62} for an old
 well-known result on the eigenvalues of the process version
 of the Gaussian Unitary Ensemble and
 e.g.~\cite{Bru91},~\cite{KO01},~\cite{KT4},~\cite{Gr99},~\cite{HW96}.
 They are found in the analysis of corner-growth model,
 see~\cite{Jo00} and~\cite{Jo02}.
 Moreover, interesting connections between non-colliding walks, random
 matrices and queues in tandem are the subject of~\cite{OC03}.
 Paper~\cite{KP} reveals a rather general mechanism
 of the construction  of the suitable $h$-transform for such
 processes.
 But processes whose components are distances between independent
 random walks are not the only class of interest.
 In~\cite{KOR02}, random walks with exchangeable
 increments and conditioned never to exit the Weyl chamber are
 considered. In~\cite{OCY},
 the authors study a certain class of random
 walks, namely $(X_{i}(n))_{1\leq i\leq k}=
   (|\{1\leq m \leq n : \xi_{m}=i\}|)_{1\leq i\leq k}$, where
   $(\xi_{m}, m\geq 1)$ is a sequence of i.i.d.\ random variables with
   common distribution on $\{1,2,\ldots, k\}$.
   The authors identify in law their conditional version with a certain
   path-transformation.
   In~\cite{ORSK1} and~\cite{ORSK2}, O'Connell relates
   these objects to the Robinson-Schensted algorithm.

 Another important area where random processes
 in angles of $\mathbb{Z}^{d}$
 conditioned never to reach the boundary
 appear is ``quantum random walks''.
 In~\cite{Bi2},  Biane
 constructs a quantum Markov chain on the von Neumann
 algebra of ${\rm SU}(n)$ and interprets the restriction
 of this quantum Markov chain to the algebra of a maximal
 torus of ${\rm SU}(n)$ as a random walk on the lattice of integral forms on ${\rm SU}(n)$
 with respect to this maximal torus. He proves that the restriction of the quantum
 Markov chain to the center of the von Neumann algebra
 is a Markov chain on the same lattice  obtained from the
 preceding by conditioning it in Doob's sense
 to exit a Weyl chamber at infinity.
 In~\cite{Bi3}, Biane extends these results
 to the case of general semi-simple connected and simply connected
 compact Lie groups, the basic notion being that of the minuscule weight.
 The corresponding random walk on the weight lattice in the
 interior of the Weyl chamber  can be obtained as follows~:
 if $2l$ is the order of the associated Weyl group,
 one draws the vector corresponding to the minuscule weight
 and its $l-1$ conjugates under
 the Weyl group~; then one translates these  vectors
 to each point of the weight lattice in the interior of
 the Weyl chamber and  assigns to them equal probabilities of jumps
 $1/l$.

 For example, in the case ${\rm U}(3)$, the Weyl chamber
 of the corresponding Lie algebra $\mathfrak{sl}_{3}(\mathbb{C})$
 is the ``angle $\pi/3$'', that is to say the
 domain of $(\mathbb{R}_{+})^{2}$ delimited on the one hand
 by the $x$-axis and on the other by the axis making an
 angle equal to $\pi/3$ with the $x$-axis.
 One gets a spatially homogeneous random walk
 in the interior of the weights lattice,
 as in the left-hand side of Picture~\ref{Flatto_Lie},
 the arrows designing transition probabilities equal to $1/3$.
 In the cases of the Lie algebras $\mathfrak{sp}_{4}(\mathbb{C})$
 or $\mathfrak{so}_{5}(\mathbb{C})$,
 the Weyl chamber is the angle $\pi/4$, see 
 the second picture of Figure~\ref{Flatto_Lie}
 for the transition probabilities.
 Both of these random walks can be
 of course thought as walks in $(\mathbb{Z}_{+})^2$
 with transition probabilities
 drawn in the third and fourth pictures of Figure~\ref{Flatto_Lie}.
 \begin{figure}[!ht]
 \begin{picture}(10.00,80.00)
 \includegraphics{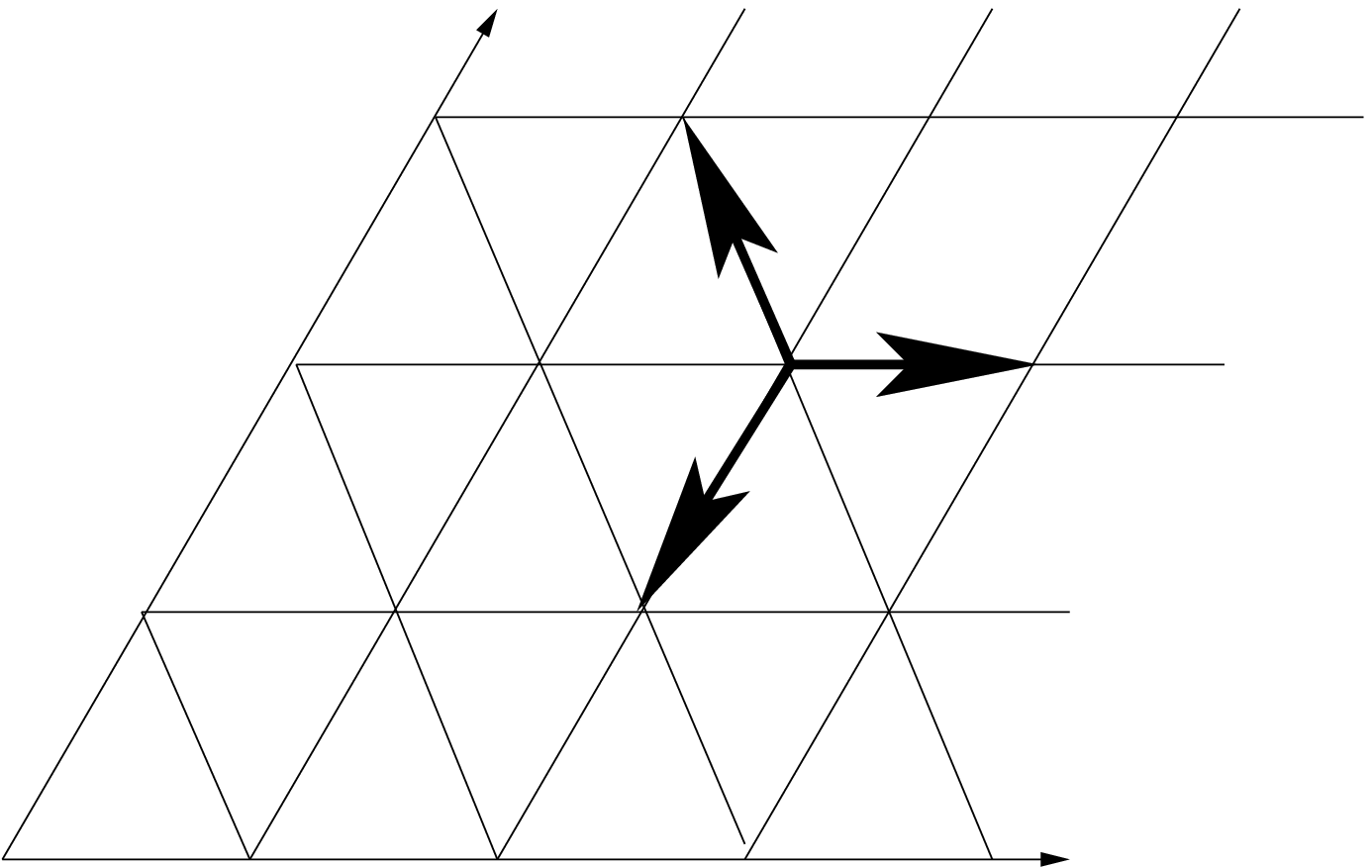}\hspace{36mm}
 \includegraphics{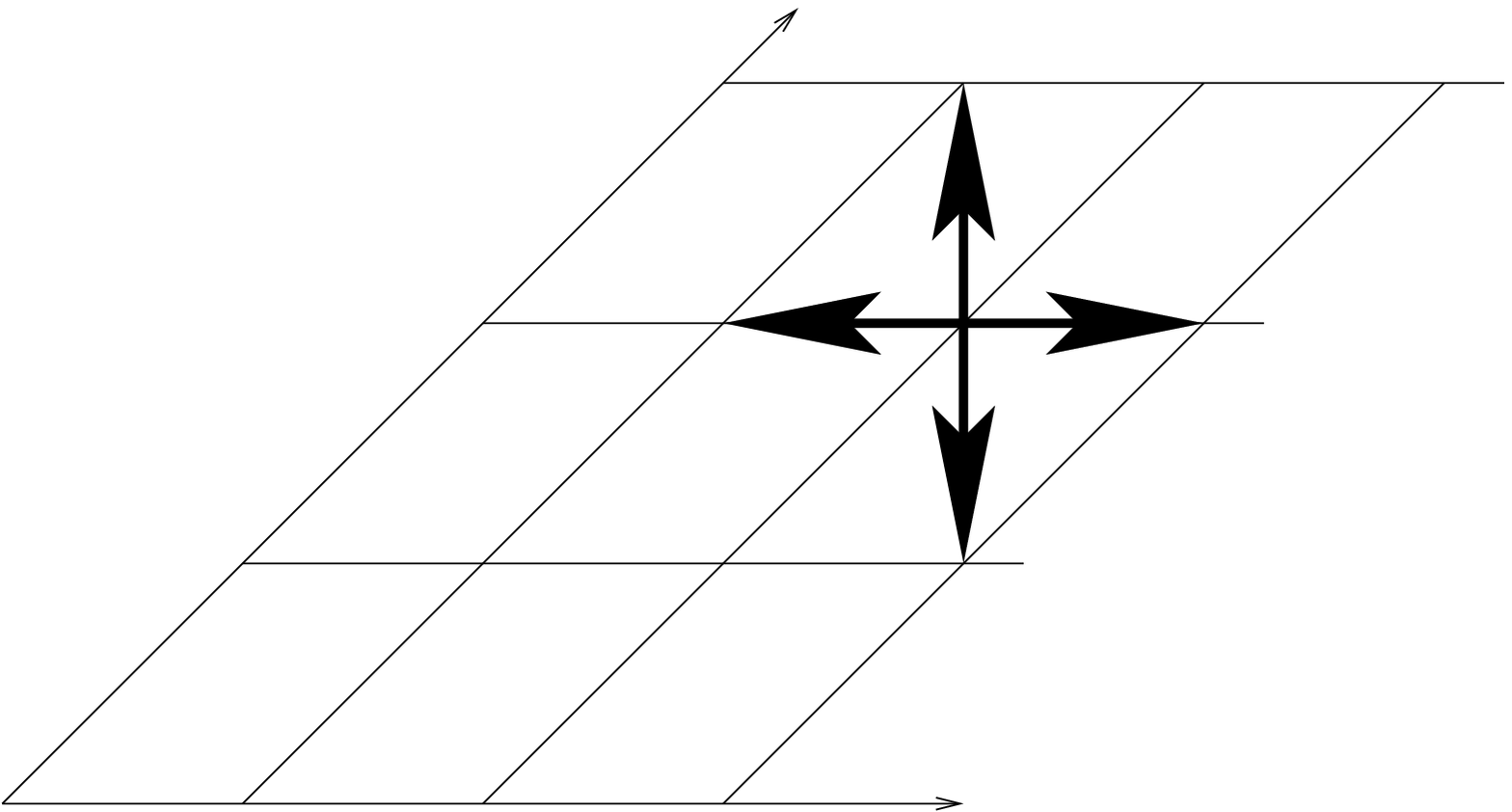}\hspace{46mm}
 \includegraphics{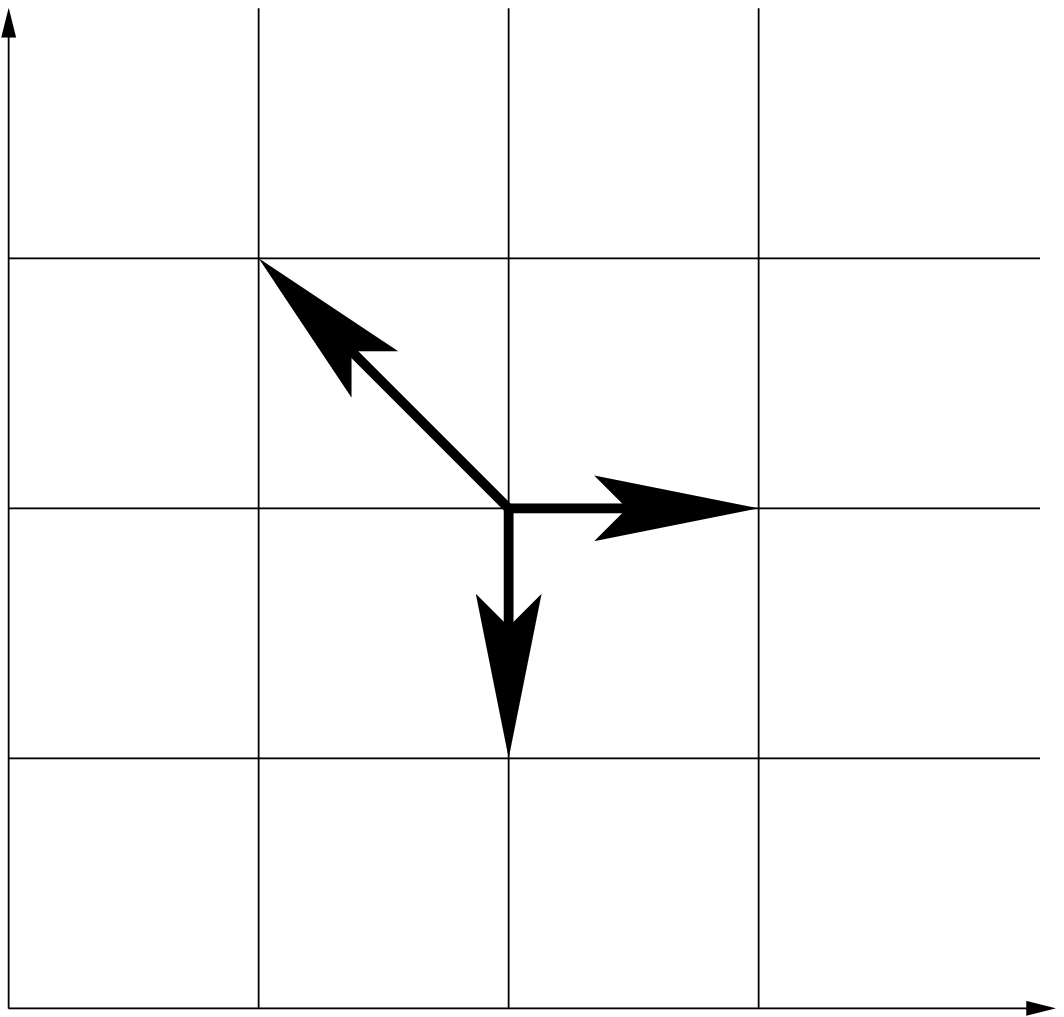}\hspace{36mm}
 \includegraphics{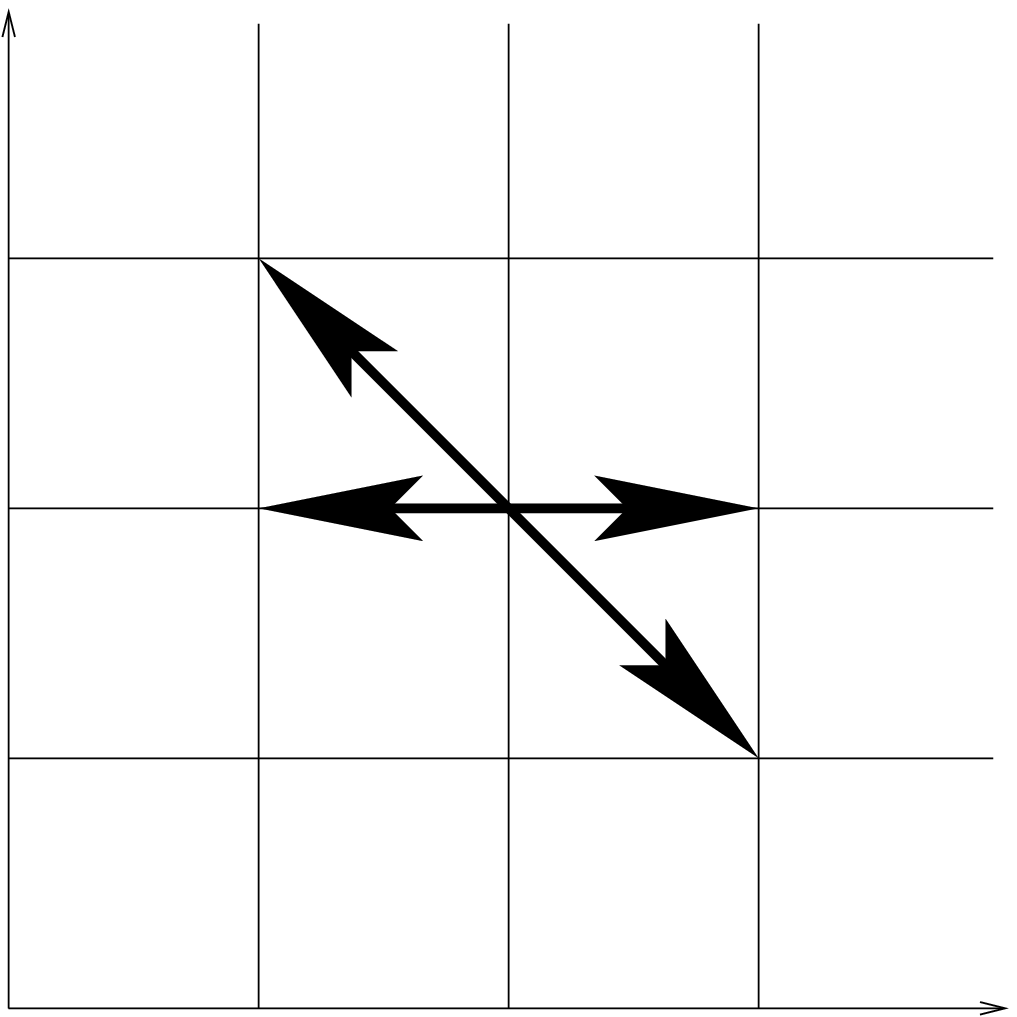}
 \end{picture}
 \caption{The walks on weights lattice of classical
 algebras --above, $\mathfrak{sl}_{3}(\mathbb{C})$ and
 $\mathfrak{sp}_{4}(\mathbb{C})$-- can be viewed
 as random walks on ${\mathbb{Z}_{+}}^{d}$}
 \label{Flatto_Lie}
 \end{figure}
 Biane shows that the  suitable Doob's $h$-transform
 $h(x,y)$ for these random walks is the dimension of the representation
 with highest weight $(x-1,y-1)$.
 In~\cite{Bi1}, again thanks to
 algebraic methods, he
  computes the asymptotic of the Green functions
  $G_{x,y}$ for the random walk with jump probabilities $1/3$ in the angle $\pi/3$
  on the Picture~\ref{Flatto_Lie},
  absorbed at the boundary,
  $x,y \to \infty$ and $y/x \to \tan( \gamma)$,
  $\gamma $ lying in $\in [\epsilon, \pi/2-\epsilon]$,
  $\epsilon>0$. The asymptotic of the Green functions
  as $y/x \to 0$ or $y/x \to \infty$ could not be found by these technics.

 In~\cite{Bi3} Biane also studies some extensions to random walks
with drift~: these are spatially homogeneous random walks in the
same Weyl chambers, with the same non-zero jump probabilities as
previously, but now these jump probabilities are admitted not to be all
equal to $1/l$, so that the mean drift
vector may have positive coordinates. Due to Choquet-Denis
theory, in~\cite{Bi3}, he
 finds all minimal non-negative harmonic functions
 for these random walks. Nevertheless this approach
 seems not allow to find the 
 Martin compactification of these random walks,
 nor to compute the asymptotic of the Green functions along
 different paths.

  In~\cite{I}, Ignatiouk-Robert obtains, under general assumptions and
  for all $d\geq 2$, the Martin boundary of some
  random walks in the half-space $\mathbb{Z}^{d-1} \times \mathbb{Z}_{+}$
  killed on the boundary.
  In this paper and in~\cite{II}, Ignatiouk-Robert proposes
  a new large deviation approach to the analysis of the Martin
  boundary combined with
  the ratio-limit theorem for Markov-additive processes.
  Ignatiouk-Robert and Loree develop this original approach
  in a recent  paper~\cite{IL09} and apply
  it with success to the analysis of
  spatially homogeneous random walks in $(\mathbb{Z}_{+})^{2}$
  killed at the axes, under hypotheses of
  unbounded jump probabilities (more precisely, having exponential decay)
  and non-zero drift.
  They compute the Martin compactification for these random walks
  and therefore obtain the full Martin boundary.
  These methods seem not to be powerful for a more detailed
  study, as for the computation of the asymptotic of the Green functions, or for
  the computation of the absorption probabilities at different points on the axes,
  or for the enumeration of lattice walks (see~\cite{BM1} and references therein
  for the study of this last
  problem for lattice walks on $(\mathbb{Z}_{+})^{2}$ by analytic methods). 
  They also seem to be difficult to generalize to the random
  walks with zero drift.

   In this paper we would like to study in detail
   the spatially homogeneous random walks  $(X(n), Y(n))_{n\geq 0}$
   in $(\mathbb{Z}_{+})^2$ with jumps at distance at most $1$.
   We denote by
   $\mathbb{P}(X(n+1)=i_{0}+i, Y(n+1)=j_{0}+j \mid X(n)=i_{0},
   Y(n)=j_{0})=p_{(i_{0},j_{0}), (i+i_{0}, j+j_{0})}$
   the transition probabilities and do the hypothesis~:
\begin{itemize}
     \item[(H1)] {\it For all $\left(i_{0},j_{0}\right)$ such that $i_{0}>0,
                 j_0>0$, $p_{\left(i_{0},j_{0}\right) ,\left(i_{0},j_{0}\right)+(i,j)}$
                 does not depend on $\left(i_{0},j_{0}\right)$
                 and can thus be denoted by $p_{i j}$.}
     \item[(H2)] {\it  $p_{i j}=0$ if $|i|>1$ or $|j|>1$.}
     \item[(H3)] {\it The boundary
                 $\left\{\left(0,0\right)\right\} \cup \left\{ \left(i,0\right):
                 \hspace{1mm}i\geq 1 \right\}
                 \cup \left\{ \left(0,j\right):\,j\geq 1 \right\}$
                 is absorbing.}
    \item[(H4)] {\it In the list $p_{11},p_{10},p_{1-1},p_{0-1},
                 p_{-1-1},p_{-1,0},p_{-11},p_{0-1}$ there are no
                 three consecutive zeros.}
\end{itemize}
The last hypothesis (H4) is purely technical and avoids  studying
degenerated random walks.

   In a companion paper
   we gave a rather complete analysis of such random walks under a
   simplifying hypothesis that also
\begin{itemize}
\item[(H2')] $p_{-11}=p_{11}=p_{1-1}=p_{-1-1}=0$.
\end{itemize}
This hypothesis made the analysis more transparent for several
reasons. First, the problem to find the generating functions of
absorption probabilities, defined
in~(\ref{absorption_probabilities})
and~(\ref{def_generating_functions}) below, could be reduced to the
resolution of Riemann boundary value problems on contours inside
unit discs, where these functions are holomorphic, being generating
functions of probabilities. These contours under general hypothesis
(H2) may lie outside the unit disc, so that we are obliged first to
continue these functions as holomorphic, then to exploit this
continuation. Secondly, the conformal gluing function responsible for
the conversion between Riemann boundary value problem and
Riemann-Hilbert problem we are faced with, has a particularly nice form.
Thirdly, this hypothesis corresponds to an easy group of Galois
automorphisms in terms of the book~\cite{FIM}, as it is of order
four. Fourthly, under (H2'), in the case of non zero drift,
  the asymptotic of the Green functions
  $G_{i,j}^{n_{0},m_{0}}$ (recall that these are the mean numbers of visits
  to $(i,j)$ starting from $(n_{0},m_{0})$)
  could be found very easily by means of Paper~\cite{KV}.



  Let $M_{x}=\sum_{i,j} i p_{i,j},\  M_{y}=\sum_{i,j}j
  p_{i,j}$ be the coordinates of the mean drift vector.
  If $M_{x}>0$, $M_{y}>0$, the time $\tau$ of absorption by the axes
  for this random walk
  is infinite with positive probability for any initial state $(n_{0},m_{0})$,
  $n_{0},m_{0}>0$. It is immediate that the suitable Doob's $h$-transform conditioning
  the process never to reach the axes is of course this probability
  and also that this $h$-transformed process is equal in
  distributional limit to the conditional process given $\{\tau>k\}$
  as $k \to \infty$.
  What is the suitable $h$-transform under the hypothesis
  $M_{x}=M_{y}=0$ and does this last statement stay true in this case~?
  The answer to this question in~\cite{Ras}
  under (H2') came from the study of the exact tail's asymptotic
  of $\tau$  as $k\to \infty$. Namely it was shown that
  $P_{(n_{0},m_{0})}(\tau>k)\sim C n_{0}m_{0} k^{-1}$, where $C$ does not
  depend on $n_{0},m_{0}$ and
  $n_{0}m_{0}$ is the unique
  non-negative harmonic function.
  Since extensions of these results on $\tau $ when (H2') 
is relaxed are rather voluminous,
  we restrict ourselves in this paper to the case of positive drift~:
\begin{itemize}
\item[(M)] $M_{x}=\sum_{i,j} i p_{i,j}>0$,
           $M_{y}=\sum_{i,j} j p_{i,j}>0$,
\end{itemize}
   and postpone to a future work the study
   of the random walks
   with zero drifts in the Weyl chambers of
   $\mathfrak{sl}_{3}(\mathbb{C})$ and $\mathfrak{sp}_{4}(\mathbb{C})$.

It is the book~\cite{FIM} that gave us the main tool of analysis
and has therefore inspired this paper.
This book studies the random walks in
$(\mathbb{Z}_{+})^{2}$ under assumptions (H1) and (H2) but not
(H3)~: the jump probabilities from the boundaries to the interior of
$(\mathbb{Z}_{+})^2$ are there not zero and the $x$-axis, the
$y$-axis and $(0,0)$ are three other domains of spatial homogeneity.
Moreover, the jumps from the boundaries are supposed such that the
Markov chain is ergodic. The authors G.~Fayolle, R.~Iasnogorodski
and V.~Malyshev elaborate a profound and ingenious analytic approach to compute the
generating functions of stationary probabilities of these random
walks. This approach serves as a starting point for our investigation and
by this reason plays an absolutely crucial role~:
preparatory Subsections~\ref{Introduction},~\ref{The_algebraic_curve_Q}
and~\ref{Galois_automorphisms} proceed along the book~\cite{FIM}
applied for the random walks killed at the boundary.

In Subsection~\ref{Finite_groups_and_s_s_functions}, using this analytic approach,
we analyze the absorption probability.Let  
     \begin{equation}\label{absorption_probabilities}
          \left.\begin{array}{ccc}
               h_{i}^{(n_{0},m_{0})} & = & \mathbb{P}_{\left(n_{0} , m_{0}\right)}
               \left(\text{to be absorbed at}\, (i,0)\right),\\
               \widetilde{h}_{j}^{(n_{0},m_{0})} & = & \mathbb{P}_{\left(n_{0} , m_{0}\right)}
               \left(\text{to be absorbed at}\, (0,j)\right),\\
               h_{00}^{(n_{0},m_{0})} & = & \mathbb{P}_{\left(n_{0} , m_{0}\right)}
               \left(\text{to be absorbed at}\, (0,0)\right).
          \end{array}\right.
     \end{equation}
be the probabilities of being absorbed at points $(i,0)$,
$(0,j)$ and $(0,0)$ starting from $(n_{0},m_{0})$.
Let $h^{n_{0},m_{0}}(x)$ and $\tilde{h}^{n_{0},m_{0}}(y)$
be their generating functions, initially defined for $|x|\leq 1$ and $|y|\leq 1$~:
     \begin{equation}\label{def_generating_functions}
               h^{n_{0},m_{0}}\left(x\right) =
               \sum_{i\geq 1 } h_{i}^{n_{0},m_{0}} x^{i} ,\hspace{10mm}
               \widetilde h^{n_{0},m_{0}}\left(y\right) =
               \sum_{j \geq 1 } \widetilde h^{n_{0},m_{0}}_{j} y^{j} .
     \end{equation}
  When no ambiguity on the initial state can arise, we drop
  the index $(n_{0},m_{0})$ and write
  $h_{i}$, $\tilde{h}_{j}$, $h_{00}$, $h(x)$, $\tilde{h}(y)$
  respectively.

 In Section~\ref{Explicit_form} the
 generating functions $h(x)$, $h(y)$ and $h_{00}$
 are computed.
 Subsection~\ref{Riemann_boundary_value_problem_with_shift}
 gives the first integral representation
 of these functions on a smooth curve, which is almost 
 directly deduced from~\cite{FIM}.
 In Subsection~\ref{Riemann_Hilbert_problem} we look closer at the conformal
 gluing function and transform this representation
 into one on a real segment, that suits better
 for further analysis, see
 Theorem~\ref{explicit_h(x)_second}.

In Section~\ref{Study_solution} we deduce the asymptotic of the
absorption probabilities $h_{i}$ and $\tilde{h}_{j}$ as $i \to
\infty$ and $j\to \infty$. We show that $h_{i} \sim  C(n_{0},m_{0})
p^{-i}i^{-3/2}$, with some (made explicit) $p>1$ and a constant
$C(n_{0},m_{0})$. This constant $C(n_{0},m_{0})$ 
is also made explicit and turns
out to depend quite interestingly on the ``group of Galois
automorphisms'' of the random walk (in the sense of
Definition~\ref{defigr}), see
Theorem~\ref{main_result_asymptotic_absorption_probabilities}.

In Section~\ref{Green_functions_Martin_boundary} we compute the
asymptotic of the Green functions $G_{i,j}^{n_{0},m_{0}}$ that is of
the mean number of visits to $(i,j)$ starting from $(n_{0},m_{0})$
as $i,j \to \infty$, $j/i \to \tan(\gamma)$ where $\gamma \in [0,
\pi/2]$.

 In the case of $\gamma \in ]0, \pi/2[$, thanks to~\cite{KV} 
 and~\cite{Ma1}, it is not a difficult task~:
 the~procedure used in~\cite{KV}
 for the Green functions asymptotic (and in fact developed much earlier in~\cite{Ma1}
 for the stationary probabilities' asymptotic)
 of the random walks in the quadrant under the simplifying
 hypothesis (H2') in the interior
 and with some non-zero jump probabilities from
 the axes  can be rather easily generalized to our random
 walks under (H2). To state the result,
   let $(u(\gamma),v(\gamma))$
   be the unique solution of
   $\text{grad}(\phi(u,v))/
   |\text{grad}(\phi(u,v))|=
   (\cos(\gamma),\sin(\gamma))$
   on $\phi(u,v)=\sum_{i,j}p_{i j}e^{i u}e^{j v}=1$.
   Let $s_{x}(\tan(\gamma))=\exp(u(\gamma))$ and $
   s_{y}(\tan(\gamma)) =\exp(v(\gamma))$.
   Then $G_{i,j}^{n_{0},m_{0}} \sim  C i^{-1/2}
   \big[s_{x}(\tan(\gamma))^{n_{0}}s_{y}(\tan(\gamma))^{m_{0}}-h(s_{x}(\tan(\gamma)))-
   \tilde{h}(s_{y}(\tan(\gamma)))-h_{00}\big]s_{x}(j/i)^{-i}s_{y}(j/i)^{-j}$,
   the constant $C$ does not depend on $n_{0}$, $m_{0}$, $i$, $j$ and is
   made explicit, see Theorem~\ref{thm_asymptotic_Green_functions_interior}.

  It is a more delicate task
  to study the asymptotic of
  the Green functions $G_{i,j}^{n_{0},m_{0}}$  in the case of
  $j/i \to 0$ (or $j/i \to \infty$) that has not been
  completed in previous works.  This is the subject
  of Subsection~\ref{Asymptotic_Green_functions_gamma_zero}.
  In Theorem~ \ref{proposition_asymptotic_Green_functions_angle_zero} we prove that
  $G_{i,j}^{n_{0},m_{0}} \sim
  C_{0} i^{-1/2} j/i \big[m_{0} s_{x}(0)^{n_{0}}
           s_{y}(0)^{m_{0}-1}-\tilde{h}'(s_{y}(0)
           )\big]
          s_{x}(j/i)^{-i}
           s_{y}(j/i)^{-j}
      $ where
      $C_{0}$ is independent of $n_{0}$ and $m_{0}$ and is
      made explicit.
   The explicit expression of
  $m_{0} s_{x}(0)^{n_{0}}s_{y}(0)^{m_{0}-1}-\tilde{h}'(s_{y}(0))$
   in terms of the parameters $p_{i j}$
   depends interestingly
   on the order of  the ``group of Galois automorphisms'', see Remark~\ref{expl}.
       The result for $j/i \to \infty$ follows after exchanging $i$
       and $j$.
   The limit of $G_{i,j}^{n_{0},m_{0}}/G_{i,j}^{n_{1},m_{1}}$,
   as $i,j>0$, $j/i \to 0$ is the same
   as  the limit 
   $h_{i}^{n_{0},m_{0}}/h_{i}^{n_{1},m_{1}}$ when
   $i\to \infty$ from Section~\ref{Study_solution}.
   Furthermore the limits of $G_{i,j}^{n_{0},m_{0}}/
   G_{i,j}^{n_{1},m_{1}}$ when  $j/i \to \tan(\gamma)$, $\gamma \in
   [0, \pi/2]$, provide  explicitly all harmonic functions of the Martin
   compactification. This leads in particular to  the result
   recently obtained in~\cite{IL09} that the Martin boundary
   is homeomorphic to $[0, \pi/2]$.

\bigskip

\noindent{\bf Acknowledgements.}

We would like to thank Professor Bougerol for pointing out the 
interest and relevance of the topic discussed in this paper and for 
the numerous and stimulating discussions we had 
concerning the random walks under consideration.

Thanks to Professor Biane for clarifying the results he obtained 
on ``Quantum random walks'' and for the interesting
discussions we had together around this topic. 

Many thanks also to M.\ Defosseux and F.\ Chapon for the generosity with 
which they shared their mathematical knowledge with us.

\section{Analytic approach}\label{h_x_z}

\subsection{A functional equation}\label{Introduction}
Define
     \begin{equation}
          G\left(x,y\right) = \sum_{i,j\geq 1} G_{i,j}^{n_{0},m_{0}}
          x^{i-1} y^{j-1}=\sum_{i,j\geq 1}\mathbb{E}_{\left(n_{0} , m_{0}\right)}
          \left[ \sum_{n\geq 0}1_{\left\{
          \left(X\left(n\right),Y\left(n\right)\right)=
          \left(i,j \right) \right\}}\right]x^{i-1}y^{j-1},
     \end{equation}
the generating function of the Green functions. With the notations
of Section~\ref{Intro}, we can state the following functional
equation~:
     \begin{equation}\label{functional_equation}
          Q\left(x,y\right)G\left(x,y\right) =
          h\left(x\right)+\widetilde{h}\left( y \right)
          +h_{00} -x^{n_{0}} y^{m_{0}},
     \end{equation}
where $Q$ is the following polynomial, depending only on the walk's
transition probabilities~:
     \begin{equation}\label{def_Q}
          Q\left(x,y\right)= x y\left( \sum_{ i,j }
          p_{i j }x^{i} y^{j}  -1 \right).
     \end{equation}
\textit{Prima facie}, Equation~(\ref{functional_equation}) has a
meaning in $\{x,y\in \mathbb{C} : |x|<1, |y|<1 \}$. The proof
of~(\ref{functional_equation}) comes from writing that for $k,l,n\in
(\mathbb{Z}_{+})^2$, 
               \begin{eqnarray*}
                \mathbb{P}((X(n+1),Y(n+1))=(k,l))=
                \sum_{i,j\geq 1} \mathbb{P}\left(\left(X\left(n\right),Y\left(n\right)\right)
                    =\left(i,j\right) \right)
                    p_{\left(i,j\right),\left(k,l\right)}+\hspace{20mm}\\                   
                    +\sum_{i\geq 1}\mathbb{P}\left(\left(X\left(n\right),Y\left(n\right)\right)
                    =\left(i,0\right) \right)
                    \delta_{\left(k,l\right)}^{\left(i,0\right)}+
                    \sum_{j\geq 1}\mathbb{P}\left(\left(X\left(n\right),Y\left(n\right)\right)
                    =\left(0,j\right) \right)
                    \delta_{\left(k,l\right)}^{\left(0,j\right)} +    \\           
                    +\mathbb{P}\left(\left(X\left(n\right),Y\left(n\right)\right)
                    =\left(0,0\right) \right)
                    \delta_{\left(k,l\right)}^{\left(0,0\right)},\hspace{20mm}\hspace{20mm}\hspace{5mm}
               \end{eqnarray*}
where $\delta^{(i,j)}_{(k,l)}=1$ if $i=k$ and
$j=l$, otherwise $0$. It remains to multiply by $x^{k} y^{l}$ and
then to sum with respect to $k,l,n$.

\subsection{The algebraic curve $Q\left(x,y \right)=0$}
\label{The_algebraic_curve_Q}

The polynomial~(\ref{def_Q}) can be written alternatively~:
     \begin{equation}\label{def_Q_alternative}
          Q\left(x,y\right) = a\left(x\right) y^{2}+ b\left(x\right) y
          + c\left(x\right) = \widetilde{a}\left(y\right) x^{2}+
          \widetilde{b}\left(y\right) x + \widetilde{c}\left(y\right),
     \end{equation}
where
     \begin{equation*}
          \left.\begin{array}{cccccc}
               a\left(x\right) &=& p_{11}x^{2}+p_{01}x+p_{-11},&
               \widetilde{a}\left(y\right)&=&p_{11}y^{2}+p_{10}y+p_{1-1}, \\
               b\left(x\right) &=&  p_{10}x^{2}-x+p_{-10}, &
               \widetilde{b}\left(y\right)&=&p_{01}y^2-y+p_{0-1}, \\
               c\left(x\right) &=& p_{1-1}x^{2}+p_{0-1}x+p_{-1-1}, &
               \widetilde{c}\left(y\right) &=& p_{-11}y^{2}+p_{-10}y+p_{-1-1},\\
               d\left(x\right)&=&b\left(x\right)^{2}-4a\left(x\right)c\left(x\right),&
               \widetilde{d}\left(y\right)&=&\widetilde{b}\left(y\right)^{2}-
               4\widetilde{a}\left(y\right)\widetilde{c}\left(y\right).
          \end{array}\right.
     \end{equation*}

We will now build the algebraic function $Y(x)$ defined by
$Q(x,y)=0$. Note first that $Q(x,y)=0$ is equivalent to
$(b(x)+2a(x)y)^{2}=d(x)$, so that the construction of the function
$Y$ is equivalent to that of the square root of the polynomial $d$.
%
We need the following precisions on the roots of $d$~:

     \begin{lem}\label{lemma_branched_points}
          (1) $d$ is a third or fourth degree polynomial, whose all roots are real
              and mutually distinct.
          (2) We call its roots the $x_{i}$, $i\in\{1,\ldots ,4\}$,
              with eventually $x_{4}=\infty$
              if $\deg(d)=3$. It turns out that
              there are two possibilities~: either the modulus of the roots are
              mutually distinct and in this case we
              enumerate the roots in such a way that
              $|x_{1}|<|x_{2}|<|x_{3}|<|x_{4}|$,
              or there are two pairs of roots and inside of each pair
              the roots are opposed one from the other, in this case
              we enumerate them $0<x_{2}=-x_{1}<x_{3}=-x_{4}$.
              This last case corresponds to the walks having transition probabilities
              such that $p_{11}+p_{-1-1}+p_{1-1}+p_{-11}=1$.
          (3) Moreover, $|x_{1}|<1$, $|x_{2}|<1$ and $|x_{3}|>1$, $|x_{4}|>1$.
          (4) $x_{2}$ and $x_{3}$ are positive.
          (5) $x_{1}=0$ (resp.\ $x_{4}=\infty$) if and only if
              $p_{-10}^{2}-4p_{-11}p_{-1-1}=0$ (resp.\ $p_{10}^{2}-4p_{11}p_{1-1}=0$).
          (6) If $p_{-10}^{2}-4p_{-11}p_{-1-1}\neq 0$
              (resp.\ $p_{10}^{2}-4p_{11}p_{1-1}\neq 0$)
              then $\textnormal{sign}(x_{1})
              =\textnormal{sign}(p_{-10}^{2}-4p_{-11}p_{-1-1})$
              (resp.\ $\textnormal{sign}(x_{4})=\textnormal{sign}
              (p_{10}^{2}-4p_{11}p_{1-1})$).
     \end{lem}

\begin{proof}
All these properties are proved in~\cite{FIM}.  Note here that it is
thanks to the hypothesis (H4), made in Section~\ref{Intro}, that the
polynomial $d$ is of degree three or four.
\end{proof}

There are two branches
of the square root of $d$. Each determination leads to a well
defined (i.e.\ single valued) and meromorphic function on the
complex plane $\mathbb{C}$ appropriately cut, that is, in our case,
on $\mathbb{C}\setminus [x_{1},x_{2}] \cup [x_{3},x_{4}]$. If
$x_{4}<0$, then $[x_{3},x_{4}]$ means $[x_{3},+\infty[\cup ]-\infty
,x_{4}]$. We can write the analytic expression of these two branches
$Y_{0}$ and $Y_{1}$ of $Y$~: $Y_{0}(x)=Y_{-}(x)$ and
$Y_{1}(x)=Y_{+}(x)$ where~:
     \begin{equation*}
          Y_{\pm }\left(x\right) = \frac{ -b\left(x\right)\pm
          \sqrt{d\left(x\right)}}{2 a\left(x\right) }.
     \end{equation*}
Just above, and in fact throughout the whole paper, we chose the
principal determination of the logarithm as soon as we use the
complex logarithm~; in this case to define the square root.

We now extend the domain of determination of $Y$ from $\mathbb{C}$
to its Riemann surface  $S$, so that $Y$ becomes single-valued on
$S$. Since there are two determinations of the square root of $d$
(opposed one from the other), the Riemann surface $S$ is formed by
$S_{0}$ and $S_{1}$, two copies of the Riemann sphere
$\mathbb{C}\cup \{\infty\}$ cut along $[x_{1},x_{2}]$ and
$[x_{3},x_{4}]$ and joined across lines lying above these cuts. This
gives a two-sheeted covering surface of $\mathbb{C}\cup \{\infty
\}$, branched over $x_{1},\ldots ,x_{4}$. By opening out the cuts in
the two sheets we see that the Riemann surface associated to $Y$ is
homeomorphic to a sphere with one handle attached, that is a Riemann
surface of genus one, a torus. For more details about the
construction of Riemann surfaces see for instance Book~\cite{SG2}.
In a similar way, the functional
equation~(\ref{functional_equation}) defines also an algebraic
function $X(y)$.  All the results concerning $X(y)$ can be deduced
from those  for $Y(x)$ after a proper change of the parameters,
namely $p_{i j}\mapsto p_{j i}$.

To conclude this part, we give a lemma that clarifies some
properties of the functions $X$ and $Y$, useful in the sequel. It is
proved in~\cite{FIM}.

     \begin{lem}\label{properties_X_Y}
          (1) $Y_{0}(1)=c(1)/a(1)$ and $Y_{1}(1)=1$.
          (2) $Y_{0}(\{x\in \mathbb{C} : |x|=1\})\subset
          \{y\in \mathbb{C} : |y|<1\}$ and
          $Y_{1}(\{x\in \mathbb{C} : |x|=1\}\setminus \{1\})\subset
          \{y\in \mathbb{C} : |y|>1\}$.
          (3) For all $x$ in this cut plane, $|Y_{0}(x)|\leq |Y_{1}(x)|$,
          with equality only on the cuts.
          (4) Suppose that $x_{4}>0$ and that the walk is non degenerated,
          see the hypothesis~(H4). If $p_{1-1}=0$, then
          $\lim_{x\to \infty}xY_{0}(x)\in ]-\infty ,0[$
          and if $p_{1-1}>0$ then
          $\lim_{x\to \infty}Y_{0}(x)\in ]-\infty ,0[$.
     \end{lem}

\subsection{Galois automorphisms and meromorphic continuation}
\label{Galois_automorphisms}

\noindent\textbf{Notation.}
In this part and throughout the whole paper, 
$\imath$ denotes the complex number~: $\imath^{2}=-1$.
\smallskip
\newline The Riemann surface $S$ associated to the algebraic function $Y$ is
naturally endowed with a covering map $\pi:S\to \mathbb{C}\cup
\{\infty \}$, such that for all $x\in \mathbb{C}\setminus
[x_{1},x_{2}] \cup [x_{3},x_{4}]$, $\pi^{-1}(x)$ is composed 
of two elements, say $s_{0}$ and $s_{1}$, such that $s_{i}\in S_{i}$,
$i=0,1$ and $\{Y(s_{0}),Y(s_{1})\}=\{Y_{0}(x),Y_{1}(x)\}$. In the
same way the Riemann surface $\tilde{S}$ associated to $X$ is
endowed with a map $\tilde{\pi} : \tilde{S}\to \mathbb{C}\cup
\{\infty \}$, such that for all $y\in \mathbb{C}\setminus
[y_{1},y_{2}] \cup [y_{3},y_{4}]$, $\tilde{\pi}^{-1}(y)$ is composed
of two elements, say $\tilde{s}_{0}$ and $\tilde{s}_{1}$, such that
$\tilde{s}_{i}\in  \tilde{S}_{i}$, $i=0,1$ and
$\{X(\tilde{s}_{0}),X( \tilde{s}_{1})\}=\{X_{0}(y),X_{1}(y)\}$.

 The surfaces $S$ and $\tilde{S}$ having the same genus,
 we consider from now on \emph{only one} surface
 $T$, conformally
 equivalent to $S$ and $\tilde{S}$, with two coverings $\pi$ and $\tilde{\pi}$.
 One can say that each $ s \in S$ has two
 (not independent) ``coordinates"
 $(x(s),y(s))$ such that $x(s)=\pi(s)$ and $y(s)=\tilde{\pi}(s)$
 and of course $Q(x(s),y(s))=0$ for all $ s \in T$.

We construct on $T$ the following covering automorphisms $\xi$ and $\eta$
defined in the previous notations by $\xi(s_{0})=s_{1}$
and $\eta(\tilde{s}_0)=\tilde{s}_1$.
Thanks to~(\ref{def_Q_alternative}), for any $s=(x,y) \in T$, $\xi$
and $\eta$ take the following explicit expressions~:
     \begin{equation}\label{def_xi_eta_Q}
          \xi\left(x,y\right)=\left(x,\frac{c\left(x\right)}
          {a\left(x\right)}\frac{1}{y}\right),\hspace{5mm}
          \eta\left(x,y\right)=\left(\frac{\widetilde{c}\left(y\right)}
          {\widetilde{a}\left(y\right)}\frac{1}{x},y\right).
     \end{equation}
$\xi$ and $\eta$ are of order two~: $\xi^{2}=\text{id}$,
$\eta^{2}=\text{id}$. In~\cite{Ma2} and~\cite{FIM}, for reasons
explained there, they are also called Galois automorphisms.
%
%
%

\begin{defn}
\label{defigr}
The group of the random walk $\mathcal{H}$ is the
group generated by $\xi$ and $\eta$.
\end{defn}

Being generated by a finite number of elements of order two,
$\mathcal{H}$ is a Coxeter group. In fact, $\mathcal{H}$ is simply a
dihedral group, since it is generated by two elements. Define
$\delta =\eta \xi$. Then the order of $\mathcal{H}$ is equal to
$2\inf\{n\in \mathbb{N}^{*} : \delta^{n}=\text{id}\}$, it can be
eventually infinite. The finiteness of this group, and in that event
its order, will turn out to be decisive in the sequel, notably in
Subsection~\ref{Finite_groups_and_s_s_functions} and Section~\ref{Study_solution}.

As implied in~\cite{FIM}, it is quite difficult to characterize
geometrically the walks having an associated group $\mathcal{H}$ of
order $2n$, except for little orders. That is how in~\cite{FIM} is
proved that $\mathcal{H}$ is of order four if and only if
     \begin{equation}\label{def_Delta}
          \Delta = \left| \begin{array}{ccc}
          p_{11}&p_{10}&p_{1-1}\\
          p_{01}&-1&p_{0-1}\\
          p_{-11}&p_{-10}&p_{-1-1}
          \end{array}\right|
     \end{equation}
is equal to zero. In particular this is the case of the walks having
transition probabilities verifying
$p_{10}+p_{-10}+p_{01}+p_{0-1}=1$, that we have studied
in~\cite{Ras}. It is also proved in~\cite{FIM} that the the walks
with transition probabilities $p_{-11}+p_{10}+p_{0-1}=1$, in the
Weyl chamber of $\mathfrak{sl}_{3}(\mathbb{C})$, see
Figure~\ref{SRW_Flatto_X}, have a group of order six for any values
of the parameters. As for the walks in the Weyl chamber of
$\mathfrak{sp}_{4}(\mathbb{C})$, see Figure~\ref{Flatto_Lie}, they
have, except for exceptional values of the parameters, a group of
order infinite. We add here that the walks with $p_{11}=p_{10}$,
$p_{-1-1}=p_{-10}$, $p_{11}+p_{-1-1}=1/2$, drawn in
Figure~\ref{SRW_Flatto_X}, have a group of order eight.
\begin{figure}[!ht]
\begin{picture}(10.00,75.00)
\includegraphics{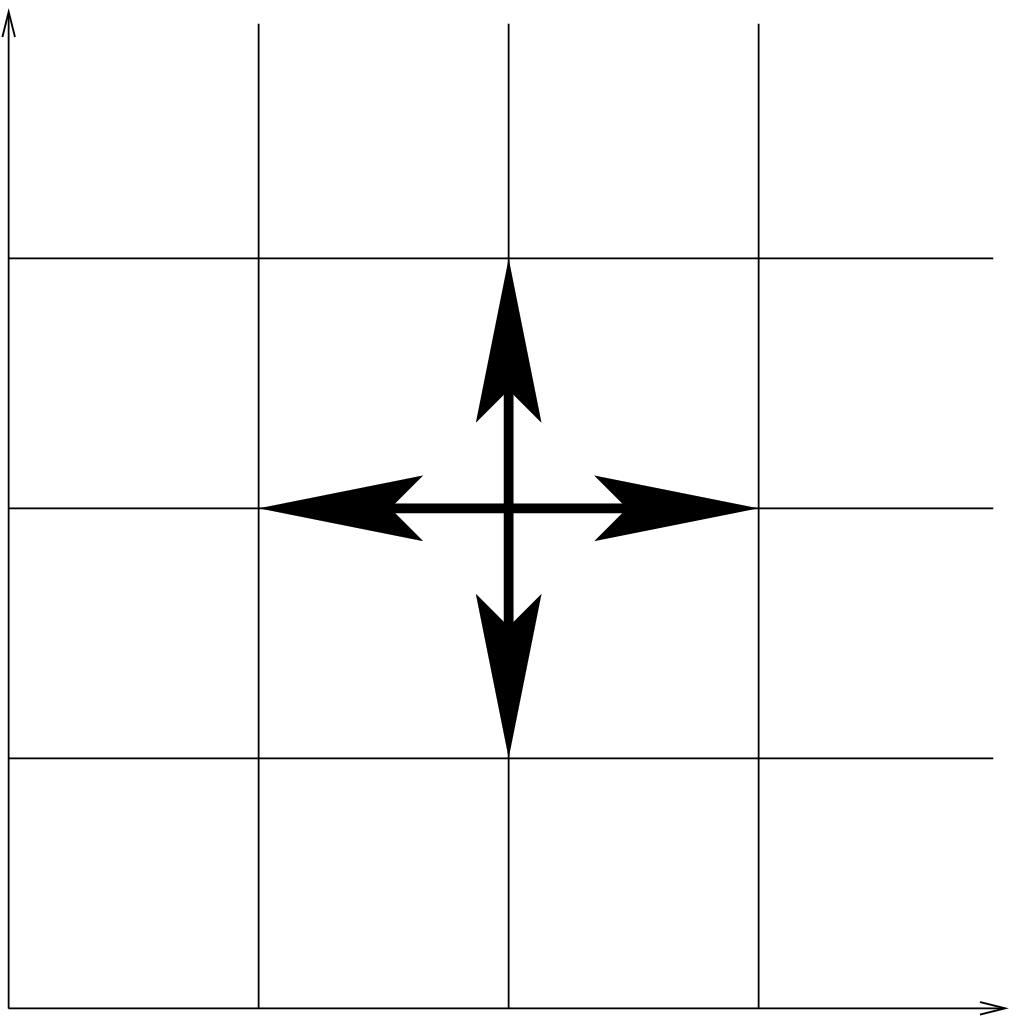}\hspace{60mm} \includegraphics{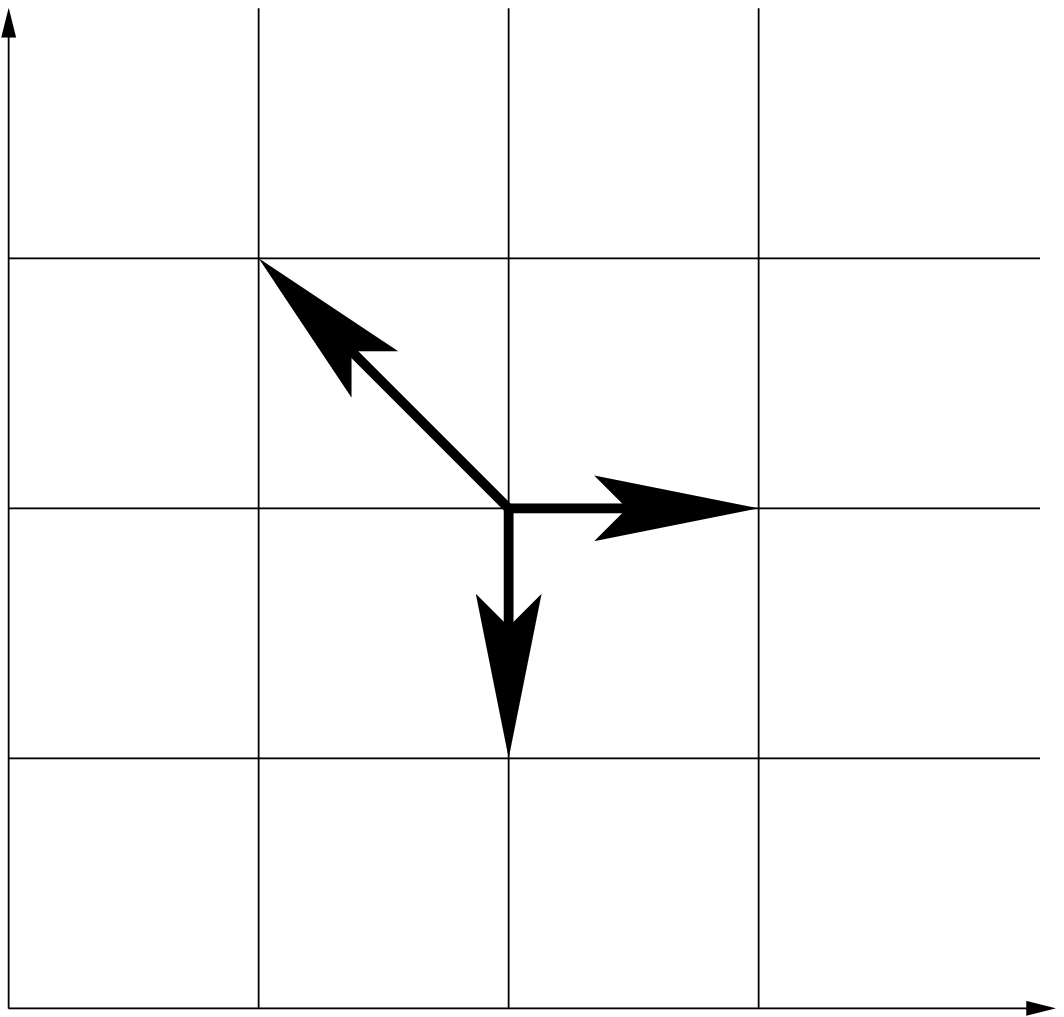}\hspace{60mm}
\includegraphics{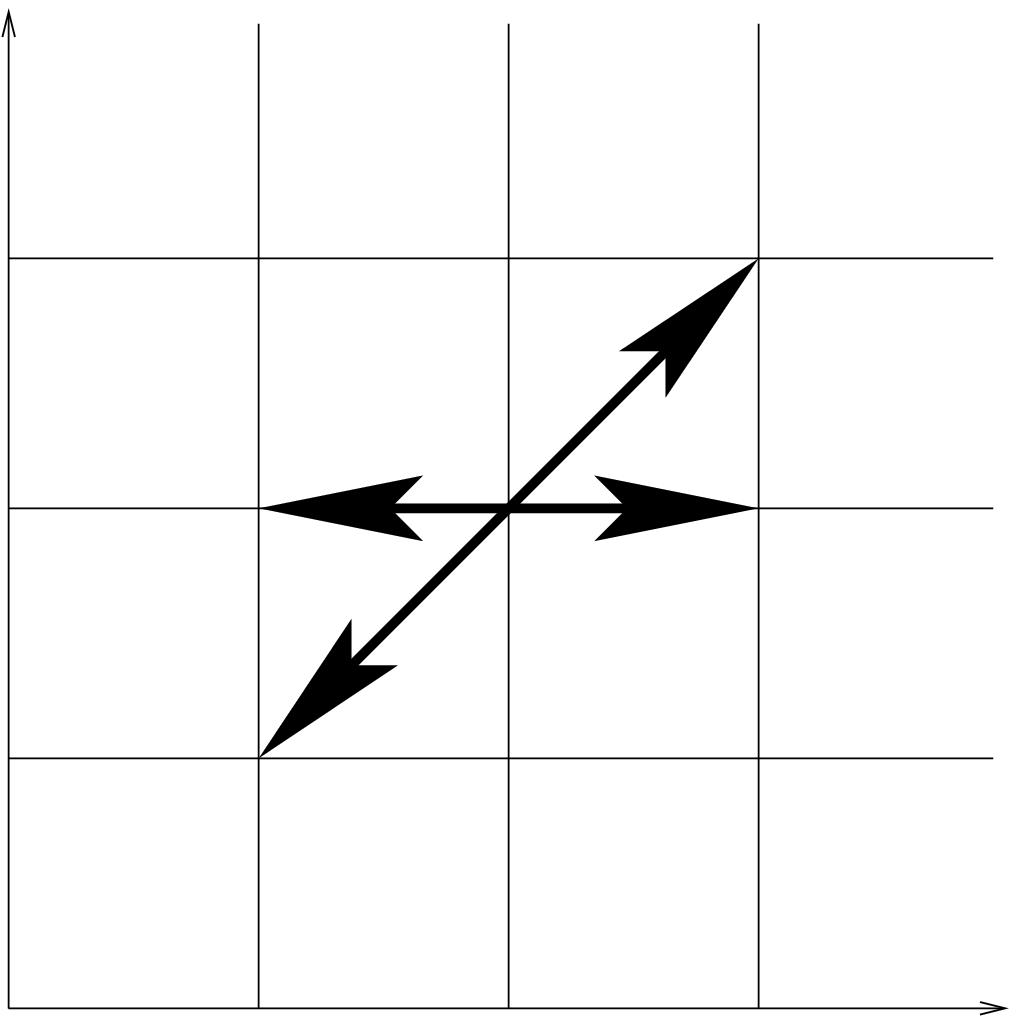}
\end{picture}
\caption{Random walks having groups of order four, six and eight
respectively} \label{SRW_Flatto_X}
\end{figure}

We will now continue the functions $h$ and $\tilde{h}$, initially
defined on the unit disc, to $\mathbb{C}\setminus [x_{3},x_{4}]$ and
$\mathbb{C}\setminus [y_{3},y_{4}]$ as holomorphic functions.
This continuation will have a
twofold interest. First, in Section~\ref{Explicit_form}, we will
have, in order to find explicit expressions of $h$ and $\tilde{h}$,
to solve a boundary value problem, with boundary condition on closed
curves that lie in the exterior of the unit disc. Secondly in
Section~\ref{Green_functions_Martin_boundary}, when we will calculate the
asymptotic of the Green functions, the quantity
$x^{n_{0}}y^{m_{0}}-h(x)-\tilde{h}(y)-h_{00}$ will naturally appear,
evaluated at some $(x,y)$ -- in fact, the saddle-point~-- that is not
in $\mathcal{D}(0,1)^{2}$.

To do this continuation, we will use a uniformization of the curve
$\{(x,y)\in \mathbb{C}^{2} : Q(x,y)=0\}$. Being a Riemann surface of
genus one, we already know that $Q=0$ is homeomorphic to some quotient
$\mathbb{C}/\Gamma$, where $\Gamma$ is a two-dimensional lattice,
that is to say to a parallelogram whose the opposed edges are identified.
In~\cite{FIM}, such a lattice $\Gamma$ and also a bijection between
$\mathbb{C}/\Gamma$ and $Q=0$ are made explicit. Indeed, the authors
find there $\omega_{1}\in \imath \mathbb{R}$ and $\omega_{2}\in
\mathbb{R}$, two functions $\phi$ and $\psi$, such that $\{(x,y)\in
\mathbb{C}^{2} : Q(x,y)=0\}=\{(x(\omega),y(\omega)), \omega\in
\mathbb{C}/\Gamma\}$, where $x(\omega)=\phi(\wp_{1,2}(\omega),\wp_{1,2}'(\omega))$,
$y(\omega)=\psi(\wp_{1,2}(\omega),\wp_{1,2}'(\omega))$,
$\Gamma=\{n_{1}\omega_{1}+n_{2}\omega_{2}, n_{1},n_{2}\in
\mathbb{Z}\}$ and $\wp_{1,2}$ is the classical Weierstrass elliptic function
associated to the periods $\omega_{1}$ and $\omega_{2}$, that are equal to~:
     \begin{equation}\label{omega123}
          \omega_{1}= \imath \int_{x_{1}}^{x_{2}}
          \frac{\text{d}x}{\sqrt{-d(x)}},\hspace{5mm}
          \omega_{2}= \int_{x_{2}}^{x_{3}}
          \frac{\text{d}x}{\sqrt{d(x)}},\hspace{5mm}
          \omega_{3}= \int_{X\left(y_{1}\right)}^{x_{1}}
          \frac{\text{d}x}{\sqrt{d(x)}},
     \end{equation}
$\omega_{3}\in ]0,\omega_{2}[$ being a period that will turn out to
be quite important in the sequel. The functions $\phi$ and $\psi$
are also made explicit~: for instance, if $x_{4}\neq \infty $, then
it is possible to take $\phi(p,p')=x_{4}+d'(x_{4})/(p-d''(x_{4})/6)$ and if
$x_{4}=\infty$, $\phi(p,p')=(6p-d''(0))/d'''(0)$.
\begin{figure}[!ht]
\begin{center}
\begin{picture}(200.00,110.00)
\includegraphics{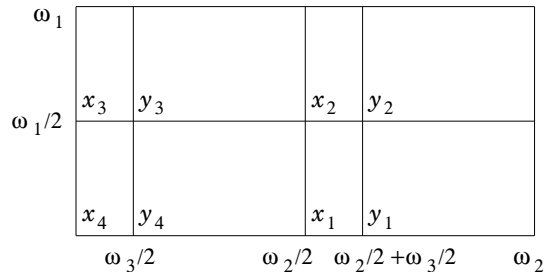}
\end{picture}
\end{center}
\caption{Location of the cuts on the covering surface}
\label{Locations_of_the_cuts}
\end{figure}
\newline Moreover, on $\mathbb{C}/\Gamma$, the automorphisms $\xi$ and $\eta$
take the following particularly nice form~:
     \begin{equation}\label{xi_eta_covering}
          \xi\left(\omega\right)=-\omega,
          \hspace{10mm}
          \eta\left(\omega\right)=-\omega+\omega_{3},
          \hspace{10mm}
          \delta\left(\omega \right)=
          \eta\left(\xi\left(\omega\right)\right)=\omega+\omega_{3}.
     \end{equation}
In particular, the group $\mathcal{H}$ has a finite order if and
only if $\omega_{3}/\omega_{2}\in \mathbb{Q}$ and in this case the
order is given by $2\inf \{ n\in \mathbb{N}^{*} : n
\omega_{3}/\omega_{2}\in \mathbb{N}\}$.

Any function $f$ of the variable $x$ (resp.\ $y$) defined  on some
domain $\mathcal{D}\subset \mathbb{C}$ can be lifted on $\{\omega
\in \mathbb{C}/\Gamma : x(\omega)\in \mathcal{D} \}$ (resp.\
$\{\omega \in \mathbb{C}/\Gamma : y(\omega)\in \mathcal{D} \}$) by
setting $F(\omega)=f(x(\omega))$ (resp.\ $F(\omega)=f(y(\omega))$).
In particular we can lift the generating functions $h$ and
$\tilde{h}$ and we set $H(\omega)=h(x(\omega))$ and
$\tilde{H}(\omega)=\tilde{h}(y(\omega))$, well defined on $\{\omega\in \mathbb{C}/\Gamma
: |x(\omega)|\leq 1\}$ and $\{\omega \in\mathbb{C}/\Gamma: |y(\omega)|\leq 1\}$
respectively. In particular, on $\{\omega \in \mathbb{C}/\Gamma: |x(\omega)|\leq 1,
|y(\omega)|\leq 1\}$, using~(\ref{functional_equation}), we have
$H(\omega)+\tilde{H}(\omega)+h_{00}-x(\omega)^{n_{0}}y(\omega)^{m_{0}}=0$.
Applying several times the Galois automorphisms $\xi$ and $\eta$ to
any point of this domain and laying down
$H(\omega)=H(\xi(\omega))$, $\tilde{H}(\omega)=
\tilde{H}(\eta(\omega))$ the authors of~\cite{FIM} prove the
following fundamental proposition.


\begin{prop}\label{continuation_h_h_tilde_covering}
The functions $H$ and $ \tilde{H}$ can be continued as holomorphic
functions on $\mathbb{C}/\Gamma$, where they satisfy~:
     \begin{eqnarray*}
          H\left(\omega\right)=
          H\left(\xi\left(\omega\right)\right),
          \hspace{5mm}
          \widetilde{H}\left(\omega\right)=
          \widetilde{H}\left(\eta\left(\omega\right)\right),
          \hspace{10mm}
          \forall \omega \in \mathbb{C}/\Gamma,\hspace{17mm} \\
          H\left(\omega\right)+\widetilde{H}\left(\omega\right)+h_{00}-
          x\left(\omega\right)^{n_{0}}
          y\left(\omega\right)^{m_{0}}=0,\hspace{10mm}
          \forall \omega \in \left[\omega_{3}/2,\omega_{2}\right]\times
          \left[0,\omega_{1}/\imath\right].
     \end{eqnarray*}
\end{prop}

\begin{cor}\label{continuation_h_h_tilde}
The function $h$ and $\tilde{h}$ can be continued into holomorphic
functions on $\mathbb{C}\setminus [x_{3},x_{4}]$ and
$\mathbb{C}\setminus [y_{3},y_{4}]$ respectively.
\end{cor}



\subsection{Absorption probability}
\label{Finite_groups_and_s_s_functions}

In the next section we will find explicitly
$h(x)$, $\tilde{h}(y)$ and $h_{00}$
that will provide of course the absorption
probability $h(1)+\tilde{h}(1)+h_{00}$.
However, this expression is usable difficultly.
In this subsection we prove that in a special case of finite
groups of the random walk (see Definition~\ref{defigr})
the probability of absorption
takes a particularly nice form, see Corollary~\ref{probability_being_absorbed}.
Furthermore, in the case of the group of any order,
Proposition~\ref{absas}
gives the precise exponential asymptotic of the absorption
probability as
$n_{0},m_{0}\to \infty$.

We first note that the quantity
$H(\omega)+\tilde{H}(\omega)+h_{00}-x(\omega)^{n_{0}}
y(\omega)^{m_{0}}$, for $\omega$ in $[0,\omega_{3}/2]\times
[0,\omega_{1}/\imath]$, can be considerably simplified in some cases,
namely when the group is finite (i.e.\ $\omega_{2}/\omega_{3}\in
\mathbb{Q}_{+}$) and when in addition $\omega_{2}/\omega_{3}\in
\mathbb{N}$. This is for example the case of the walks such that
$\Delta=0$, for which $\omega_{2}/\omega_{3}=2$ -- indeed, we have
already seen that both assertions $\Delta=0$ and $\omega_{2}/\omega_{3}=2$
are equivalent to the fact that $\mathcal{H}$ is of order four~--~; this is also the
case of the walk in the Weyl chamber of
$\mathfrak{sl}_{3}(\mathbb{C})$, see Figure~\ref{SRW_Flatto_X},
since in this case the group is of order six, hence
$\omega_{2}/\omega_{3}$ is equal to $3/2$ or $3$, and by a direct
calculation we show that $\omega_{2}/\omega_{3}=3$. On the other
hand, this is not the case of the walk whose transition
probabilities are represented on the right part of
Figure~\ref{SRW_Flatto_X}~: although the group is of order eight, we
don't have $\omega_{3}=\omega_{2}/4$ but $\omega_{3}=3\omega_{2}/4$.

\begin{prop}\label{simplification_omega_omega_prop}
Suppose that $\omega_{2}/\omega_{3}\in \mathbb{N}$~; in particular
this implies that $\mathcal{H}$ is of order
$2\omega_{2}/\omega_{3}$. Then if $\omega\in [0,\omega_{3}/2]\times
[0,\omega_{1}/\imath]$,
     \begin{equation}\label{simplification_omega_omega}
          H\left(\omega\right)+\widetilde{H}\left(\omega\right)+h_{00}-
          x\left(\omega\right)^{n_{0}}y\left(\omega\right)^{m_{0}}=
          -\sum_{w\in \mathcal{H}}\left(-1\right)^{l\left(w\right)}
          x\left(w\left(\omega\right)\right)^{n_{0}}
          y\left(w\left(\omega\right)\right)^{m_{0}},
     \end{equation}
where $l(w)$ is the length of the word $w$, that is the smallest $r$
for which we can write $w=s_{1}\cdots s_{r}$, with $s_{i}$ equal to
$\xi$ or $\eta$.
\end{prop}


\begin{proof}
The key point of the proof of
Proposition~\ref{simplification_omega_omega_prop}, that also
explains why we have done the hypothesis $\omega_{2}/\omega_{3}\in
\mathbb{N}$, is that in this only case, the fundamental domain
$\chi_{0}=[0,\omega_{2}/(2n)[\times [0,\omega_{1}/\imath[$ and the domain
$[0,\omega_{3}/2[\times [0,\omega_{1}/\imath[$ of
Proposition~\ref{continuation_h_h_tilde_covering} coincide
(by $\chi_{0}$ is a fundamental domain we mean that each $\omega \in
\mathbb{C}/\Gamma $ is conjugate under $\mathcal{H}$ to one and only
one point of $\chi_{0}$).

Let us first give a proof in the case of the groups of order four.
Note that $H+\tilde{H}+h_{00}=H(\xi)+\tilde{H}(\eta)+h_{00}$, since
$H$ (resp.\ $\tilde{H}$) is invariant w.r.t.\ $\xi$ (resp.\ $\eta$),
thanks to Proposition~\ref{continuation_h_h_tilde_covering}. So
$H+\tilde{H}+h_{00}=H(\xi)+\tilde{H}(\xi)+h_{00}
+H(\eta)+\tilde{H}(\eta)+h_{00}-(H(\eta)+\tilde{H}(\xi)+h_{00})$.
Using once again the invariance properties of $H$ and $\tilde{H}$,
we can write $H+\tilde{H}+h_{00}=H(\xi)+\tilde{H}(\xi)+h_{00}
+H(\eta)+\tilde{H}(\eta)+h_{00}-(H(\xi \eta)+\tilde{H}(\eta
\xi)+h_{00})$. Since the order of $\mathcal{H}$ is four, $\xi
\eta=\eta \xi$ and the previous equation becomes~:
$H+\tilde{H}+h_{00}=H(\xi)+\tilde{H}(\xi)+h_{00}+
H(\eta)+\tilde{H}(\eta)+h_{00}-(H(\xi \eta)+\tilde{H}(\xi
\eta)+h_{00})$. If $\omega\in \chi_{0}$, then $w(\omega)\in
(\mathbb{C}/\Gamma) \setminus \chi_{0}$ for all $w\in
\mathcal{H}\setminus \{\text{id}\}$. Indeed, we will prove in
Lemma~\ref{lemma_domain_fundamental} that $\chi_{0}$ is a
fundamental domain. In addition, thanks to
Proposition~\ref{continuation_h_h_tilde_covering}, the functional
equation $H(\omega)+\tilde{H}(\omega)+h_{00}-x(\omega)^{n_{0}}y(\omega)^{m_{0}}=0$
is verified in $[\omega_{3}/2,\omega_{2}[\times [0,\omega_{1}/\imath[$
which coincides with $(\mathbb{C}/\Gamma) \setminus \chi_{0}$. In
other words, we can replace
$H(w(\omega))+\tilde{H}(w(\omega))+h_{00}$ by
$x(w(\omega))^{n_{0}}y(w(\omega))^{m_{0}}$ for any of three elements
$w\in \mathcal{H}\setminus \{\text{id}\}$.
Proposition~\ref{simplification_omega_omega_prop} is thus proved in
the case of the groups $\mathcal{H}$ of order four.

In the general case  $\omega_{2}/\omega_{3}=n$, for $k\in \{1,\ldots
,n-1\}$, denote by $w_{1,k}$ and $w_{2,k}$ the two reduced words of
length $k$, i.e.\ the words $s_{1}\cdots s_{k}$ and $s_{2}\cdots
s_{k}s_{1}$, where for $r\geq 1$, $s_{2r}=\xi$ and $s_{2r-1}=\eta$,
and denote by $w_{n}$ the only word of length $n$. The fact that
there is only one word of length $n$ follows from the equality $\inf
\{n\in \mathbb{N}^{*} : \delta^{n}=\text{id} \}=\inf
\{n\in\mathbb{N}^{*} : s_{1}s_{2}\cdots s_{n}=s_{2}\cdots
s_{n}s_{1}\}$. Then, by induction, we prove that
     \begin{eqnarray*}
          H\left(\omega\right)+\widetilde{H}\left(\omega\right)
          &=&\sum_{k=1}^{n-1}\left(-1\right)^{k+1}
          \Big(H\left(w_{1,k}\left(\omega\right)\right)
          +\widetilde{H}\left(w_{1,k}\left(\omega\right)\right)+
          H\left(w_{2,k}\left(\omega\right)\right)
          +\widetilde{H}\left(w_{2,k}\left(\omega\right)\right)\Big)\\
          &-&\left(-1\right)^{n}
          \Big(H\left(w_{n}\left(\omega\right)\right)
          +\widetilde{H}\left(w_{n}\left(\omega\right)\right)\Big).
     \end{eqnarray*}
Since $\mathcal{H}=\{\text{id},w_{1,1},w_{2,1},\ldots
,w_{1,n-1},w_{2,n-1},w_{n}\}$ and since $[0,\omega_{3}/2[\times
[0,\omega_{1}/\imath[$ is a fundamental domain, if $\omega \in
[0,\omega_{3}/2[\times [0,\omega_{1}/\imath[= [0,\omega_{2}/(2n)[\times
[0,\omega_{1}/\imath[$ then thanks to
Proposition~\ref{continuation_h_h_tilde_covering}, for any $w\in
\mathcal{H}\setminus \{\text{id}\}$,
$H(w(\omega))+\tilde{H}(w(\omega))+h_{00}=x(w(\omega))^{n_{0}}y(w(\omega))^{m_{0}}$.
Moreover, $l(w_{n})=n$ and for $k\in \{1,\ldots ,n-1\}$ and $i\in
\{1,2\}$, $l(w_{i,k})={k}$, so~(\ref{simplification_omega_omega}) is proved.
\end{proof}

\begin{lem}\label{lemma_domain_fundamental}
Suppose that the group $\mathcal{H}$ is finite of order $2n$.
Then for any $k\in\{0,\ldots ,2n-1\}$, the domain
$\chi_{k}=[k\omega_{2}/(2n),(k+1)\omega_{2}/(2n)[\times
[0,\omega_{1}/\imath[$ is a fundamental domain, i.e.\ each $\omega \in
[0,\omega_{2}[\times [0,\omega_{1}/\imath[ $ is conjugate under
$\mathcal{H}$ to one and only one point of $\chi_{k}$.
\end{lem}

\begin{proof}
Denote by $\Lambda_{\mu}=\mu +[0,\omega_{1}]$ the vertical segment
with abscissa $\mu$. Then, with~(\ref{xi_eta_covering}), we can
describe the actions of $\xi$ and $\eta$ on these segments. So, for
any $\mu$ in $[0,\omega_{2}]$,
$\xi(\Lambda_{\mu})=\Lambda_{\omega_{2}-\mu}$. Also, if $\mu \in
[0,\omega_{3}]$, then $\eta(\Lambda_{\mu})=\Lambda_{\omega_{3}-\mu}$
and if $\mu \in ]\omega_{3},\omega_{2}]$ then
$\eta(\Lambda_{\mu})=\Lambda_{\omega_{3}+\omega_{2}-\mu}$. Of
course, we also know the action of the elements of $\mathcal{H}$ on
the domains $\chi_{k}=[k\omega_{2}/(2n),(k+1)\omega_{2}/(2n)[\times
[0,\omega_{1}/\imath[$, since we know how these automorphisms act on the
boundaries of these sets.

Suppose first that $k$ is even. Then the $\cup_{p=0}^{n-1}
\delta^{p}(\chi_{k})=\cup_{q=0}^{n-1} \chi_{2q}$. In particular,
there exists $m\in \{0,\ldots ,n-1\}$ such that
$\delta^{m}(\chi_{k})=\chi_{0}$. Thanks to~(\ref{xi_eta_covering}),
we have $\xi(\chi_{0})=\chi_{2n-1}$, so
$\xi(\delta^{m}(\chi_{k}))=\chi_{2n-1}$. Also, $\cup_{p=0}^{n-1}
\delta^{p}(\chi_{2n-1})=\cup_{q=0}^{n-1} \chi_{2q+1}$. Of course,
$(\cup_{q=0}^{n-1} \chi_{2q+1})\cup (\cup_{q=0}^{n-1}
\chi_{2q})=[0,\omega_{2}[\times [0,\omega_{1}/\imath[$ and
$\mathcal{H}=\{\text{id},\delta,\ldots ,
\delta^{n-1},\xi\delta^{m},\delta\xi\delta^{m},\ldots,
\delta^{n-1}\xi\delta^{m}\}$, for any $m\in \{0,\ldots ,n-1\}$. For
example, in Figure~\ref{domain_fundamental} are represented the
domain $\chi_{0}$ and its images under $\mathcal{H}$ in the
particular case $\omega_{3}=\omega_{2}/n$.
\begin{figure}[!ht]
\begin{center}
\begin{picture}(310.00,90.00)
\includegraphics{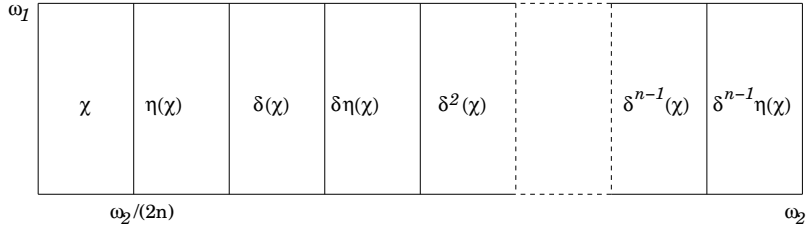}
\end{picture}
\end{center}
\caption{$\chi_{0} =[0,\omega_{2}/(2n)[\times [0,\omega_{1}/\imath[$ is a
fundamental domain} \label{domain_fundamental}
\end{figure}
Lemma~\ref{lemma_domain_fundamental} is thus proved if $k$ is even.
The proof is quite similar in case of odd $k$, so we omit it.
\end{proof}

\begin{cor}\label{probability_being_absorbed}
Suppose that $\omega_{2}/\omega_{3}\in \mathbb{N}$. Then the
probability of being absorbed is equal to~:
     \begin{equation*}
          \mathbb{P}_{n_{0},m_{0}}\left((X,Y)
          \ \textnormal{is absorbed}\right)=
          h\left(1\right)+\widetilde{h}\left(1\right)+h_{00}=
          1-\sum_{w\in \mathcal{H}}\left(-1\right)^{l\left(w\right)}
          f_{n_{0},m_{0}}\left(w\left(1,1\right)\right),
     \end{equation*}
where $f_{n_{0},m_{0}}(x,y)=x^{n_{0}}y^{m_{0}}$ and the
automorphisms of $\mathcal{H}$ are here defined by
using~(\ref{def_xi_eta_Q}).
\end{cor}
\begin{proof}
The proof is simply based on the fact that the point lying over
$(1,1)$ belongs to $[0,\omega_{3}/2]\times [0,\omega_{1}/\imath]$,
so that Corollary~\ref{probability_being_absorbed} is an immediate consequence of
Proposition~\ref{simplification_omega_omega_prop}.
\end{proof}

We can therefore easily calculate the probability of being absorbed
for the walks verifying $\Delta=0$, since in this case the group is
of order four and $\omega_{2}/\omega_{3}=2$. In particular, we find
again Proposition~28 of~\cite{Ras}.
Corollary~\ref{probability_being_absorbed} applies also to the walk
in the Weyl chamber of $\mathfrak{sl}_{3}(\mathbb{C})$, see the walk
whose transition probabilities are  drawn in the middle of
Figure~\ref{SRW_Flatto_X}, since in this case
$\omega_{2}/\omega_{3}=3$.

In the general case, the probability of being absorbed,
$h_{00}+h(1)+\tilde{h}(1)$, verifies the following inequality.
\begin{prop}\label{absas}
The probability of being absorbed can be 
bounded from above and below as follows~:
     \begin{equation}\label{probability_absorbed_general_case}
          A/2\leq h_{00}+h(1)+\widetilde{h}(1) \leq  A,\hspace{5mm}
          A=\left(\frac{p_{1-1}+p_{0-1}+p_{-1-1}}{p_{11}+p_{01}+p_{-11}} \right)^{n_{0}}
           +\left(\frac{p_{-1-1}+p_{-10}+p_{-11}}{p_{11}+p_{10}+p_{1-1}} \right)^{m_{0}}.
     \end{equation}
\end{prop}

\begin{proof}
We begin by writing the following equality~:
     \begin{eqnarray*}
          2\left(h_{00}+H\left(\omega\right)+\widetilde{H}\left(\omega\right)\right)&=&
          h_{00}+H\left(\xi\left(\omega\right)\right)+
          \widetilde{H}\left(\xi\left(\omega\right)\right)+
          h_{00}+H\left(\eta\left(\omega\right)\right)+
          \widetilde{H}\left(\eta\left(\omega\right)\right)\\&+&
          H\left(\omega\right)-H\left(\eta\left(\omega\right)\right)+
          \widetilde{H}\left(\omega\right)-
          \widetilde{H}\left(\xi\left(\omega\right)\right),
     \end{eqnarray*}
obtained by using the invariance properties of $H$ and $\tilde{H}$
claimed in  Proposition~\ref{continuation_h_h_tilde_covering}. In
particular, if $\omega\in[0,\omega_{3}/2]\times [0,\omega_{1}/\imath]$,
then $\xi(\omega)$ and $\eta(\omega)$ belong to
$\omega\in[\omega_{3}/2,\omega_{2}]\times [0,\omega_{1}/\imath]$, so that
using once again Proposition~\ref{continuation_h_h_tilde_covering},
we obtain that $2(H(\omega)+\tilde{H}(\omega)+h_{00})$ is equal to~:
     \begin{equation}\label{equality_before_absorption}
          x\left(\xi\left(\omega\right)\right)^{n_{0}}y\left(\xi\left(\omega\right)\right)^{m_{0}}+
          x\left(\eta\left(\omega\right)\right)^{n_{0}}y\left(\eta\left(\omega\right)\right)^{m_{0}}+
          H\left(\omega\right)-H\left(\eta\left(\omega\right)\right)+
          \widetilde{H}\left(\omega\right)-\widetilde{H}\left(\xi\left(\omega\right)\right).
     \end{equation}
In particular, if we take $\omega$ lying over $(1,1)$, that belongs
to $[0,\omega_{3}/2]\times [0,\omega_{1}/\imath]$, as said in the proof
of Corollary~\ref{probability_being_absorbed}, and if we use that
for this $\omega$, $x(\eta(\omega))=c(1)/a(1)$,
$y(\xi(\omega))=\tilde{c}(1)/\tilde{a}(1)$, we obtain~:
     \begin{equation*}
          2\left(h_{00}+h\left(1\right)+\widetilde{h}\left(1\right)\right)=
          A+h\left(1\right)-h\left(\widetilde{c}\left(1\right)/\widetilde{a}\left(1\right)\right)+
          \widetilde{h}\left(1\right)-\widetilde{h}\left(c\left(1\right)/a\left(1\right)\right).
     \end{equation*}
Then, using that $c(1)/a(1)>0$ and $\tilde{c}(1)/\tilde{a}(1)>0$
(what implies that $h(\tilde{c}(1)/\tilde{a}(1))>0$ and
$\tilde{h}(c(1)/a(1))>0$) allows to get the lower bound, and using
that $c(1)/a(1)<1$ and $\tilde{c}(1)/\tilde{a}(1)<1$ (what is
equivalent to the positivity of the two drifts, in accordance with
our assumption (M)), added to the fact that $h$ and $\tilde{h}$ are
increasing, allows to get the upper bound.
\end{proof}

\section{Explicit form of the absorption probabilities generating functions }
\label{Explicit_form}

\subsection{Riemann boundary value problem with shift}
\label{Riemann_boundary_value_problem_with_shift}
Using the notations of Subsection~\ref{The_algebraic_curve_Q}, we define the
two following curves~:
     \begin{equation}\label{def_curves_L_M}
          \mathcal{L} = Y_{0}\left(\left[\overrightarrow{\underleftarrow{x_{1}
          ,x_{2}}}\right]\right),\hspace{5mm} \mathcal{M} =
          X_{0}\left(\left[\overrightarrow{\underleftarrow{y_{1}
          ,y_{2}}}\right]\right).
     \end{equation}
Just above, we use the notation of~\cite{FIM}~:
$\Big[\overrightarrow{\underleftarrow{u,v}}\Big]$ stands for the contour $[u,v]$
traversed from $u$ to $v$ along the upper edge of the slit $[u,v]$
and then back to $u$ along the lower edge of the slit.

In~\cite{FIM} is proved that the curves $\mathcal{L}$ and
$\mathcal{M}$ are quartics, symmetrical w.r.t.\ the real axis,
closed and simple, included in $\mathbb{C}\setminus
[y_{1},y_{2}]\cup [y_{3},y_{4}]$ and $\mathbb{C}\setminus
[x_{1},x_{2}]\cup [x_{3},x_{4}]$ respectively.

The reason why we have introduced these curves appears now~: the
functions $h$ and $\tilde{h}$, defined
in~(\ref{def_generating_functions}), verify the following boundary
conditions on $\mathcal{M}$ and $\mathcal{L}$~:
     \begin{eqnarray}
          \forall t \in \mathcal{M} :\hspace{5mm}
          h\left( t \right) - h\left( \overline{t} \right)&=&
          t^{n_{0}}Y_{0}\left(t\right)^{m_{0}}-\overline{t}^{n_{0}}
          Y_{0}\left(\overline{t}\right)^{m_{0}},\label{SR_problem_h}\\
          \forall t \in \mathcal{L}:\hspace{5mm}
          \widetilde{h}\left( t \right) - \widetilde{h}\left( \overline{t} \right)&=&
          X_{0}\left(t\right)^{n_{0}}t^{m_{0}}-
          X_{0}\left(\overline{t}\right)^{n_{0}}\overline{t}^{m_{0}}.\nonumber
     \end{eqnarray}
The way to obtain these boundary conditions is described
in~\cite{FIM}, so we refer to this book for the details. The
function $h$, as a generating function of probabilities, is well
defined on the closed unit disc and with
Corollary~\ref{continuation_h_h_tilde} is continuable into a
holomorphic function on $\mathbb{C}\setminus [x_{3},x_{4}]$, domain
that contains the bounded domain delimited by $\mathcal{M}$. Now we
have the problem {\it to find $h$ holomorphic inside $\mathcal{M}$,
continuous up to the boundary $\mathcal{M}$ and verifying the boundary
condition~(\ref{SR_problem_h}). In addition $h(0)=0$.}

Problems with boundary conditions like~(\ref{SR_problem_h}) are
called Riemann boundary value problems with shift. The classical way
to study this kind of problems is to reduce them to Riemann-Hilbert
problems, for which there exists a suitable and complete theory. The
conversion between Riemann problems with shift and Riemann-Hilbert
problems is done thanks to the use of conformal gluing functions,
notion defined just above. For details about boundary value
problems, we refer to~\cite{LU}.

\begin{defn}\label{def_CGF}
     Let $\mathcal{C}$ be a simple closed curve, symmetrical
     w.r.t.\ the real axis.
     Denote by $\mathcal{G}_{\mathcal{C}}$
     the interior of the bounded domain delimited by $\mathcal{C}$.
     $w$ is called a conformal gluing function (CGF) for
     the curve $\mathcal{C}$ if {\rm (i)} $w$ is meromorphic in
     $\mathcal{G}_{\mathcal{C}}$, continuous up
     to its boundary {\rm (ii)} $w$
     establishes a conformal mapping of
     $\mathcal{G}_{\mathcal{C}}$ onto the complex plane
     cut along a smooth arc $U$ {\rm (iii)}
     for all $t\in \mathcal{C}$,~$w(t)=w(\overline{t})$.
     \end{defn}

For the walks such that $p_{10}+p_{-10}+p_{01}+p_{0-1}=1$, that we
have studied in~\cite{Ras}, we easily see that $\mathcal{L}$ and
$\mathcal{M}$ are simply the circles
$\mathcal{C}(0,(p_{0-1}/p_{01})^{1/2})$ and
$\mathcal{C}(0,(p_{-10}/p_{10})^{1/2})$ and the functions
$p_{01}t+p_{0-1}/t$ and $p_{10}t+p_{-10}/t$ are proper CGF.


In the general case, it is very pleasant to notice that we still
have the existence and even the explicit expression of suitable CGF
for the curves $\mathcal{L}$ and $\mathcal{M}$. The following result
is  due to~\cite{FIM}. Define for $t\in \mathbb{C}$
     \begin{equation}\label{def_CGF_general_case}
          w(t)=\wp_{1,3}\left( -\frac{\omega_{1}+\omega_{2}}{2}+
          x^{-1}\left(t\right)\right),
     \end{equation}
where the $\omega_{i}$, $i=1,2,3$ are defined in~(\ref{omega123}),
$\wp_{1,3}$ is the classical Weierstrass function associated to the
periods $\omega_{1}$ and $\omega_{3}$, $x^{-1}$ is the reciprocal
function of the uniformization built in
Subsection~\ref{Galois_automorphisms}, we recall that it was
$x(\omega)= x_{4}+d'(x_{4})/(\wp_{1,2}(\omega)-d''(x_{4})/6)$ if
$x_{4}\neq \infty$ and
$x(\omega)=(6\wp_{1,2}(\omega)-d''(0))/d'''(0)$ if $x_{4}=\infty$,
$\wp_{1,2}$ being the Weierstrass function with periods
$\omega_{1}$ and $\omega_{2}$.

Then $w$ is single-valued and meromorphic on
$\mathcal{G}_{\mathcal{M}}$, continuous up to its boundary and
establishes a conformal mapping of the domain
$\mathcal{G}_{\mathcal{M}}$ onto the domain $\mathbb{C}\setminus U$,
where $U=[w(X(y_{1})),w(X(y_{2}))]$. Moreover, on
$\mathcal{G}_{\mathcal{M}}$, $w$ has one pole of order one, it is at $x_{2}$.

\begin{prop}\label{w}
    Let $\mathcal{M}$ be the curve defined in~(\ref{def_curves_L_M}). $w$,
    defined in~(\ref{def_CGF_general_case}),
    is a CGF for $\mathcal{M}$.
\end{prop}
Proposition~\ref{w} and the different properties mentioned above it
are proved  in~\cite{FIM}. Then, following Subsection~5.4 of this
book --~though making use of the index lightly different~--, we
obtain the following integral representation of the function $h$.

\begin{prop}\label{explicit_h(x)_FIM}
     Let $\mathcal{M}$ be the curve defined in~(\ref{def_curves_L_M})
     and $\mathcal{M}_{u}=\mathcal{M}\cap \{t\in \mathbb{C} : \textnormal{Im}(t)\geq 0\}$.
     The function $h$ admits in $\mathcal{G}_{\mathcal{M}}$
     the following integral representation, the function $w$ being
     defined in~(\ref{def_CGF_general_case})~:
          \begin{equation*}
               h\left(x\right) = \frac{1}{2\pi \imath}
               \int_{\mathcal{M}_{u}}
               \left(t^{n_{0}}Y_{0}\left(t\right)^{m_{0}}-
               \overline{t}^{n_{0}}Y_{0}\left(\overline{t}\right)^{m_{0}}\right)
               \left(\frac{w'\left(t\right)}
               {w\left(t\right)-w\left(x\right)}-\frac{w'\left(t\right)}
               {w\left(t\right)-w\left(0\right)}\right)\textnormal{d}t.
          \end{equation*}
     \end{prop}

\subsection{Integral representation of the generating functions}
\label{Riemann_Hilbert_problem}

In this subsection, we simplify the integral
representation of $h$ obtained in Proposition~\ref{explicit_h(x)_FIM}
for several reasons~: 
indeed, this formulation does not
highlight the singularities of $h$ and hardly allows to obtain the
asymptotic of the absorption probabilities, furthermore it makes
appear $h$ asymmetrically as an integral on $\mathcal{M}_{u}$.
Before stating, in
Theorem~\ref{explicit_h(x)_second}, the final result, we give the
definition~:
     \begin{equation}\label{def_mu}
          \mu_{m_{0}}\left(t\right) = \frac{1}{\left(2 a\left(t\right)\right)^{m_{0}}}
          \sum_{k=0}^{\left\lfloor \left(m_{0}-1\right)/2
          \right\rfloor}{\genfrac(){0cm}{0}{m_{0}}{2k+1}  }
          d\left(t\right)^{k} \left( - b\left(t\right) \right)^{m_{0}-\left(2k+1\right)}.
     \end{equation}
The function $\mu_{m_{0}}$ is such that for all $t\downarrow [x_{1},x_{2}]$
(resp.\ $t\uparrow [x_{1},x_{2}]$),$Y_{0}(t)^{m_{0}}-
\overline{Y_{0}(t)}{}^{m_{0}}$ is equal to
$- 2\imath (-d(t))^{1/2} \mu_{m_{0}}(t)$ (resp.\ $2 \imath(-d(t))^{1/2} \mu_{m_{0}}(t)$).

     \begin{thm}\label{explicit_h(x)_second}
     The function $h$ admits, on $\mathcal{G}_{\mathcal{M}}$,
     the following integral representation~:
          \begin{equation}\label{new_integral_form_h}
               h\left(x\right) = x^{n_{0}}Y_{0}\left(x\right)^{m_{0}}
               +\frac{1}{\pi }\int_{x_{1}}^{x_{2}}
               t^{n_{0}}\mu_{m_{0}}\left(t\right)
               \left(\frac{w'\left(t\right)}{w\left(t\right)-w\left(x\right)}-
               \frac{w'\left(t\right)}{w\left(t\right)-w\left(0\right)}\right)
               \sqrt{-d\left(t\right)}\textnormal{d}t.
          \end{equation}
     where the function $w$ is defined in~(\ref{def_CGF_general_case}),
     $\mu_{m_{0}}$ in~(\ref{def_mu}).
     \end{thm}

The function $h$ appears in~(\ref{new_integral_form_h})
at the first sight as the sum of a
function holomorphic on $\mathbb{C}\setminus [x_{1},x_{2}]\cup
[x_{3},x_{4}]$ and an other function holomorphic on
$\mathbb{C}\setminus \{w^{-1}(w([x_{1},x_{2}]))\cup
[x_{3},x_{4}]\}$, set that is included in $\mathbb{C}\setminus
[x_{1},x_{2}]\cup [x_{3},x_{4}]$.
We will later split this representation into two 
terms~(\ref{residue_term}) and (\ref{int_phi}), where
the first one is holomorphic in  $[x_1,x_2]$ by
Lemma~\ref{lemma_residue_term} and the second one 
is also holomorphic in $[x_{1},x_{2}]$ by~(\ref{equality_l1})
and~(\ref{before_def_f_R}),
so that this  representation gives in fact that $h$ is holomorphic in the
neighborhood of $[x_{1},x_{2}]$, what we already knew, since
Lemma~\ref{lemma_branched_points} implies that $x_{1}$ and $x_{2}$
lye in the unit disc.

\begin{proof}
We start by expressing the integral obtained in
Proposition~\ref{explicit_h(x)_FIM} as an integral on a closed
contour, namely $\mathcal{M}$. Making the change of variable
$t\mapsto \overline{t}$ and using that on $\mathcal{M}$,
$w(t)=w(\overline{t})$, we obtain
     \begin{equation}\label{explicit_h(x)_secondly}
               h\left(x\right) = \frac{1}{2\pi \imath}
               \int_{\mathcal{M}}
               t^{n_{0}}Y_{0}\left(t\right)^{m_{0}}
               \left(\frac{w'\left(t\right)}
               {w\left(t\right)-w\left(x\right)}-\frac{w'\left(t\right)}
               {w\left(t\right)-w\left(0\right)}\right) \textnormal{d}t.
     \end{equation}

Then we transform~(\ref{explicit_h(x)_secondly}) into an integral on
the cut $[x_{1},x_{2}]$. To do this, start by remarking that the
function of two variables $(t,x)\mapsto
w'(t)/(w(t)-w(x))-(x_{2}-x)/((x_{2}-t)(t-x))$ is continuable into a
holomorphic function in ${\mathcal{G}_{\mathcal{M}}}^{2}$. This
property comes from the fact that $w$ is one to one in
$\mathcal{G}_{\mathcal{M}}$ and has a pole of order one at $x_{2}$.
In particular, the function $\phi$, initially well defined on
${\mathcal{G}_{\mathcal{M}}}^{2}\setminus \{(y,y) : y\in
\mathcal{G}_{\mathcal{M}}\}$ by
     \begin{equation}\label{def_phi}
          \phi\left(t,x\right)= \frac{w'\left(t\right)}
          {w\left(t\right)-w\left(x\right)}-
          \frac{w'\left(t\right)}{w\left(t\right)-w\left(0\right)}-
          \frac{x}{t\left(t-x\right)},
     \end{equation}
is continuable into a holomorphic function in
${\mathcal{G}_{\mathcal{M}}}^{2}$, again denoted by $\phi$.

Consider now the contour $\mathcal{H}_{\epsilon }=
\mathcal{M}_{\epsilon} \cup \mathcal{S}^{1}_{\epsilon } \cup
\mathcal{S}^{2}_{\epsilon } \cup \mathcal{C}^{1}_{\epsilon } \cup
\mathcal{C}^{2}_{\epsilon } \cup \mathcal{D}^{1}_{\epsilon } \cup
\mathcal{D}^{2}_{\epsilon }$, represented in Figure~\ref{pacman}.
\begin{figure}[!ht]
\begin{center}
\begin{picture}(250.00,120.00)
\includegraphics{contour_integration.eps}
\end{picture}
\end{center}
\caption{Contour of integration} 
\label{pacman}
\end{figure}
\newline A consequence of the holomorphy of $\phi$ in ${\mathcal{G}_{\mathcal{M}}}^{2}$
is that for all $x\in \mathcal{G}_{\mathcal{M}}$~:
     \begin{equation*}
          \int_{\mathcal{H}_{\epsilon}} t^{n_{0}}Y_{0}
          \left(t\right)^{m_{0}}\phi\left(t,x\right)\text{d}t=0.
     \end{equation*}
In particular, letting $\epsilon \to 0$ and using the definition of
$\mu_{m_{0}}$ given in~(\ref{def_mu}), we obtain that~:
     \begin{equation}\label{several_item_1}
          \frac{1}{2\pi \imath}\int_{\mathcal{M}}t^{n_{0}}Y_{0}
          \left(t\right)^{m_{0}}\phi\left(t,x\right)\text{d}t=
          \frac{1}{\pi }\int_{x_{1}}^{x_{2}}t^{n_{0}}\mu_{m_{0}}
          \left(t\right)\phi\left(t,x\right)\text{d}t.
     \end{equation}
Furthermore, the residue theorem implies that for all $x$ in the
bounded domain delimited by $\mathcal{H}_{\epsilon}$~:
     \begin{equation*}
          \frac{1}{2\pi \imath}\int_{\mathcal{H}_{\epsilon}}\frac{t^{n_{0}-1}Y_{0}
          \left(t\right)^{m_{0}}}{t-x}\text{d}t=x^{n_{0}-1}Y_{0}\left(x\right)^{m_{0}}.
     \end{equation*}
So, letting $\epsilon \to 0$ and using the definition of
$\mu_{m_{0}}$ yield~:
     \begin{equation}\label{several_item_2}
          \frac{1}{2\pi \imath}\int_{\mathcal{M}}\frac{t^{n_{0}-1}Y_{0}
          \left(t\right)^{m_{0}}}{t-x}\text{d}t=x^{n_{0}-1}Y_{0}\left(x\right)^{m_{0}}+
          \frac{1}{\pi }\int_{x_{1}}^{x_{2}}\frac{t^{n_{0}-1}\mu_{m_{0}}
          \left(t\right)}{t-x}\text{d}t.
     \end{equation}
Note that to obtain~(\ref{several_item_1})
and~(\ref{several_item_2}) we have used that the integral on
$\mathcal{S}_{\epsilon}^{1}\cup \mathcal{S}_{\epsilon}^{2}$ of a
function holomorphic in the neighborhood of
$\mathcal{S}_{\epsilon}^{1}\cup \mathcal{S}_{\epsilon}^{2}$ goes to
zero with $\epsilon$, since the two contours
$\mathcal{S}_{\epsilon}^{1}$ and $\mathcal{S}_{\epsilon}^{2}$ get
closer of the same contour but covered in the two opposite
directions. For $\mathcal{C}_{\epsilon}^{1}$ and
$\mathcal{C}_{\epsilon}^{2}$, we have used that the integral of a
function integrable goes to zero as the length of the contour goes
to zero. Finally, Theorem~\ref{explicit_h(x)_second} follows
from~(\ref{def_phi}),~(\ref{several_item_1})
and~(\ref{several_item_2}).
\end{proof}

In the particular case $\Delta=0$ and $x_{4}>0$, we can quite
simplify the integral representation found in
Theorem~\ref{explicit_h(x)_second}. We have the following
result, already mentioned but not proved in~\cite{Ras}~:

\begin{prop}\label{explicit_Delta_zero}
Suppose that $\Delta=0$ and that $x_{4}>0$. Then we have the equality~:
     \begin{equation}\label{explicit_h(x)_Delta=zero}
          h\left(x\right) = \frac{x}{\pi}\int_{x_{3}}^{x_{4}}\left(t^{n_{0}}-
          \sigma\left(t\right)^{n_{0}}\right)
          \frac{\mu_{m_{0}}\left(t\right)}{t\left(t-x\right)}
          \sqrt{-d\left(t\right)}\textnormal{d}t+
          xP_{\infty}\left(x\mapsto
          x^{n_{0}-1}Y_{0}\left(x\right)^{m_{0}}\right)\left(x\right),
     \end{equation}
where if $f$ is a function meromorphic at
infinity, $P_{\infty}(f)$ denotes the principal part at infinity of $f$,
that is to say the polynomial part 
of the Laurent expansion at infinity of $f$,
$\sigma(t)=(l_{1}+l_{2})/2+
((l_{2}-l_{1})/2)^{2}/(t-(l_{1}+l_{2})/2)$, $l_{1}<0$ and $l_{2}>0$
being the roots of the second degree polynomial $b(x) a'(x)-b'(x)a(x)$.
In particular, the probabilities of absorption are in
this case easily  made explicit, by expanding
in~(\ref{explicit_h(x)_Delta=zero}) $1/(t-x)$ according to the
powers of $x$.
\end{prop}

The proof of Proposition~\ref{explicit_Delta_zero} is postponed to
the end of Section~\ref{Study_solution}, since necessary notations
introduced there and some facts shown in
the proofs of Lemmas~\ref{lemma_residue_term} and~\ref{lemma_int_phi}
will shorten it.

Of course, by a similar analysis, we obtain
integral representations of the function $\tilde{h}$.
Also, to get the quantity $h_{00}$ -- as yet unknown --, we evaluate
Equation~(\ref{functional_equation}) at some $(x,y)$ in $\{(x,y) \in \mathbb{C}^{2} :
|x|<1, |y|<1, Q(x,y)=0\}$, for example 
$(1-\epsilon,Y_{0}(1-\epsilon))$, where $\epsilon$ is
small enough, see Lemma~\ref{properties_X_Y}.

\section{Asymptotic of the absorption probabilities}
\label{Study_solution}

In this part, we will study the asymptotic of the absorption
probabilities, defined in~(\ref{absorption_probabilities}). For that
purpose we need to study the properties of the function $w$ defined
in~(\ref{def_CGF_general_case}), notably in relation with the
finiteness of the group $\mathcal{H}$~: we will have  to distinguish
the cases for which $\omega_{2}/\omega_{3}\in \mathbb{N}$ from the
other cases.

\begin{prop}\label{properties_w}
The function $w$, defined on $\mathcal{G}_{\mathcal{M}}$
by~(\ref{def_CGF_general_case}), can be continued~; this
continuation is meromorphic on $\mathbb{C}$ if
$\omega_{2}/\omega_{3}\in \mathbb{N}$, meromorphic on
$\mathbb{C}\setminus [x_{3},x_{4}]$ and algebraic in the
neighborhood of $[x_{3},x_{4}]$ if $\omega_{2}/\omega_{3}\notin
\mathbb{N}$. In every case, $w$ has a simple pole at $x_{2}$ and
$\lfloor \omega_{2}/(2\omega_{3})\rfloor$ double poles at points
lying in $]x_{2},x_{3}[\cap (\mathbb{C}\setminus
\mathcal{G}_{\mathcal{M}})$. The behavior of $w$ at $x_{3}$ depends
strongly on the group $\mathcal{H}$~:
     \begin{enumerate}
          \item Suppose first that $\omega_{2}/\omega_{3}\in 2\mathbb{N}$. Then
                $w$ has a simple pole at $x_{3}$.
          \item Suppose now that $\omega_{2}/\omega_{3}\in 2\mathbb{N}+1$. Then
                $w$ is holomorphic at $x_{3}$.
          \item Suppose at last that $\omega_{2}/\omega_{3}\notin \mathbb{N}$. Then
                $w$ has an algebraic singularity at $x_{3}$, and more precisely
                in the neighborhood of $x_{3}$, $w$ can be written as
                $w_{1}(t)+w_{2}(t)(x_{3}-t)^{1/2}$, where $w_{1}$ and $w_{2}$
                are holomorphic functions in the
                neighborhood of $x_{3}$ and $w_{2}(x_{3})\neq 0$.
     \end{enumerate}
\end{prop}

\begin{proof}
The explicit formula~(\ref{def_CGF_general_case}) of $w$ shows that
we have to study the reciprocal function of the uniformization
$x(\omega)$. With Subsection~\ref{Galois_automorphisms} we get
$x(\omega)=t$ if and only if $\wp_{1,2}(\omega)=f(t)$, where if
$x_{4}\neq \infty$ then $f(t)=d''(x_{4})/6+d'(x_{4})/(t-x_{4})$ and
if $x_{4}=\infty$ then $f(t)=(d''(0)+d'''(0)t)/6$. Moreover, by
construction, in both cases, $f(X(y_{1})) =
\wp_{1,2}((\omega_{2}+\omega_{3})/2)$, $f(X(y_{2})) = \wp_{1,2}
((\omega_{1}+\omega_{2}+\omega_{3})/2)$, $f(x_{1}) = \wp_{1,2}
(\omega_{2}/2)$, $f(x_{2}) = \wp_{1,2} ( (\omega_{1}+\omega_{2})/2
)$ and $f(\mathcal{M}) =\wp_{1,2}([(\omega_{2}+\omega_{3})/2,
(\omega_{2}+\omega_{3})/2+\omega_{1} ])$. In particular, $f$ is an
automorphism of $\mathbb{C}$ and maps $\mathcal{G}_{\mathcal{M}}$
onto $]\omega_{2}/2,\omega_{2}/2+\omega_{3}/2[\times
]0,\omega_{1}/\imath[$, see Figure~\ref{Locations_of_the_cuts}.

Recall that on a fundamental parallelogram $[0,\omega_{2}[\times
[0,\omega_{1}/\imath[$, $\wp_{1,2}$ takes each value twice. Also, in
$[0,\omega_{2}/2[\times [0,\omega_{1}/\imath [$ or in
$[\omega_{2}/2,\omega_{2}[\times [0,\omega_{1}/\imath [$ $\wp_{1,2}$ is one
to one. For these reasons and since $\omega_3<\omega_2$, we obtain
the existence of a function $t\mapsto \omega(t)$ defined on
$\mathbb{C}$, two-valued for $t\in \mathbb{R}\setminus
[x_{2},x_{3}]$, one-valued everywhere else, that verifies, for all
$t\in \mathbb{C}$, $\wp_{1,2}( \omega(t) )=f(t)$~; moreover
$\omega(\mathbb{C})=[\omega_{2}/2,\omega_{2}]\times
[0,\omega_{1}/\imath]$.

We show now that though $\omega$ is two-valued on
$\mathbb{R}\setminus [x_{2},x_{3}]$, $w$ is single-valued on
$\mathbb{C}\setminus [x_{3},x_{4}]$. We do this by studying
precisely the equation $\wp_{1,2}( \omega )=f(t)$ at points $t$
where it has more than one solution~:
\begin{itemize}
\item For $t\in [x_{4},x_{1}]$, the two values of $\omega(t)$,
      say $\omega_{1}(t)$ and $\omega_{2}(t)$, verify~:
      $\omega_{1}(t)\in [\omega_{2}/2,\omega_{2}]$,
      $\omega_{2}(t)\in [\omega_{1}+\omega_{2}/2,\omega_{1}+\omega_{2}]$
      and $\omega_{2}(t)-\omega_{1}(t)=\omega_{1}$.
\item For $t\in [x_{1},x_{2}]$, the two values of $\omega(t)$ are such that~:
      $\omega_{1}(t),\omega_{2}(t)\in [\omega_{2}/2,\omega_{2}/2+\omega_{1}]$
      and $\omega_{1}(t)-(\omega_{1}/2+\omega_{2}/2)=
      (\omega_{1}/2+\omega_{2}/2)-\omega_{2}(t)$.
\item For $t\in [x_{3},x_{4}]$, the two values of $\omega(t)$ verify~:
      $\omega_{1}(t),\omega_{2}(t)\in [\omega_{2},\omega_{2}+\omega_{1}]$
      and $\omega_{1}(t)-\omega_{2}/2=\omega_{2}/2-\omega_{2}(t)$.
\end{itemize}
Using respectively the facts that $\wp_{1,2}$ is even and $\omega_{1}$ periodic,
$w$ is in fact single-valued on $[x_{1},x_{2}]$ and $[x_{4},x_{1}]$
respectively, hence on $\mathbb{C}\setminus [x_{3},x_{4}]$.

To show that $w$ has a simple pole at $x_{2}$, we can for instance
use an explicit expression of $\omega(t)$. For example,  for all
$t\in [x_{2},X(y_{2})]$,
     \begin{equation*}
          \omega\left(t\right)=\frac{\omega_{1}+\omega_{2}}{2}+
          \int_{f\left(x_{2}\right)}^{f(t)}
          \frac{\textnormal{d}z}{\sqrt{4\left(z-f\left(x_{1}\right)\right)
          \left(z-f\left(x_{2}\right)\right)\left(z-f\left(x_{3}\right)\right)}}.
     \end{equation*}
Making an expansion of this quantity in the neighborhood of zero,
using that $\wp_{1,3}(\omega)=1/\omega^{2}+\mathcal{O}(\omega^{2})$
and that $\wp_{1,3}$ is even on $\mathbb{C}$, as well
as~(\ref{def_CGF_general_case}), we obtain that
     \begin{equation*}
          w\left(t\right)=\frac{\left(f\left(x_{1}\right)-f\left(x_{2}\right)\right)
          \left(f\left(x_{2}\right)-f\left(x_{3}\right)\right)}
          {f'\left(x_{2}\right)\left(x_{2}-t\right)}+\widehat{w}(t),
     \end{equation*}
where $\hat{w}$ is holomorphic in the neighborhood of $x_{2}$.

We will now study the behavior of the function $w$ in the neighborhood of
$x_{3}$. Define $R=[0,\omega_{3}]\times
[-\omega_{1}/(2\imath),\omega_{1}/(2\imath)]$ and note that
$\omega(\mathbb{C})-(\omega_{1}+\omega_{2})/2=
[0,\omega_{2}/2]\times [-\omega_{1}/(2\imath),\omega_{1}/(2\imath)]$.

We will consider separately the three cases $\Delta=0$, $\Delta<0$ and $\Delta>0$.
First of all let us show  that $\Delta<0$ (resp.\ $\Delta=0$, $\Delta>0$) implies
$\omega_{3}>\omega_{2}/2$ (resp.\ $\omega_{3}=\omega_{2}/2$, $\omega_{3}<\omega_{2}/2$).
We already know from Subsections~\ref{Galois_automorphisms}
and~\ref{Finite_groups_and_s_s_functions} that $\Delta=0$ is equivalent to
$\omega_{2}/\omega_{3}=2$. As a consequence, we obtain that to prove
that $\Delta<0$ (resp.\ $\Delta>0$) is equivalent to
$\omega_{3}>\omega_{2}/2$ (resp.\ $\omega_{3}<\omega_{2}/2$), it suffices to
prove that there exists \emph{one} walk verifying simultaneously
$\Delta<0$ and $\omega_{3}>\omega_{2}/2$ (resp.\ $\Delta>0$ and
$\omega_{3}<\omega_{2}/2$). Indeed, using the continuity of
$\omega_{2}$, $\omega_{3}$ and $\Delta$ w.r.t.\ the parameters
$p_{i j}$ and the intermediate value theorem, we obtain the results
for \emph{all} walks. We have already seen that the second walk of
Figure~\ref{SRW_Flatto_X} (i.e.\ $p_{-11}+p_{10}+p_{0-1}=1$)
is such that $\omega_{2}/\omega_{3}=3$.
Moreover, in this case, $\Delta=p_{-11}p_{10}p_{0-1}>0$. Also, we
can verify by a direct calculation starting from~(\ref{omega123})
that the walk $p_{-1-1}+p_{10}+p_{01}=1$ verifies
$\omega_{2}/\omega_{3}=3/2$~; moreover, for this walk
$\Delta=-p_{-1-1}p_{10}p_{01}<0$.

{\rm (i)}   Suppose that $\Delta=0$. Then $\omega_{2}/\omega_{3}=2$
            and
            $R=\omega(\mathbb{C})-(\omega_{1}+\omega_{2})/2$, so that by the same analysis
            as the one done just above for $x_{2}$, we find that
            $w$ has a pole of order one at $x_{3}$.

{\rm (ii)}  Suppose now that $\Delta<0$. Then
            $\omega_{3}>\omega_{2}/2$ and $\omega(\mathbb{C})-(\omega_{1}+\omega_{2})/2$
            is strictly included in $R$, in such a way that $w$ has no
            poles except at $x_{2}$
            and is two-valued on $[x_{3},x_{4}]$, for any (finite or infinite)
            order of the group $\mathcal{H}$.

{\rm (iii)} Suppose at last that $\Delta>0$. This implies
            $\omega_{3}<\omega_{2}/2$,
            thus $\omega(\mathbb{C})-(\omega_{1}+\omega_{2})/2$
            contains strictly $R$. Moreover, we can write,
            with $n=\lfloor  \omega_{2}/\left(2\omega_{3}\right)\rfloor$~:
     \begin{equation*}
          \omega\left(\mathbb{C}\right)-\frac{\omega_{1}+\omega_{2}}{2}=
          \bigcup_{k=0}^{n-1}
          \left(k\omega_3+R\right)\cup \widehat{R},\hspace{5mm}
          \widehat{R}=\left[n\omega_{3}/2,\omega_{2}/2\right]\times
          \left[-\omega_{1}/\left(2\imath\right),\omega_{1}/\left(2\imath\right)\right].
     \end{equation*}

This equality, added to the fact that $\wp_{1,3}$ has, on the
fundamental parallelogram $[0,\omega_{3}[\times [0,\omega_{1}/\imath[$,
only one pole, at zero and of order two, shows that $w$ has $\lfloor
\omega_{2}/(2\omega_{3})\rfloor$ double poles at points lying in
$]x_{2},x_{3}[\cap (\mathbb{C}\setminus \mathcal{G}_{\mathcal{M}})$.

Consider the particular case $\omega_{2}/\omega_{3}\in 2\mathbb{N}$.
Then the rectangle $\hat{R}$ is reduced to one point and
$\omega([x_{3},x_{4}])-(\omega_{1}+\omega_{2})/2$ is congruent to
$[-\omega_{1}/(2\imath),\omega_{1}/(2\imath)]$, so that for the same reasons
as in the case $\Delta=0$, we see that $w$ has a simple pole at
$x_{3}$.

Consider next the other particular case $\omega_{2}/\omega_{3}\in
2\mathbb{N}+1$. Then
$\omega([x_{3},x_{4}])-(\omega_{1}+\omega_{2})/2$ is congruent to
$\omega_{3}/2+[-\omega_{1}/(2\imath),\omega_{1}/(2\imath)]$, in the
neighborhood of which $\wp_{1,3}$ is holomorphic, so $w$ has no pole
at $x_{3}$ in this case.

If $\omega_{2}/\omega_{3}\notin \mathbb{N}$, then
$\omega([x_{3},x_{4}])-(\omega_{1}+\omega_{2})/2$ is congruent
neither to $\omega_{3}/2+[-\omega_{1}/(2\imath),\omega_{1}/(2\imath)]$ nor to
$[-\omega_{1}/(2\imath),\omega_{1}/(2\imath)]$, in particular $w$ has no pole
at $x_{3}$.

Consider now more global aspects and show that $w$ is meromorphic
on $\mathbb{C}$ if and only~if~$\omega_{2}/\omega_{3}\in\mathbb{N}$.

Recall from the beginning of the proof that for $t\in
[x_{3},x_{4}]$, the two values of $\omega(t)$ are such that
$\omega_{1}(t),\omega_{2}(t)\in [\omega_{2},\omega_{2}+\omega_{1}]$
and $\omega_{1}(t)-\omega_{2}/2=\omega_{2}/2-\omega_{2}(t)$. In
addition, we have shown just above that if
$\omega_{2}/\omega_{3}\in 2\mathbb{N}$ (resp.\
$\omega_{2}/\omega_{3}\in 2\mathbb{N}+1$), then
$\omega([x_{3},x_{4}])-(\omega_{1}+\omega_{2})/2$ is congruent to
$[-\omega_{1}/(2\imath),\omega_{1}/(2\imath)]$ (resp.\
$\omega_{3}/2+[-\omega_{1}/(2\imath),\omega_{1}/(2\imath)]$). But
$\omega\mapsto \wp_{1,3}(\omega)$ and $\omega\mapsto
\wp_{1,3}(\omega_{3}/2+\omega)$ are even functions, so that in both
cases $w$ is single-valued and meromorphic in the neighborhood of
$[x_{3},x_{4}]$.

Suppose now that $\omega_{2}/\omega_{3}\notin \mathbb{N}$. Since
$\omega(x_{3})=\omega_{1}/2+\omega_{2}$, then
$w(t)=\wp_{1,3}(\omega_{2}/2+(\omega(t)-\omega(x_{3})))$. On the
other hand, if $\omega$ is close to zero, we have~:
     \begin{equation*}
          \wp_{1,3}\left(\omega_{2}/2+\omega\right)=\wp_{1,3}\left(\omega_{2}/2\right)+
          \sum_{k=1}^{+\infty } \frac{\wp_{1,3}^{(2k)}
          \left(\omega_{2}/2\right)}{\left(2k\right)!}\omega^{2k}+
          \omega \sum_{k=0}^{+\infty } \frac{\wp_{1,3}^{(2k+1)}
          \left(\omega_{2}/2\right)}{\left(2k+1\right)!} \omega^{2k}.
     \end{equation*}
Also, by a similar calculation as the one done when we have
studied the behavior of $w$ in the neighborhood of $x_{2}$, we
obtain~:
     \begin{eqnarray*}
          \omega\left(t\right)-\omega\left(x_{3}\right)&=&
          \int_{f\left(t\right)}^{f\left(x_{3}\right)}
          \frac{\textnormal{d}z}{\sqrt{\left(z-f\left(x_{1}\right)\right)
          \left(z-f\left(x_{2}\right)\right)\left(z-f\left(x_{3}\right)\right)}}\\&=&
          -\left(\frac{-f'\left(x_{3}\right)\left(x_{3}-t\right)}
          {4\left(f\left(x_{3}\right)-f\left(x_{2}\right)\right)
          \left(f\left(x_{3}\right)-f\left(x_{1}\right)\right)}\right)^{1/2}
          \left(1+\left(t-x_{3}\right)\check{w}\left(t\right)\right),
\end{eqnarray*}
where $\check{w}$ is holomorphic in a neighborhood of $x_{3}$. Thus,
in a neighborhood of $x_{3}$, we can write $w(t)$ as the sum
$w(t)=w_{1}(t)+w_{2}(t)(x_{3}-t)^{1/2}$, where $w_{1}$ and $w_{2}$
are holomorphic in a neighborhood of $x_{3}$ and
     \begin{equation}\label{w1_w2_x3}
          w_{1}\left(x_{3}\right)=\wp_{1,3}\left(\omega_{2}/2\right),\hspace{5mm}
          w_{2}\left(x_{3}\right)=-\left(\frac{-f'\left(x_{3}\right)}
          {4\left(f\left(x_{3}\right)-f\left(x_{2}\right)\right)
          \left(f\left(x_{3}\right)-f\left(x_{1}\right)\right)}\right)^{1/2}
          \wp_{1,3}'\left(\omega_{2}/2\right).
     \end{equation}
 This closes the proof of Proposition~\ref{properties_w}.
\end{proof}

\begin{rem}
As a consequence of Proposition~\ref{properties_w}, the CGF $w$ is a
rational function in the particular cases $\omega_{2}/\omega_{3}\in
\mathbb{N}$. The theory of transformation of elliptic functions
gives a constructive way to write the expression of $w$ on $\mathbb{C}$.
Indeed, if we note $\wp_{1,2,n}$ the Weierstrass function
associated to the periods $\omega_{1}$ and $\omega_{2}/n$, then
the following formula, that can be found in~\cite{SG2},
     \begin{equation*}
          \wp_{1,2,n}\left(\omega\right)
          =\sum_{k=1}^{n-1}\big(\wp_{1,2}\left(\omega +k
          \omega_{2}/n\right)-
          \wp_{1,2}\left(k\omega_{2}/n \right)\big)+\wp_{1,2}\left(\omega\right),
     \end{equation*}
allows to express $\wp_{1,3}(\omega)$ in terms of $\wp_{1,2}(\omega)$
and $\wp_{1,2}'(\omega)$, using also
addition formulas for the Weierstrass function $\wp_{1,2}$. In practice,
that is what we use.
For instance, after some calculations, this method gives a suitable
CGF for the second walk of Figure~\ref{SRW_Flatto_X}, with
transition probabilities verifying $p_{-11}+p_{10}+p_{0-1}=1$~:
     \begin{equation*}
          w\left(t\right)=\frac{t}{\left(t-x_{2}\right)
          \left(t-\left(p_{-11}p_{0-1}/\left(p_{10}^{2}x_{2}\right)
          \right)^{1/2}\right)^{2}}.
     \end{equation*}
\end{rem}

We are now ready to state and prove the main result of
Section~\ref{Study_solution} on the asymptotic of the absorption
probabilities.

\noindent\textbf{Notation.}
Throughout the whole paper, for two sequences
$(a_{k})_{k}$ and $(b_{k})_{k}$ 
we will write $a_{k}\sim_{k\to +\infty} b_{k}$ or $a_{k}\sim b_{k}$ 
if $\lim_{k\to +\infty}a_{k}/b_{k}=1$.

\begin{thm}\label{main_result_asymptotic_absorption_probabilities}
We recall that for $k\in \mathbb{N}^{*}$, $h_{k}$ denotes $\mathbb{P}_{(n_{0} ,
m_{0})} (\text{to be absorbed at}\ (k,0))$.

Suppose first that $p_{11}+p_{-1-1}+p_{1-1}+p_{-11}<1$.
\newline $\bullet$  If $\omega_{2}/\omega_{3}\in 2\mathbb{N}$, then
$h_{k}\sim_{k\to +\infty} h_{1,k}+h_{2,k}$, $h_{1,k}$ and $h_{2,k}$ being defined
in~(\ref{asymptotic_hk1}) and~(\ref{asymptotic_h2k_omega23_inN}).
\newline $\bullet$ If $\omega_{2}/\omega_{3}\in 2\mathbb{N}+1$, then
$h_{k}\sim_{k\to +\infty} h_{1,k}$, $h_{1,k}$ being defined
in~(\ref{asymptotic_hk1}).
\newline $\bullet$ If $\omega_{2}/\omega_{3}\notin \mathbb{N}$, then
$h_{k}\sim_{k\to +\infty} h_{1,k}+h_{2,k}$, $h_{1,k}$ and $h_{2,k}$ being defined
in~(\ref{asymptotic_hk1}) and~(\ref{asymptotic_h2k_omega23_notinN}).

Suppose now that $p_{11}+p_{-1-1}+p_{1-1}+p_{-11}=1$.
\newline $\bullet$ If $k$ and $n_{0}+m_{0}$ don't have the same parity, then
$h_{k}=0$, since $(k,0)$ is not reachable.
\newline $\bullet$ If they have the same parity, then we obtain the
asymptotic of $h_{k}$ by multiplying by two the one of $h_{k}$ in the
case $p_{11}+p_{-1-1}+p_{1-1}+p_{-11}<1$, and this in the three
cases $\omega_{2}/\omega_{3}\in 2\mathbb{N}$,
$\omega_{2}/\omega_{3}\in 2\mathbb{N}+1$ and
$\omega_{2}/\omega_{3}\notin \mathbb{N}$.
\end{thm}

By Theorem~\ref{explicit_h(x)_second} and thanks to~(\ref{def_phi}),
$h$ can be split as $h=h_{1}+h_{2}$, where
     \begin{eqnarray}
          h_{1}\left(x\right)&=& x^{n_{0}}Y_{0}\left(x\right)^{m_{0}}
          +\frac{x}{\pi }\int_{x_{1}}^{x_{2}}
          \frac{t^{n_{0}-1}\mu_{m_{0}}\left(t\right)}{t-x}
          \sqrt{-d\left(t\right)}
          \textnormal{d}t,\label{residue_term}\\
          h_{2}\left(x\right)&=&\frac{1}{\pi }\int_{x_{1}}^{x_{2}}
          t^{n_{0}}\phi\left(t,x\right)
          \mu_{m_{0}}\left(t\right)\sqrt{-d\left(t\right)}
          \textnormal{d}t.\label{int_phi}
     \end{eqnarray}
Theorem~\ref{main_result_asymptotic_absorption_probabilities} is an
immediate consequence of Lemmas~\ref{lemma_residue_term}
and~\ref{lemma_int_phi} below on the asymptotic behavior of Taylor
coefficients of $h_{1}$ and $h_{2}$ respectively.

\begin{lem}\label{lemma_residue_term}
The function $h_{1}$, initially defined in $\mathbb{C}\setminus
[x_{1},x_{2}]\cup [x_{3},x_{4}]$ by~(\ref{residue_term}), admits a
holomorphic continuation in $\mathbb{C}\setminus [x_{3},x_{4}]$. We
still note $h_{1}$ the continuation of this function and we set
$h_{1}(x)=\sum_{k=0}^{+\infty}h_{1,k}x^{k}$. Suppose first that
$p_{11}+p_{-1-1}+p_{1-1}+p_{-11}<1$~; then
     \begin{equation}\label{asymptotic_hk1}
          h_{1,k}\sim \frac{m_{0}x_{3}^{n_{0}+1/2}}{4\sqrt{\pi}}
          \left(\frac{-d'\left(x_{3}\right)}
          {a\left(x_{3}\right)c\left(x_{3}\right)}\right)^{1/2}
          \left(\frac{c\left(x_{3}\right)}{a\left(x_{3}\right)}\right)^{m_{0}/2}
          \frac{1}{k^{3/2}x_{3}^{k}},\hspace{5mm}k\to \infty .
     \end{equation}
Suppose now that $p_{11}+p_{-1-1}+p_{1-1}+p_{-11}=1$~; then the
process can hit $(k,0)$ if and only if $k$ and $n_{0}+m_{0}$ have the same
parity. Therefore $h_{k}=0$ if $k$ and $n_{0}+m_{0}$ don't have the
same parity. If they have the same parity, then $h_{1,k}$ is equivalent to two times
the right member of~(\ref{asymptotic_hk1}).
\end{lem}

\begin{lem}\label{lemma_int_phi}
The function $h_{2}$, defined in~(\ref{int_phi}), is holomorphic in
$\mathbb{C}\setminus (w^{-1}([x_{1},x_{2}])\setminus [x_{1},x_{2}])$ and we set
$h_{2}(x)=\sum_{k=0}^{+\infty}h_{2,k}x^{k}$. Suppose first that
$p_{11}+p_{-1-1}+p_{1-1}+p_{-11}<1$~; in addition,
     \begin{enumerate}
          \item If $\omega_{2}/\omega_{3}\in 2\mathbb{N}$, then~:
               \begin{equation}\label{asymptotic_h2k_omega23_inN}
                    h_{2,k}\sim -\frac{m_{0}x_{2}^{n_{0}}x_{3}^{1/2}}{4\sqrt{\pi}}
                    \left(\frac{d'\left(x_{2}\right)}
                    {a\left(x_{2}\right)c\left(x_{2}\right)}\right)^{1/2}
                    \left(-\frac{\text{Res}\left(w,x_{2}\right)}
                    {\text{Res}\left(w,x_{3}\right)}\right)^{1/2}
                    \left(\frac{c\left(x_{2}\right)}
                    {a\left(x_{2}\right)}\right)^{m_{0}/2}
                    \frac{1}{k^{3/2}x_{3}^{k}},\hspace{5mm}k\to \infty ,
               \end{equation}
          where for $i=2,3$,
          $\text{Res}(w,x_{i})$ denotes
          the residue of the function $w$
          at $x_{i}$, where from
          Proposition~\ref{properties_w}
          it has a pole (of order one).
     \item If $\omega_{2}/\omega_{3}\in 2\mathbb{N}+1$, then
     $h_{2,k}=o(h_{1,k})$.
     \item If $\omega_{2}/\omega_{3}\notin \mathbb{N}$,
     then the following asymptotic holds
     as $k$ goes to infinity~:
          \begin{equation}\label{asymptotic_h2k_omega23_notinN}
               h_{2,k}\sim \frac{\sqrt{x_{3}}\wp_{1,3}'\left(\omega_{2}/2\right)
               \sqrt{-f'\left(x_{3}\right)}}
               {\sqrt{\pi \left(f\left(x_{3}\right)-f\left(x_{2}\right)\right)
               \left(f\left(x_{3}\right)-f\left(x_{1}\right)\right)}}
               \left(\int_{x_{1}}^{x_{2}}\frac{w'\left(t\right)
               t^{n_{0}}\mu_{m_{0}}\left(t\right)\sqrt{-d\left(t\right)}}
               {\left(w\left(t\right)-\wp_{1,3}\left(\omega_{2}/2\right)\right)^{2}}
               \textnormal{d}t\right)\frac{1}{k^{3/2}x_{3}^{k}},
          \end{equation}
     where if $x_{4}\neq \infty$ then
     $f(t)=d''(x_{4})/6+d'(x_{4})/(t-x_{4})$ and
     if $x_{4}=\infty$ then $f(t)=(d''(0)+d'''(0)t)/6$.
\end{enumerate}
Suppose now that $p_{11}+p_{-1-1}+p_{1-1}+p_{-11}=1$~; if $k$ and
$n_{0}+m_{0}$ don't have the same parity, then $h_{2,k}=0$. If they
have the same parity and if $\omega_{2}/\omega_{3}\in 2\mathbb{N}$ 
(resp.\ $\omega_{2}/\omega_{3}\in 2\mathbb{N}+1$, $\omega_{2}/\omega_{3}\notin \mathbb{N}$),
then the asymptotic of $h_{2,k}$ is given by two
times the right member of~(\ref{asymptotic_h2k_omega23_inN}) (resp.\ is
negligible w.r.t.~(\ref{asymptotic_hk1}), is given by two times the right member
of~(\ref{asymptotic_h2k_omega23_notinN})).
\end{lem}

\noindent{\it Proof of Lemma~\ref{lemma_residue_term}.} Suppose
first that $x_{4}>0$. Apply the residue theorem at infinity on the
contour $\mathcal{C}_{\epsilon }$, represented at the left of
Figure~\ref{Contours_residues}. So we obtain that for all $x$ inside the
infinite domain delimited by~$\mathcal{C}_{\epsilon }$,
     \begin{equation*}
          \frac{1}{2\pi \imath} \int_{\mathcal{C}_{\epsilon}}
          \frac{t^{n_{0}-1}Y_{0}\left(t\right)^{m_{0}}}{t-x}\textnormal{d}t=
          x^{n_{0}-1}Y_{0}\left(x\right)^{m_{0}}-
          P_{\infty}\left(x\mapsto x^{n_{0}-1}
          Y_{0}\left(x\right)^{m_{0}}\right)\left(x\right),
     \end{equation*}
where $P_{\infty}(x\mapsto x^{n_{0}-1}Y_{0}(x)^{m_{0}})$ is the
principal part at infinity of the meromorphic function at infinity
$x\mapsto x^{n_{0}-1}Y_{0}(x)^{m_{0}}$, see Proposition~\ref{explicit_Delta_zero}
for more details about the principal part at infinity
of a function. Furthermore, by definition of $\mu_{m_{0}}$,
     \begin{equation*}
          \lim_{\epsilon \to 0}\frac{1}{2\pi \imath} \int_{\mathcal{C}_{\epsilon}}
          \frac{t^{n_{0}-1}Y_{0}\left(t\right)^{m_{0}}}{t-x}\textnormal{d}t=
          -\frac{1}{\pi}\int_{x_{1}}^{x_{2}}\frac{t^{n_{0}-1}\mu_{m_{0}}\left(t\right)
          \sqrt{-d\left(t\right)}}{t-x}\textnormal{d}t+
          \frac{1}{\pi}\int_{x_{3}}^{x_{4}}\frac{t^{n_{0}-1}\mu_{m_{0}}\left(t\right)
          \sqrt{-d\left(t\right)}}{t-x}\textnormal{d}t.
     \end{equation*}
Therefore the function $h_{1}$ is just equal to~:
     \begin{equation}\label{equality_l1}
          h_{1}\left(x\right)=\frac{x}{\pi}\int_{x_{3}}^{x_{4}}
          \frac{t^{n_{0}-1}\mu_{m_{0}}\left(t\right)
          \sqrt{-d\left(t\right)}}{t-x}\textnormal{d}t+
          xP_{\infty}\left(x\mapsto x^{n_{0}-1}
          Y_{0}\left(x\right)^{m_{0}}\right)\left(x\right).
     \end{equation}
Moreover, thanks to Lemma~\ref{properties_X_Y}, we obtain that the
degree of the polynomial defined by the principal part above is
equal to $n_{0}$ if $p_{1-1}\neq 0$, $-\infty$ if $p_{1-1}=0$ and
$n_{0}\leq m_{0}$, $n_{0}-m_{0}$ if $p_{1-1}=0$ and $n_{0}>m_{0}$.
In any case, if $k$ is larger than this degree, then the following
equality holds~:
     \begin{equation}\label{h_1_k_before_asymptotic}
          h_{1,k}=\frac{1}{\pi}\int_{x_{3}}^{x_{4}}
          \frac{\mu_{m_{0}}\left(t\right)\sqrt{-d\left(t\right)}}
          {t^{k+1-n_{0}}}\textnormal{d}t.
     \end{equation}
We can then easily obtain the asymptotic of this integral as $k$ goes
to infinity, using Laplace's method, see
e.g.~\cite{Chat} page 275. We make an
expansion of the numerator of the integrand in~(\ref{h_1_k_before_asymptotic})
in the neighborhood of $x_{3}$, we obtain
$\mu_{m_{0}}(t)(-d(t))^{1/2}=\mu_{m_{0}}(x_{3})(-d'(x_{3}))^{1/2}(t-x_{3})^{1/2}
+(t-x_{3})^{3/2} f(t)$, where $f$ is holomorphic at $x_{3}$.
Classically, the second term in the previous sum
will lead to a negligible contribution,
so to get~(\ref{asymptotic_hk1}), it suffices therefore to use one hand that
     \begin{equation*}
          \int_{x_{3}}^{x_{4}} \frac{\sqrt{t-x_{3}}}{t^{k}}\text{d}t=
          \frac{\sqrt{\pi}}{2}\frac{1}{x_{3}^{k-3/2}}\frac{1}{k^{3/2}}+
          \mathcal{O}\left(\frac{1}{k^{5/2}}\right),
     \end{equation*}
and on an other hand to simplify $\mu_{m_{0}}(x_{3})$, using for this~(\ref{def_mu})
and the fact that $d(x_{3})=0$.


\begin{figure}[!ht]
\begin{picture}(30.00,135.00)
\includegraphics{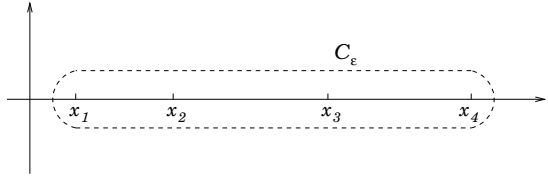}\hspace{95mm} \includegraphics{contournega2.eps}
\end{picture}
\caption{Contours of integration in the cases $x_{4}>0$ and $x_{4}<0$
respectively} 
\label{Contours_residues}
\end{figure}

Suppose now that $x_{4}<0$ and suppose in addition that
$p_{11}+p_{-1-1}+p_{1-1}+p_{-11}<1$. In this case, $Y_{0}$ is no
more meromorphic at infinity and has no principal part at infinity,
the previous argument does not run anymore. However, we will show
that the asymptotic~(\ref{asymptotic_hk1}) is still correct. To that
purpose, fix $R>-x_{4}$ and apply the classical residue theorem on
the contour $\mathcal{C}_{\epsilon,R}$ described in the right side
of Figure~\ref{Contours_residues}. After that $\epsilon$ has gone to
zero, we obtain
     \begin{equation}\label{before_def_f_R}
          x^{n_{0}-1}Y_{0}(x)^{m_{0}}=
          -\frac{1}{\pi}\int_{x_{1}}^{x_{2}}\frac{t^{n_{0}-1}
          \mu_{m_{0}}\left(t\right)\sqrt{-d\left(t\right)}}{t-x}\text{d}t+
          \frac{1}{\pi}\int_{x_{3}}^{R}\frac{t^{n_{0}-1}
          \mu_{m_{0}}\left(t\right)\sqrt{-d\left(t\right)}}
          {t-x}\text{d}t+f_{R}\left(x\right),
     \end{equation}
where $f_{R}$ is defined by~:
     \begin{equation*}
          f_{R}\left(x\right)=\frac{1}{2\pi \imath }
          \int_{\mathcal{C}\left(0,R\right)}\frac{t^{n_{0}-1}
          Y_{0}\left(t\right)^{m_{0}}}{t-x}\text{d}t+
          \frac{1}{\pi}\int_{-R}^{x_{4}}\frac{t^{n_{0}-1}\mu_{m_{0}}\left(t\right)
          \sqrt{-d\left(t\right)}}{t-x}\text{d}t,
     \end{equation*}
The first (resp.\ the second) Cauchy-type integral in the sum defining 
$f_{R}$ is holomorphic in the disc
$\mathcal{D}(0,R)$ (resp.\ $\mathcal{D}(0,-x_{4})$) so that $f_{R}$
is holomorphic in $\mathcal{D}(0,-x_{4})$. Moreover, since we have
supposed $p_{11}+p_{-1-1}+p_{1-1}+p_{-11}<1$,
Lemma~\ref{lemma_branched_points} yields $x_{3}<-x_{4}$. In
particular, this implies that, as $k$ goes to infinity, the $k$th coefficient of
the Taylor series at zero of $f_{R}$ is $o(1/r^{k})$, where
$x_{3}<r<-x_{4}$. These coefficients will be therefore
negligible w.r.t.\ those of
     \begin{equation*}
          \frac{1}{\pi}\int_{x_{3}}^{R}\frac{t^{n_{0}-1}
          \mu_{m_{0}}\left(t\right)\sqrt{-d\left(t\right)}}
          {t-x}\text{d}t.
     \end{equation*}
We calculate the asymptotic of the coefficients of the function
above using Laplace's method, as in the case $x_{4}>0$,
the asymptotic~(\ref{asymptotic_hk1}) is thus still valid.

If $p_{11}+p_{-1-1}+p_{1-1}+p_{-11}=1$ (which implies, see
Lemma~\ref{lemma_branched_points}, that $x_{4}=-x_{3}<0$), then the process
can reach $(k,0)$ if and only if $k$ and $n_{0}+m_{0}$ have the same
parity. In particular, if $k$ and $n_{0}+m_{0}$ don't have the same
parity then $h_{k}=0$. Besides, we can show that in
this case $h_{1,k}=0$,
using that $[x_{4},x_{1}]=[-x_{3},-x_{2}]$ and the fact that
$Y_{0}$, $Y_{1}$ and $w$ are odd functions. If they have the same
parity, then the asymptotic of the coefficients of $f_{R}$ is no
more negligible~: indeed, thanks to
Lemma~\ref{lemma_branched_points}, $x_{3}=-x_{4}$ and after the use
of $\mu_{m_{0}}(-t)=(-1)^{m_{0}-1}\mu_{m_{0}}(t)$,
see~(\ref{def_mu}), we obtain~:
     \begin{equation}\label{property_int_diagonal_walk}
          \frac{1}{\pi}\int_{-R}^{x_{4}}\frac{t^{n_{0}-1}
          \mu_{m_{0}}\left(t\right)\sqrt{-d\left(t\right)}}
          {t-x}\text{d}t=\frac{\left(-1\right)^{n_{0}+m_{0}}}
          {\pi}\int_{x_{3}}^{R}\frac{t^{n_{0}-1}
          \mu_{m_{0}}\left(t\right)\sqrt{-d\left(t\right)}}
          {t+x}\text{d}t.
     \end{equation}
For this reason, if $k$ and $n_{0}+m_{0}$ have the same parity, we
obtain that the asymptotic of $h_{1,k}$ is given by two
times~(\ref{asymptotic_hk1}), using once again Laplace's method.
\hfill $\square $

\medskip

\noindent{\it Proof of Lemma~\ref{lemma_int_phi}.} This proof is
based on the following principle, known as Pringsheim Theorem~:
the asymptotic of the coefficients of a Taylor series at $0$
can be found starting from the precise
knowledge of the first singularity of this Taylor series~; 
here ``the first singularity" means the singularity the nearest in modulus 
from zero. So we have to find the first singularity of the function $h_{2}$, defined
in~(\ref{int_phi}). To do this, we have to find the singularities of
the function $\phi$, defined in~(\ref{def_phi}), that appears in the
integral defining $h_{2}$. \emph{Grosso modo}, $\phi$, as a function
of two variables, can have singularities of two different kinds~:
either they are fixed in the sense that for all $t$ (resp.\ all $x$)
the function $\phi$ of the variable $x$ (resp.\ $t$) has the same
singularities, and in this case the set of all the singularities can
be written as a direct product, or they are movable in the sense
that the set of the singularities can not be written as a direct
product. In our case, according to $w$ is meromorphic or not, these
two possibilities can arise, eventually simultaneously. Indeed, if
$w$ is not meromorphic, $x_{3}$ will be a fixed singularity of
$\phi$, hence also a singularity of $h_{2}$, and $\{(t,x) \in \mathbb{C}^{2} :
w(t)=w(x)\}\setminus \{(y,y) : y\in \mathbb{C}\}$ is the set of
movable singularities of $\phi$~; in particular if we define
$\hat{x}=\inf \{x\in ]x_{2},+\infty[ : w(x)\in
w([x_{1},x_{2}])\}$, then $\inf\{x_{3},\hat{x}\}$
is the first singularity of $h_{2}$.

In our case, we will show that
if $\omega_{2}/\omega_{3}\in 2\mathbb{N}$, then $\hat{x}=x_{3}$,
and if $\omega_{2}/\omega_{3}\in 2\mathbb{N}+1$ or
$\omega_{2}/\omega_{3}\notin \mathbb{N}$, then $\hat{x}>x_{3}$.

We first suppose that $p_{11}+p_{-1-1}+p_{1-1}+p_{-11}<1$ and we
will say at the end of the proof how to adapt our arguments to the
walks having transition probabilities such that
$p_{11}+p_{-1-1}+p_{1-1}+p_{-11}=1$.

Suppose first that $\omega_{2}/\omega_{3}\in 2\mathbb{N}$. In this
case, we obtain $\lim_{x
\stackrel{>}{\to} x_{2}}w(x)=\lim_{x \stackrel{<}{\to} x_{3}}w(x)=-\infty $,
thanks to Proposition~\ref{properties_w}. More
generally, from Proposition~\ref{properties_w} and its proof we
deduce that there exists a holomorphic function $\sigma $ defined at
least in a neighborhood of $[x_{2},x_{3}]$ such that $\sigma \circ
\sigma=\text{id}$, $\sigma(x_{3})=x_{2}$ and $w\circ\sigma=w$.

In addition, we have already seen that the function $(t,x)\mapsto
w'(t)/(w(t)-w(x))-(x_{2}-x)/((x_{2}-t)(t-x))$ is holomorphic in
${\mathcal{G}_{\mathcal{M}}}^{2}$, in particular in the neighborhood
of $(x_{2},x_{2})$ and more generally in the neighborhood of
$\{(y,y) : y\in \mathcal{G}_{\mathcal{M}}\}$, but it is not
holomorphic at $(x_{2},x_{3})$ and in fact is not holomorphic at every point of
$\{(y,\sigma(y)) : y\in \mathcal{G}_{\mathcal{M}}\}$. On the other
hand, the function $(t,x)\mapsto
w'(t)/(w(t)-w(x))-(x_{2}-x)/((x_{2}-t)(t-x))
-\sigma'(t)(x_{3}-x)/((x_{3}-\sigma(t))(\sigma(t)-x))$ is
holomorphic in the neighborhood of $\{(y,\sigma(y)) : y\in
\mathcal{G}_{\mathcal{M}}\}$. This is why the function
     \begin{equation}\label{def_hat_phi}
          \widehat{\phi}\left(t,x\right)=\phi\left(t,x\right)
          -\frac{\sigma'\left(t\right)x}
          {\sigma\left(t\right)
          \left(\sigma\left(t\right)-x\right)}
     \end{equation}
is holomorphic in a neighborhood of $[x_{1},x_{2}]\times
\{x\in \mathbb{C} : |x|<R\}$ where $R>x_{3}$ and in particular for
all $t\in [x_{1},x_{2}]$, $x\mapsto \hat{\phi}(t,x)$ is holomorphic
on $\{x\in \mathbb{C} : |x|<R\}$. Therefore, the $k$th
coefficient of the Taylor series at zero of the function
     \begin{equation*}
          \frac{1}{\pi }\int_{x_{1}}^{x_{2}}
          \widehat{\phi}\left(t,x\right)\mu_{m_{0}}\left(t\right)
          t^{n_{0}}\sqrt{-d\left(t\right)}\textnormal{d}t
     \end{equation*}
is $o(1/R^{k})$ with $R>x_{3}$~; in particular it is
exponentially negligible w.r.t.\ to (\ref{asymptotic_hk1})
and (\ref{asymptotic_h2k_omega23_inN}). In other words, it remains
to evaluate the contribution of the coefficients of $g_{2}$, defined by
     \begin{equation}\label{def_l2}
          g_{2}\left(x\right)=\frac{x}{\pi }\int_{x_{1}}^{x_{2}}
          \frac{\sigma'\left(t\right)\mu_{m_{0}}\left(t\right)
          t^{n_{0}}}{\sigma\left(t\right)
          \left(\sigma\left(t\right)-x\right)}
          \sqrt{-d\left(t\right)}\textnormal{d}t=-
          \frac{x}{\pi }\int_{x_{3}}^{\sigma\left(x_{1}\right)}
          \frac{\sigma\left(t\right)^{n_{0}}
          \mu_{m_{0}}\left(\sigma\left(t\right)\right)}{t\left(t-x\right)}
          \sqrt{-d\left(\sigma\left(t\right)\right)}\textnormal{d}t,
     \end{equation}
where the second equality above comes from the change of variable
$u=\sigma(t)$. Then, we apply Laplace's method, and we will
obtain~(\ref{asymptotic_h2k_omega23_inN}) as soon as we will have
proved that
$\sigma'(x_{3})=\text{Res}(w,x_{2})/\text{Res}(w,x_{3})$. To do
this, start by differentiating the equality $w(t)=w(\sigma(t))$, we
obtain that $\sigma'(x_{3})=\lim_{x \to x_{3}}
w'(t)/w'(\sigma(t))$, what implies $\sigma'(x_{3})=\lim_{x\to x_{3}} 
\text{Res}(w,x_{3})(\sigma(t)-x_{2})^{2}/(\text{Res}(w,x_{2})(t-x_{3})^{2})$. Since
$\sigma'(x_{3})\neq0$ it follows that
$\sigma'(x_{3})=\text{Res}(w,x_{2})/\text{Res}(w,x_{3})$.

Suppose now that $\omega_{2}/\omega_{3}\in 2\mathbb{N}+1$. In this
case,  since $w$ has a pole at $x_{2}$ but is holomorphic at $x_{3}$
by Proposition~\ref{properties_w}, the function $\phi $ is
continuous on $[x_{1},x_{2}]\times \{x\in \mathbb{C} : |x|<R\}$
where $R>x_{3}$ and for all $t\in [x_{1},x_{2}]$, $x\mapsto
\phi(t,x)$ is holomorphic on $\{x\in \mathbb{C} : |x|<R\}$. This is
why the $k$th coefficient of the Taylor series at zero
of $h_{2}$ is in this case $o(1/R^{k})$, with $R>x_{3}$, that is,
in particular $h_{2,k}=o(h_{1,k})$

Consider now the general case, namely $\omega_{2}/\omega_{3}\notin
\mathbb{N}$. In  Proposition~\ref{properties_w}, we have shown that
in the neighborhood of $x_{3}$, $w$ can be written as
$w_{1}(x)+w_{2}(x)(x_{3}-x)^{1/2}$, where $w_{1}$ and $w_{2}$ are
functions holomorphic at $x_{3}$, and their values at $x_{3}$ are
made explicit in~(\ref{w1_w2_x3}). In particular, in the
neighborhood of $x_{3}$, we can write
$h_{2}(x)=f(x)+g(x)(x_{3}-x)^{1/2}$, with $f$ and $g$ holomorphic in
the neighborhood of $x_{3}$, and with~(\ref{int_phi}),
     \begin{equation}\label{value_g(x3)}
          g\left(x_{3}\right)=
          \frac{w_{2}\left(x_{3}\right)}{\pi }
          \int_{x_{1}}^{x_{2}}\frac{w'\left(t\right)}
          {\left(w\left(t\right)-w\left(x_{3}\right)\right)^{2}}
          t^{n_{0}}\mu_{m_{0}}\left(t\right)\sqrt{-d\left(t\right)}\text{d}t.
     \end{equation}

We can now easily find the asymptotic of the coefficients of the
Taylor series at $0$ of the function $h_{2}$, following 
Pringsheim Theorem, mentioned at the beginning
of the proof and summarized below~:
if $F(z)=\sum_{k}c_{k}z^{k}$ is a function (i) holomorphic
in the open disc of radius~$r$ (ii) having
a holomorphic continuation at every point
of the circle of radius~$r$ except $r$ (iii) having at $r$
an algebraic singularity in the sense
that in the neighborhood of $r$,
$F$ can be written as $F(z)=F_{0}(z)+
\sum_{i=1}^{d}F_{i}(z)(1-z/r)^{\theta_{i}}$
where the $F_{i}$, $i\geq 0$, are holomorphic functions
in the neighborhood of $r$,
not vanishing at $r$ for $i\geq 1$, the $\theta_{1}<\cdots <\theta_{d}$ are rational
but not integer, then the asymptotic of the coefficients
of the Taylor series at $0$ can easily be calculated~:
$c_{k}\sim F_{1}(r)r^{k}/(\Gamma(-\theta_{1})k^{\theta_{1}+1})$ as $k\to +\infty $, 
$\Gamma$ being the classical Gamma function.

So, using Pringsheim Theorem with $F=h_{2}$
--~thanks to Corollary~\ref{continuation_h_h_tilde}, or from the current proof,
$h_{2}$ is continuable holomorphically through every point of the circle
$\mathcal{C}(0,x_{3})$ except $x_{3}$~; this is here that we use the
hypothesis $p_{11}+p_{-1-1}+p_{1-1}+p_{-11}<1$~: indeed, under this
assumption, $h_{2}$ is holomorphic at $-x_{3}$~--,
$F_{1}=g$, $r=x_{3}$, $F_{1}(r)=g(x_{3})$ written
in~(\ref{value_g(x3)}), $\theta_{1}=1/2$ and using the
fact that $\Gamma(-1/2)=-2\sqrt{\pi}$, we get immediately the
announced asymptotic. 

Suppose now that $p_{11}+p_{-1-1}+p_{1-1}+p_{-11}=1$, and note that
for these walks, we can equally have $\omega_{2}/\omega_{3}\in
2\mathbb{N}$, $\omega_{2}/\omega_{3}\in 2\mathbb{N}+1$ or 
$\omega_{2}/\omega_{3}\notin \mathbb{N}$. In fact,
as in the proof of Lemma~\ref{lemma_residue_term}, we have, as well
as the contribution of the point $x_{3}$,  to take under account the
one of the point $x_{4}$, equal in this case to $-x_{3}$. So we do
the same analysis, but then we have in addition a term
like~(\ref{def_l2}) or~(\ref{value_g(x3)}) with $x_{4}$ instead
$x_{3}$. By doing the change of variable $t\mapsto -t$, as
in~(\ref{property_int_diagonal_walk}), we obtain that if $k$ and
$n_{0}+m_{0}$ don't have the same parity, then the contribution is
zero, and if they have, then the asymptotic of $h_{2,k}$ is given by
two times~(\ref{asymptotic_h2k_omega23_notinN}), using once again
Laplace's method and an adaptation of Pringsheim
Theorem for odd and even functions. \hfill $\square $

\begin{rem}
If $\omega_{2}/\omega_{3}\in 2\mathbb{N}+1$, then $\omega_{2}/2$
is congruent to $\omega_{3}/2$ and in particular
$\wp_{1,3}'(\omega_{2}/2)$ vanishes, what proves differently that
there is in this case a negligible contribution of the integral~(\ref{int_phi}) to
the asymptotic of the coefficients of $h$.
\end{rem}

Theorem~\ref{main_result_asymptotic_absorption_probabilities} can be
summarized as follows.

\begin{prop}\label{continuity_Martin_boundary}
The absorption probabilities $h_{i}$ admit
the following asymptotic as $i$ goes to $\infty$~:
     \begin{equation*}
          h_{i}\sim \frac{\left(-x_{3}d'\left(x_{3}\right)\right)^{1/2}}
          {4\pi^{1/2}a\left(x_{3}\right)}
          \left[ m_{0}\left(\frac{c\left(x_{3}\right)}{a\left(x_{3}\right)}
          \right)^{\left(m_{0}-1\right)/2}x_{3}^{n_{0}}
          -\widetilde{h}'\left(\left(\frac{c\left(x_{3}\right)}{a\left(x_{3}\right)}
          \right)^{1/2}\right)\right]
          \frac{1}{i^{3/2}x_{3}^{i}}.
     \end{equation*}
\end{prop}

\begin{proof}
We chose to do the proof in the general case, namely
$\omega_{2}/\omega_{3}\notin \mathbb{N}$, since the ideas of the
proof in the
cases $\omega_{2}/\omega_{3}\in 2\mathbb{N}$
and $\omega_{2}/\omega_{3}\in 2\mathbb{N}+1$
are rather similar.

We will need the following
consequence of Proposition~\ref{continuation_h_h_tilde_covering} and
Corollary~\ref{continuation_h_h_tilde}~: for all $x\in
\mathbb{C}\setminus [x_{3},x_{4}]$,
     \begin{equation}\label{relation_h_h_tilde}
          h\left(x\right)=x^{n_{0}}Y_{0}\left(x\right)^{m_{0}}
          -\widetilde{h}\left(Y_{0}\left(x\right)\right)-h_{00}.
     \end{equation}
Then, use one hand the explicit expression
of $h$, given in~(\ref{new_integral_form_h})
of Theorem~\ref{explicit_h(x)_second},
and on the other hand~(\ref{relation_h_h_tilde})
just above. We obtain~:
     \begin{equation*}
          \widetilde{h}\left(Y_{0}\left(x\right)\right)+h_{00}=-
          \frac{1}{\pi }\int_{x_{1}}^{x_{2}}
          t^{n_{0}}\mu_{m_{0}}\left(t\right)
          \left(\frac{w'\left(t\right)}{w\left(t\right)-w\left(x\right)}-
          \frac{w'\left(t\right)}{w\left(t\right)-w\left(0\right)}\right)
          \sqrt{-d\left(t\right)}\text{d}t.
     \end{equation*}
Differentiating this equality w.r.t.\ $x$, we get~:
     \begin{equation*}
           \widetilde{h}'\left(Y_{0}\left(x\right)\right)=-
           \frac{w'\left(x\right)}{Y_{0}'\left(x\right)}
           \frac{1}{\pi }\int_{x_{1}}^{x_{2}}
           t^{n_{0}}\mu_{m_{0}}\left(t\right)
           \frac{w'\left(t\right)}
           {\left(w\left(t\right)-w\left(x\right)\right)^{2}}
           \sqrt{-d\left(t\right)}\text{d}t.
     \end{equation*}
Since we have supposed $\omega_{2}/\omega_{3}\notin \mathbb{N}$,
we know from Proposition~\ref{properties_w}
that $w$ has at $x_{3}$ an algebraic singularity, so that
we can write $w$ as the sum
$w(x)=w_{1}(x)+w_{2}(x)(x_{3}-x)^{1/2}$,
where $w_{1}$ and $w_{2}$ are holomorphic at $x_{3}$,
moreover the values of
$w_{1}(x_{3})$ and $w_{2}(x_{3})$ are given
in~(\ref{w1_w2_x3}).
In addition, using the explicit expression of $Y_{0}$,
we obtain that $d(x)^{1/2}Y_{0}'(x)$ goes to
$-d'(x_{3})/(4a(x_{3}))$ as $x$ goes to $x_{3}$.
Also $(x_{3}-x)^{1/2}w'(x)$ goes to $-w_{2}(x_{3})/2$ as
$x$ goes to $x_{3}$. So we obtain~:
     \begin{equation*}
          \widetilde{h}'\left(Y_{0}\left(x_{3}\right)\right)=
          \widetilde{h}'\left(\left(\frac{c\left(x_{3}\right)}{a\left(x_{3}\right)}
          \right)^{1/2}\right)=
          \frac{2a\left(x_{3}\right)w_{2}\left(x_{3}\right)}
          {\left(-d'\left(x_{3}\right)\right)^{1/2}}
          \frac{1}{\pi }\int_{x_{1}}^{x_{2}}
          t^{n_{0}}\mu_{m_{0}}\left(t\right)
          \frac{w'\left(t\right)}
          {\left(w\left(t\right)-w\left(x_{3}\right)\right)^{2}}
          \sqrt{-d\left(t\right)}\text{d}t.
     \end{equation*}
With this last equality,~(\ref{asymptotic_hk1})
and~(\ref{asymptotic_h2k_omega23_notinN}), we have proved
Proposition~\ref{continuity_Martin_boundary}.
\end{proof}

We are now ready to prove Proposition~\ref{explicit_Delta_zero},
where we have given an explicit expression of $h$ in the particular
case $\Delta=0$ and $x_{4}>0$.

For this we need two preliminary results, stated
in Lemmas~\ref{Delta_zero_q_q1_q2} and~\ref{x1_x2_x3_x4_Delta_0}.
Before, recall that in~\cite{FIM}, the authors
find the explicit expression of the curve
$\mathcal{M}$, defined in~(\ref{def_curves_L_M})~; more precisely they
make explicit $q$, $q_{1}$ and $q_{2}$, three polynomials of
degree two, such that
$\mathcal{M}$ is equal to $\{u+i v \in \mathbb{C} :
q(u,v)^{2}-q_{1}(u,v)q_{2}(u,v)=0\}$.
In~\cite{Ras} is observed that these polynomials
can be written as~:
     \begin{equation}\label{def_q_q1_q2}
          \left|\begin{array}{ccc}
          p_{11}&1&p_{1-1}\\
          p_{01}&-2u&p_{0-1}\\
          p_{-11}&u^{2}+v^{2}&p_{-1-1}
          \end{array}\right|,\hspace{8mm}
          \left|\begin{array}{ccc}
          1&p_{10}&p_{1-1}\\
          -2u&-1&p_{0-1}\\
          u^{2}+v^{2}&p_{-10}&p_{-1-1}
          \end{array}\right|,\hspace{8mm}
          \left|\begin{array}{ccc}
          p_{11}&p_{10}&1\\
          p_{01}&-1&-2u\\
          p_{-11}&p_{-10}&u^{2}+v^{2}
          \end{array}\right|.
     \end{equation}

\begin{lem}\label{Delta_zero_q_q1_q2}
If $\Delta=0$, then there exist $\alpha\in \mathbb{R}$ and
$\beta \in \mathbb{R}\setminus \{0\}$ such that $q=\alpha q_{1}$
and $q_{2}=\beta q_{1}$.
\end{lem}

\begin{proof}
Take the following notations~:
$q(u,v)=\eta_{q,2}(u^{2}+v^{2})+\eta_{q,1} u + \eta_{q,0}$,
$q_{1}(u,v)=-\eta_{q_{1},2} (u^{2}+v^{2})+\eta_{q_{1},1} u + \eta_{q_{1},0}$,
$q_{2}(u,v)=\eta_{q_{2},2}(u^{2}+v^{2})+\eta_{q_{2},1}u - \eta_{q_{2},0}$.
Of course,
we get the explicit expression of these coefficients
by expanding the determinants~(\ref{def_q_q1_q2}).
In particular, we immediately
notice that $\eta_{q_{1},2}$,
$\eta_{q_{1},0}$, $\eta_{q_{2},2}$ and
$\eta_{q_{2},0}$ are positive. Also,
from straightforward calculations we obtain
     \begin{equation*}
          \frac{\eta_{q_{1},0}}{\eta_{q_{1},2}}-
          \frac{\eta_{q_{2},0}}{\eta_{q_{2},2}}=
          \frac{\Delta }{\eta_{q_{1},2}\eta_{q_{2},2}},
          \hspace{10mm}
          \frac{\eta_{q_{1},1}}{\eta_{q_{1},2}}-
          \frac{\eta_{q_{2},1}}{\eta_{q_{2},2}}=
          \frac{2p_{10}\Delta }{\eta_{q_{1},2}\eta_{q_{2},2}}.
     \end{equation*}
This immediately yields that if $\Delta=0$ then
there exists $\beta\neq 0$ such that
$q_{2}=\beta q_{1}$.

In turn, this fact and the equality
$q(u,v)=p_{01}q_{2}(u,v)+p_{0-1}q_{1}(u,v)+\Delta u$,
which is a Cramer relationship, consequence of~(\ref{def_q_q1_q2}),
entail that there exists $\alpha$, eventually zero, such that
$q=\alpha q_{1}$.
\end{proof}

Let $l_{1}<0$ and $l_{2}>0$ be the two roots of $q_{1}(x,0)$.
Then, thanks to Lemma~\ref{Delta_zero_q_q1_q2},
$\mathcal{M}$ is simply the circle of center
$\gamma=(l_{2}+l_{1})/2$ and radius $\rho=(l_{2}-l_{1})/2$.

The following result, that is a generalization of a straightforward result
concerning the walks verifying $p_{10}+p_{-10}+p_{01}+p_{0-1}=1$,
will also be useful in the proof of
Proposition~\ref{explicit_Delta_zero}.

\begin{lem}\label{x1_x2_x3_x4_Delta_0}
Suppose here that $\Delta=0$
and denote by $l_{1}<0$ and $l_{2}>0$ the two roots of $q_{1}(x,0)$,
$\gamma=(l_{2}+l_{1})/2$ and $\rho=(l_{2}-l_{1})/2$.
Then $(x_{2}-\gamma)(x_{3}-\gamma)=(x_{1}-\gamma)(x_{4}-\gamma)=\rho^{2}$.
\end{lem}

\begin{proof}
Define $\tilde{l}_{i}=Y(x_{i})=\epsilon_{i}(|c(x_{i})/a(x_{i})|)^{1/2}$,
$i\in\{1,\ldots,4\}$, with
$\epsilon_{2}=\epsilon_{3}=-\epsilon_{1}=-\epsilon_{4}=1$. From~(\ref{P_l}) of
Subsection~\ref{Saddle_point}, it can be deduced
that they are the four roots of
$\tilde{P}_{l}(y)=\tilde{q}(y,0)^{2}-\tilde{q}_{1}(y,0)\tilde{q}_{2}(y,0)$,
the polynomial $\tilde{q}$ (resp.\ $\tilde{q}_{1}$,
$\tilde{q}_{2}$) being obtained by making in $q$ 
(resp.\ $q_{1}$, $q_{2}$), defined in~(\ref{def_q_q1_q2}),
the change of parameters $p_{i j}\mapsto
p_{j i}$. Of course, we could prove an analogous of
Lemma~\ref{Delta_zero_q_q1_q2} that would be~: if $\Delta=0$ then
$\tilde{P}_{l}$ has two double roots~; in particular
$\tilde{l}_{2}=\tilde{l}_{3}$, so that, $a$ and $c$ being positive
for positive values of the argument,
$c(x_{2})/a(x_{2})=c(x_{3})/a(x_{3})$. In addition, by a direct
calculation, we notice that $c(x)/a(x)=c(y)/a(y)$ if and only if $x$
and $y$  are joined together by~:
     \begin{equation*}
          \left(p_{1-1}p_{01}-p_{0-1}p_{11}\right) x y+
          \left(p_{1-1}p_{-11}-p_{-1-1}p_{11}\right)
          \left(x+y\right)+\left(p_{0-1}p_{-11}-p_{-1-1}p_{01}\right)=0.
     \end{equation*}

$\bullet $ If $p_{1-1}p_{01}-p_{0-1}p_{11}\neq 0$, then the
polynomial $q(x,0)$ is non zero and thanks to Lemma~\ref{Delta_zero_q_q1_q2},
$l_{1}$ and $l_{2}$ are the two roots of $q(x,0)$~; in other
words they are the roots of $(p_{1-1}p_{01}-p_{0-1}p_{11})
x^{2}+2(p_{1-1}p_{-11}-p_{-1-1}p_{11})x
+(p_{0-1}p_{-11}-p_{-1-1}p_{01})$. Therefore, using the root-coefficient 
relationships, we have
$x_{2}x_{3}-(x_{2}+x_{3})(l_{1}+l_{2})/2+l_{1}l_{2}=0$, what exactly
means that $(x_{2}-\gamma)(x_{3}-\gamma)=\rho^{2}$.

$\bullet$ If $p_{1-1}p_{01}-p_{0-1}p_{11}=0$ then $q(x,0)$ is the null
polynomial. Indeed, in this case, $\deg(q)\leq 1$~; in addition,
thanks to Lemma~\ref{Delta_zero_q_q1_q2}, there exists $\beta\in \mathbb{R}$
such that $q=\beta q_{1}$. Since $\deg(q_{1})=2$, $\beta=0$ and $q$ is the 
null polynomial. In particular, we have,
for all $x$ and $y$, $c(x)/a(x)=c(y)/a(y)$ which leads to the
equality $a=\delta c$, where $\delta$ is some positive constant. In
such cases, $x_{2}$ and $x_{3}$ are easily calculated since they are
the two roots of $-b(x)-2(a(x)c(x))^{1/2}$, which is a polynomial of degree
two. So, it suffices to calculate explicitly $x_{2}$, $x_{3}$, also
$l_{1}$, $l_{2}$, and to notice that $(x_{2}-\gamma)(x_{3}-\gamma)=\rho^{2}$.

With the same arguments, we show that $(x_{1}-\gamma)(x_{4}-\gamma)=\rho^{2}$.
\end{proof}

\medskip

\noindent{\it Proof of Proposition~\ref{explicit_Delta_zero}.}
Proposition~\ref{explicit_Delta_zero} will result from the three
following facts.

(1) First, since $\mathcal{M}$ is a circle of center
    $\gamma$ and radius $\rho$ and thanks to Lemma~\ref{x1_x2_x3_x4_Delta_0},
    we easily verify that the function $w(t)=(x_{2}-\gamma)/(t-x_{2})-(x_{3}-\gamma)/(t-x_{3})$
    is a suitable CGF for the curve $\mathcal{M}$. In particular, the function
    $\sigma(t)=\gamma+\rho^{2}/(t-\gamma)$, defined in
    Proposition~\ref{explicit_Delta_zero}, lets
    $w$ invariant~; also,
    $\sigma'(x_{3})=\text{Res}(w,x_{2})/\text{Res}(w,x_{3})=
    -(x_{2}-\gamma)/(x_{3}-\gamma)$.

(2) Then, the key point is that since $\Delta=0$, the function $\hat{\phi}$
    defined in~(\ref{def_hat_phi}) is equal to zero. So,
    $h(x)$ is equal to the sum $h_{1}(x)+g_{2}(x)$, $h_{1}$ being
    defined in~(\ref{equality_l1}) and $g_{2}$ in~(\ref{def_l2}).

(3) Moreover, once again thanks to Lemma~\ref{x1_x2_x3_x4_Delta_0},
    $\sigma(x_{1})=x_{4}$, so that the integral~(\ref{def_l2})
    is an integral between $x_{3}$ and $x_{4}$.
    Moreover, by a direct calculation, we can show that if $P$
    stands for $a$, $b$ or $c$, then $P(\sigma(t))=(\rho/(t-\gamma))^{2}P(t)$.
    In particular, $d(\sigma(t))=(\rho/(t-\gamma))^{4}d(t)$ and
    $\mu_{m_{0}}(\sigma(t))=((t-\gamma)\rho)^{2}P(t)$.
    
Bringing together all these facts, we obtain~(\ref{explicit_h(x)_Delta=zero}).
\hfill $\square $

\section{Asymptotic of the Green functions}
\label{Green_functions_Martin_boundary}
In this section we find the
asymptotic of the Green functions
$G_{i,j}^{n_{0},m_{0}}=\mathbb{E}_{n_{0},m_{0}}
\big[\sum_{n\geq 0}1_{\{(X(n),Y(n))=(i,j)\}}\big]$ if $i,j>0$ and $j/i \to
\tan(\gamma)$ where $\gamma \in [0, \pi/2]$.

\subsection{Case $\gamma \in ]0,\pi/2[$}
\label{Saddle_point}

As it has been said in the introduction,
in the case $\gamma \in ]0, \pi/2[$ the procedure is essentially
the same as in~\cite{KV}
and~\cite{Ma1}, we just outline some details that are different.

It follows from~(\ref{functional_equation}) that
by Cauchy formula, for $\epsilon$ small enough,
     \begin{eqnarray}
          G_{i,j}^{n_{0},m_{0}}&=&\frac{1}{\left(2\pi \imath \right)^{2}}
          \int_{\left|x\right|=1-\epsilon }\frac{h\left(x\right)}{x^{i}}
          \left(\int_{\left|y\right|=1-\epsilon }\frac{\text{d}y}
          {Q\left(x,y\right)y^{j}}\right)\text{d}x\label{three_one}\\
          &+&\frac{1}{\left(2\pi \imath \right)^{2}}
          \int_{\left|y\right|=1-\epsilon}
          \frac{\widetilde{h}\left(y\right)+h_{00}}
          {y^{j}}\left(\int_{\left|x\right|=1-\epsilon}
          \frac{\text{d}x}
          {Q\left(x,y\right)x^{i}}\right)\text{d}y\label{three_two} \\
          &-&\frac{1}{\left(2\pi \imath \right)^{2}}
          \int_{\left|y\right|=1-\epsilon}
          \frac{1}
          {y^{j-m_{0}}}\left(\int_{\left|x\right|=1-\epsilon}
          \frac{\text{d}x}
          {Q\left(x,y\right)x^{i-n_{0}}}\right)\text{d}y.\label{three_three}
     \end{eqnarray}
Then we apply the residue theorem at infinity
to each inner integral above. Since
$Q(x,y)=a(x)(y-Y_{0}(x))(y-Y_{1}(x))=\tilde{a}(y)(x-X_{0}(y))(x-X_{1}(y))$,
we have to know the positions of $Y_{i}(x)$ and $X_{i}(y)$ w.r.t.\ the
circle $\mathcal{C}(0,1-\epsilon)$ when $|x|=|y|=1-\epsilon$. In
fact, we will prove that for any $x$, $y$ such that
$|x|=|y|=1-\epsilon$ and $\epsilon>0$ small enough~:
     \begin{equation}\label{location_curves_X_Y_1}
          \left|Y_{0}\left(x\right)\right|<1-\epsilon,\hspace{5mm}
          \left|Y_{1}\left(x\right)\right|>1-\epsilon,\hspace{5mm}
          \left|X_{0}\left(y\right)\right|<1-\epsilon,\hspace{5mm}
          \left|X_{1}\left(y\right)\right|<1-\epsilon.
     \end{equation}
Thanks to a proper change of parameters, it suffices of course to
prove the first two inequalities. We already know from
Lemma~\ref{properties_X_Y} that $Y_{1}(\{x\in\mathbb{C} : |x|=1\}\setminus
\{1\})\subset \{y\in\mathbb{C} : |y|>1\}$ and 
$Y_{0}(\{x\in\mathbb{C} : |x|=1\})\subset \{y \in\mathbb{C} :
|y|<1\}$. In particular, by continuity, this immediately leads to
the first inequality in~(\ref{location_curves_X_Y_1}), for
sufficiently small values of $\epsilon$. This also entails that
there exists $\theta_{0}(\epsilon)$, going to $0$ as $\epsilon$ goes
to $0$, such that for all $x=(1-\epsilon)\exp(\imath\theta)$ with
$\theta\in]\theta_{0}(\epsilon),2\pi-\theta_{0}(\epsilon)[$,
$|Y_{1}(x)|>1-\epsilon$. To conclude, it suffices to show that for
all $x=(1-\epsilon)\exp(\imath\theta)$ with
$\theta\in]-\theta_{0}(\epsilon),\theta_{0}(\epsilon)[$,
$|Y_{1}(x)|>1-\epsilon$. For this we will prove that there exists a
neighborhood of $1$, independent of $\epsilon$, where the curves
$Y_{1}(\{x\in\mathbb{C} : |x|=1-\epsilon\})$ and 
$Y_{1}(\{x\in\mathbb{C} : |x|=1\})$ don't
intersect~; then we will also show that
$Y_{1}(1-\epsilon)>Y_{1}(1)=1$. In order to show that the two above
curves don't intersect, remark that if they do, this means that
$Y_{1}(x)=Y_{1}(\hat{x})$, with some $x$, $\hat{x}$ such that
$|x|=1$, $|\hat{x}|=1-\epsilon$. This last equality is equivalent to
$\hat{x}x=\tilde{c}(Y_{1}(x))/\tilde{a}(Y_{1}(x))$. Since
$Y_{1}(1)=1$ and $\tilde{c}(1)/\tilde{a}(1)\in ]0,1[$, the previous
equality is not possible in a neighborhood of $1$ for $x$ and
$\hat{x}$. To prove that $Y_{1}(1-\epsilon)>1$, we remark that
an explicit calculation shows that $Y_{1}(x)>1$ if and only if
$a(x)+b(x)+c(x)<0$. But the polynomial $a+b+c$ goes to $\infty$ when $x\to
\pm \infty$ and has two real roots, $1$ and
$\tilde{c}(1)/\tilde{a}(1)<1$, so that $Y_{1}(1-\epsilon)>1$. 

Hence the inner integral of~(\ref{three_one}) 
(resp.\ of~(\ref{three_two}),~(\ref{three_three})) 
equals the residue at
$Y_{1}(x)$ (resp.\ at $X_{1}(y)$) with the constant $-2\pi \imath$,
the residue at infinity being zero. Then, letting $\epsilon \to 0$,
$G_{i,j}^{n_{0},m_{0}}$ is represented as the sum of the simple
integrals
     \begin{equation}\label{G_i_j_after_residue}
          G_{i,j}^{n_{0},m_{0}}=-\frac{1}{2\pi \imath }
          \int_{\left|x\right|=1} \frac{ h\left(x\right)}
          {d\left(x\right)^{1/2} x^{i} Y_{1}\left(x\right)^{j}}\text{d}x-
          \frac{1}{2\pi \imath}\int_{\left|y\right|=1}
          \frac{\widetilde{h}\left(y\right)+h_{00}-
          X_{1}\left(y\right)^{n_{0}}y^{m_{0}}}
          {\widetilde{d}\left(y\right)^{1/2} X_{1}\left(y\right)^{i}y^{j}} \text{d}y.
     \end{equation}      
These integrals are typical to apply  the saddle-point
method, see~\cite{Fed}.

To find the suitable saddle-point for
$\ln(xY_{1}(x)^{\tan(\gamma)})$ or equivalently for
$\ln(X_{1}(y)y^{\tan(\gamma)})$,
let us first have a closer look on the critical points of
$\chi_{\gamma,0}$ and $\chi_{\gamma,1}$, defined by
    \begin{equation}\label{def_chi}
          \chi_{\gamma,0}\left(x\right)= x Y_{0}\left(x\right)^{\tan\left(\gamma\right)},\hspace{5mm}
          \chi_{\gamma,1}\left(x\right)= x Y_{1}\left(x\right)^{\tan\left(\gamma\right)},\hspace{5mm}
          \gamma\in \left]0,\pi/2\right[.
     \end{equation}
The equations $\chi_{\gamma,0}'(x)=0$ and $\chi_{\gamma,1}'(x)=0$ are equivalent to
     \begin{eqnarray}
     \label{zty}
          \lefteqn{\pm d\left(x\right)^{1/2} 
          \left( a\left(x\right) c\left(x\right) -x\left( 
          a'\left(x\right)c\left(x\right)-a\left(x\right)  c'\left(x\right) \right)
          \tan\left(\gamma\right)/2\right)} \\
          &=& x \tan\left(\gamma\right)\left( a\left(x\right)  
          c\left(x\right) b'\left(x\right) - b\left(x\right) 
          \left(a'\left(x\right) c\left(x\right)  +a\left(x\right)  c'\left(x\right)
          \right)/2 \right). \nonumber
     \end{eqnarray}
Taking the square of both sides, we obtain that $P(\gamma, x)=0$,
where $P(\gamma, x)$ is the eight degree polynomial~:
   \begin{equation}\label{facto}
          P\left(\gamma,x\right)=-\left( a\left(x\right)c\left(x\right)+
          x\tan\left(\gamma \right) r\left(x\right) \right) d\left(x\right)
          +\left(x\tan(\gamma ) \right)^{2}P_l\left(x\right)
     \end{equation}
where we note
     \begin{equation}\label{P_l}
          P_{l}\left(x\right)=r\left(x\right)^{2}-r_{1}\left(x\right)r_{2}\left(x\right)=
          \lim_{\gamma \to \pi/2}
          \frac{P\left( \gamma ,x \right)}{\left(x\tan\left(\gamma\right)\right)^{2}}
     \end{equation}
and
     \begin{equation}\label{def_q_q1_q2_a_b_c_'}
          r=a c'-a' c,\hspace{8mm}
          r_{1}=b a' - b' a,\hspace{8mm}
          r_{2}=c b' - c' b.
     \end{equation}
Note that with the notations~(\ref{def_q_q1_q2}),
$r(x)=q(x,0)$ and for $i=1,2$, $r_{i}(x)=q_{i}(x,0)$.

The eight roots of the polynomial $P(\gamma,x)$
are the four critical points of
$\chi_{\gamma,0}(x)$ and those four of $\chi_{\gamma,1}(x)$,
$\gamma \in ]0, \pi/2[$.
It is immediate that in the limiting case $\gamma=0$
its roots are the branch points $x_{i}$, $i\in\{1,\ldots ,4\}$ and the roots of $a$ and $c$.
If $\gamma=\pi/2$, two of its roots are $0$,
two equal $\infty$ and four of them are the $X(y_{i})$,
$i\in \{1,\ldots ,4\}$, that are roots of $P_{l}(x)$.

Note that under the restricted hypothesis (H2')
the critical points can be made explicit.  The polynomial~(\ref{facto}) equals
$(\tan(\gamma)^{2}-1)x^{2}P_{1,4}(\gamma,x)P_{2,3}(\gamma,x)$,
where $P_{1,4}(\gamma, x)$ and $P_{2,3}(\gamma, x)$ are polynomials of second
degree, namely~:
     \begin{eqnarray*}
          P_{1,4}\left(\gamma,x\right)&=&p_{10}x^{2}-
          \frac{1+\left(1-(1-\tan(\gamma)^{2})
          (1-4p_{0-1}p_{01}+4p_{-10}p_{10}\tan(\gamma)^{2})\right)^{1/2}}
          {1-\tan(\gamma)^{2}}x+p_{-10},\\
          P_{2,3}\left(\gamma,x\right)&=&p_{10}x^{2}-
          \frac{1-\left(1-(1-\tan(\gamma)^{2})
          (1-4p_{0-1}p_{01}+4p_{-10}p_{10}\tan(\gamma)^{2})\right)^{1/2}}
          {1-\tan(\gamma)^{2}}x+p_{-10}.
   \end{eqnarray*}
The saddle-point for the first integral in~(\ref{G_i_j_after_residue})
is the biggest root of $P_{2,3}(\gamma,x)$.
This is the (unique) critical point of $\chi_{\gamma,1}(x)$
such that $x>0$ and $Y_{1}(x)>0$.
   In~\cite{KV} it has been characterized as the solution
   of~(\ref{uv}) below. Let us do it under (H2).
   We need to introduce the function $\phi(u,v)=\sum_{i,j}p_{i j}e^{i u}e^{j v}$
   for $(u,v) \in \mathbb{R}^{2}$.
   The equation $Q(x,y)=0$ with $x,y>0$ is equivalent
   to $\phi(u,v)=1$ with $u=\ln(x)$, $v=\ln(y)$.
   If $x>0$ is the critical point
   of $\chi_{0,\gamma}(x)$ such that $Y_{0}(x)>0$
   (resp.\ the one of $\chi_{1,\gamma}(x)$ such that $Y_{1}(x)>0$), then
   after some algebraic manipulations with $u=\ln(x)$ and $v =\ln(Y_{0}(x))$
   (resp.~$v=\ln(Y_{1}(x))$) we see that the equation~(\ref{zty}) is equivalent~to
   \begin{equation}\label{uv}
        \frac{\partial \phi\left(u,v\right)/\partial u}
        {\partial \phi\left(u,v\right)/\partial v}=\tan\left(\gamma\right).
   \end{equation}
   Then either
   \begin{equation}\label{eq_Ney_Spitzer}
          \frac{ \textnormal{grad}\left( \phi\left(u,v\right) \right)}
          {\left|\textnormal{grad}\left( \phi\left(u,v\right)\right)\right|}=
          \left( \cos\left(\gamma\right),\sin\left( \gamma \right) \right)
   \end{equation}
   or
   \begin{equation}\label{eq_Ney_Spitzer1}
          \frac{ \textnormal{grad}\left( \phi\left(u,v\right) \right)}
          {\left|\textnormal{grad}\left( \phi\left(u,v\right)\right)\right|}=
          \left( \cos\left(\gamma+\pi\right),\sin\left( \gamma+\pi \right)
          \right).
     \end{equation}
The mapping $(u,v) \mapsto \text{grad}( \phi(u,v))/|\text{grad}(\phi(u,v))|$
is a homeomorphism from $D=\{(u,v) \in \mathbb{R}^{2} : \phi(u,v)=1\}$
to the unit two-dimensional sphere, see~\cite{HE}.
Hence, for any $\gamma \in [0, \pi/2]$ there is
one solution of~(\ref{eq_Ney_Spitzer}) on $D$,
called $(u(\gamma), v(\gamma))$
and one solution of~(\ref{eq_Ney_Spitzer1}) on $D$,
called $(u(\gamma+\pi), v(\gamma+\pi))$.
Thus the positive critical point of $\chi_{\gamma,0}(x)$ with
$Y_{0}(x)>0$ and the one of $\chi_{\gamma,1}(x)$ with $Y_{1}(x)>0$
are among $e^{ u(\gamma)}$ and $e^{ u(\gamma+\pi)}$. In addition
$e^{u(\gamma)}$ (resp.~$e^{u(\gamma+\pi)}$) is critical for
$\chi_{i,\gamma}(x)$ if and only if $e^{v(\gamma)}$ (resp.\ $e^{
v(\gamma+\pi)}$) equals $Y_{i}(e^{u(\gamma)})$ (resp.\
$Y_{i}(e^{u(\gamma+\pi)})$, $i=0,1$. We verify that
$e^{v(\gamma)}=Y_{1}(e^{u(\gamma)})$ and $e^{
v(\gamma+\pi)}=Y_{0}(e^{u(\gamma+\pi)})$, so that $e^{u(\gamma)}$ is
the critical point of $\chi_{\gamma,1}(x)$ and $e^{u(\gamma+\pi)}$
is the one of $\chi_{\gamma,0}(x)$. Indeed, for $i=0,1$,
     \begin{equation}\label{Y_0_or_Y_1}
     \frac{\partial \phi}{\partial v}\left(u\left(\gamma\right),
     Y_{i}\left(e^{u(\gamma)}\right)\right)=
     \left[2Y_{i}\left(e^{u(\gamma)}\right)a\left(e^{u(\gamma)}\right)+
     b\left(e^{u(\gamma)}\right)\right]Y_{i}\left(e^{u(\gamma)}\right)^{-1}.
     \end{equation}
Moreover, on $[x_{2},x_{3}]$, $2a(x)Y_{1}(x)+b(x)=d(x)^{1/2}$ and
$2a(x)Y_{0}(x)+b(x)=-d(x)^{1/2}$, so that~(\ref{Y_0_or_Y_1}) is
negative for $Y=Y_{0}$, positive for $Y=Y_{1}$, what answers to the problem.

\noindent\textbf{Notation.}
 We put $s_{x}(\tan(\gamma))=e^{u(\gamma)}$ and
 $s_{y}(\tan(\gamma))=Y_{1}(e^{u(\gamma)})$.
\newline The mapping $\gamma \mapsto ( s_{x}(\tan(\gamma)), s_{y}(\tan(\gamma)))
$ is a homeomorphism between $[0,\pi/2]$ and
$\{(x,y)\in \mathbb{C}^{2} : x>0, y>0, Q(x,y)=0\}$. 
We note that $s_{x}(0)=x_{3}$, $s_{y}(0)=Y(x_{3})$ and
$s_x(\infty)=X(y_{3})$, $s_y(\infty)=y_{3}$. When $\gamma$ runs $[0,
\pi/2]$, $s_{x}(\tan(\gamma))$ monotonously decreases from
$x_{3}$ to $Y(x_{3})$ and $s_{y}(\tan(\gamma))$ monotonously
increases from $Y(x_{3})$ to~$y_{3}$.

Obviously the unique critical point of
$\tilde{\chi}_{\gamma,1}(y)=X_{1}(y)y^{\tan(\gamma)}$ with $y>0$
and $X_{1}(y)>0$ is $s_{y}(\tan(\gamma))$ defined above and
$X_{1}(s_{y}(\tan(\gamma))=s_{x}(\tan(\gamma))$.

\begin{thm}\label{thm_asymptotic_Green_functions_interior}
Let $\gamma \in ]0, \pi/2[$. If $j/i \to \tan(\gamma)$, then
    \begin{equation*}
        G_{i,j}^{n_{0},m_{0}} =
        \frac{s_{x}\left(
        \tan(\gamma) \right)^{n_{0}}s_{y}\left(\tan(\gamma)  \right)^{m_{0}}
        -h\left(s_{x}\left(\tan(\gamma) \right)\right)-\widetilde{h}
        \left(s_{y}\left(\tan(\gamma)
        \right)\right)-h_{00}}
        {s_x\left(j/i\right)^{i} s_y\left(j/i\right)^{j}}
        \left(\frac{C\left(\gamma\right)}{\sqrt{i}}+
        \mathcal{O}\left(\frac{1}{i \sqrt{i}}\right)\right),
   \end{equation*}
where the constant $C(\gamma)$, that does not depend on $(n_{0},m_{0})$, is equal to~:
     \begin{equation*}
          C(\gamma)=
          \left(2\pi\right)^{-1/2}\widetilde{d}\left(s_{y}
          \left(\tan\left(\gamma\right)\right)\right)^{-1}
          s_{y}\left(\tan(\gamma)\right)
          \left(-\left.\frac{\textnormal{d}^{2}}{\textnormal{d}y^{2}}
          \left\{ \frac{X_{1}\left(s_{y}\left(\tan(\gamma)\right)y\right)}{
          s_{x}\left(\tan(\gamma)\right)}
          y^{\tan\left(\gamma\right)}\right\} \right|_{y=1} \right)^{-1/2}.
     \end{equation*}
\end{thm}

\begin{proof}
One appropriately shifts the contour of integration $|x|=1$ (resp.\ $|y|=1$) 
in the first term of (\ref{G_i_j_after_residue}) (resp.\ the second) to a contour 
$\Gamma_\gamma$ (resp.\ $\tilde{\Gamma}_{\gamma}$) passing through $s_{x}(\tan(\gamma))$
(resp.\ $s_{y}(\tan(\gamma))$), which is the saddle-point of order one. $\Gamma_{\gamma}$ is the
contour of steepest descent (i.e.\ $\text{Im}(x
Y_{1}(x)^{\tan(\gamma)})=0$ on it) in a neighborhood of $s_{x}(\tan(\gamma))$ 
and outside this neighborhood it remains ``higher" than 
$s_{x}(\tan(\gamma))$ in the sense of the level curves of the function
$\chi_{\gamma,1}$. The construction
of $\Gamma_{\gamma}$ is done as in~\cite{KV} and~\cite{Ma1},
therefore we omit the details. Likewise, we construct the contour $\tilde{\Gamma}_{\gamma}$.
Then by Cauchy theorem, the first
(resp.\ the second) term in~(\ref{G_i_j_after_residue}) equals the
integral over $\Gamma_\gamma$ (resp.\ over $\tilde{\Gamma}_\gamma$),
whose asymptotic is computed by the saddle-point~method.
\end{proof}


\subsection{Asymptotic of the Green functions in the cases $\gamma=0,\pi/2$}
\label{Asymptotic_Green_functions_gamma_zero}

For that purpose, we first need to know
the behavior of $s_{x}(j/i)-s_{x}(0)$ and
$s_{y}(j/i)-s_{y}(0)$ when
$j/i$ is in a neighborhood of $0$.

\begin{lem}\label{rc}
Let $P_{l}$ be the polynomial defined in~(\ref{P_l}).
As $j/i \to 0$, the following expansions hold~:
     \begin{eqnarray}\label{eq1x}
          s_{x}\left(0\right)-s_{x}\left(j/i\right) &=&
          \frac{x_{3}^{2} P_{l}\left(x_{3}\right)}
          {-d'\left(x_{3}\right)a\left(x_{3}\right)c\left(x_{3}\right)}
          (j/i)^{2}+\mathcal{O}
          \left( (j/i)^{3} \right),\\
          s_{y}\left(j/i\right)-s_{y}\left(0\right) &=&
          \frac{x_{3}P_{l}(x_{3})^{1/2}}
          {2a(x_{3})^{3/2}c(x_{3})^{1/2}}
          j/i+\mathcal{O}
          \left( (j/i)^{2} \right).\label{eq2x}
     \end{eqnarray}
\end{lem}


\begin{proof}
Start by proving~(\ref{eq1x}). One hand, using~(\ref{facto}), we
obtain that $P(\arctan(j/i),x_{3})=s_{x}(j/i)^{2} P_{l}(x_{3})$~; on
the other hand, by definition of
$P(\arctan(j/i),x)$ and also with~(\ref{facto}), we have $P(\arctan(j/i)
,x_{3})=(x_{3}-s_{x}(j/i))R(j/i)$, with
$R(0)=-d'(x_{3})a(x_{3})c(x_{3})\neq 0$. Equation~(\ref{eq1x})
follows immediately.
Then, to prove~(\ref{eq2x}), start by remarking that
$Y_{1}(x)-Y_{1}(x_{3})=b(x_{3})/(2a(x_{3}))-b(x)/(2a(x))
+d(x)^{1/2}/(2a(x))$, so that
in the neighborhood of $x_{3}$,
$Y_{1}(x)-Y_{1}(x_{3})=f(x)+(-d'(x_{3})^{1/2}/(2a(x_{3})))
(x_{3}-x)^{1/2}(1+g(x))$, with $f$ and $g$ holomorphic
in the neighborhood of $x_{3}$ where they
take the value $0$.
Moreover, for all $j/i\in [0,\infty]$,
$s_{y}(j/i)=Y_{1}(s_{x}(j/i))$, see
Subsection~\ref{Saddle_point}.
This yields $s_{y}(\gamma)-s_{y}(0)=
(-d'(x_{3}))^{1/2}/(2a(x_{3}))
(x_{3}-s_{x}(\gamma))^{1/2}+\mathcal{O}(x_{3}-s_{x}(\gamma))$.
Finally, using~(\ref{eq1x}),
we obtain~(\ref{eq2x}).
\end{proof}

\begin{thm}
\label{proposition_asymptotic_Green_functions_angle_zero} Suppose
first that $p_{11}+p_{-1-1}+p_{1-1}+p_{-11}<1$. Then the Green
functions admit the following asymptotic as $i\to \infty$, $j>0$
and $j/i\to 0$~:
     \begin{equation}\label{equation_asymptotic_Green_functions_angle_zero}
          G_{i,j}^{n_{0},m_{0}}=
          \frac{C_{0}
          \left(m_{0}s_{x}\left(0\right)^{n_{0}}s_{y}\left(0\right)^{m_{0}-1}-
          \widetilde{h}'\left(s_{y}\left(0\right)\right)\right)}
          {s_{x}\left(j/i\right)^{i}s_{y}\left(j/i\right)^{j}}
          \left(\frac{j}{i\sqrt{i}}+\mathcal{O}\left(\frac{j^{2}}{i^{2}\sqrt{i}}\right)\right),
     \end{equation}
where the constant $C_{0}$ is equal to~:
     \begin{equation}\label{constant_asymptotic_Green_functions_angle_zero}
          C_{0}=\left(\frac{2}{\pi}\right)^{1/2}
          \frac{s_{y}'\left(0\right)s_{x}\left(0\right)^{1/2}}
          {\left(-\widetilde{d}\left(s_{y}\left(0\right)\right)
          X_{1}''\left(s_{y}\left(0\right)\right)\right)^{1/2}},
     \end{equation}
$s_{x}(0)=x_{3}$, $s_{y}(0)=Y_{1}(x_{3})$ and $s_{y}'(0)$ is
obtained from Lemma~\ref{rc}.

The analogous result as $j\to \infty$, $i>0$ and $j/i\to \infty$
holds after the proper change of parameters.

Suppose now that $p_{11}+p_{-1-1}+p_{1-1}+p_{-11}=1$.
If $n_{0}+m_{0}$ and $i+j$ don't have the same parity
then $G_{i,j}^{n_{0},m_{0}}=0$~; if they have the same parity, the asymptotic
of $G_{i,j}^{n_{0},m_{0}}$ is given by two times the right member
of~(\ref{equation_asymptotic_Green_functions_angle_zero}).
\end{thm}

\begin{rem}
\label{expl} Theorem
\ref{main_result_asymptotic_absorption_probabilities} and
Proposition~\ref{continuity_Martin_boundary}
 give immediately an explicit expression for
$m_{0}s_{x}(0)^{n_{0}}s_{y}(0)^{m_{0}-1}-\tilde{h}'(s_{y}(0))$
in the three cases $\omega_{2}/\omega_{3}\in 2\mathbb{N}$,
$\omega_{2}/\omega_{3}\in 2\mathbb{N}+1$
and $\omega_{2}/\omega_{3}\not\in \mathbb{N}$
that we do not list here.
\end{rem}

\begin{proof}
We detail the proof in the case $p_{11}+p_{-1-1}+p_{1-1}+p_{-11}<1$
and explain at the end what changes if
$p_{11}+p_{-1-1}+p_{1-1}+p_{-11}=1$.
 $G_{i,j}^{n_0,m_0}$ appears in (\ref{G_i_j_after_residue}) as the
sum of two integrals, one on the contour $|x|=1$, the other
on $|y|=1$. Using
Cauchy Theorem we will move these contours up to $s_x(j/i)$ and
$s_y(j/i)$ respectively  in a such way that the asymptotic of the
integrals on the new contours will be calculated by the saddle-point
method.
In order to define these new contours of integration, we need to
introduce the following -- eventually multivalued -- functions~:
     \begin{equation}\label{kappa_kappa_tilde}
          \kappa_{j/i}\left(x\right)=\ln \left(x\right)+
          \frac{j}{i}\ln \left(\frac{Y_{1}\left(s_{x}\left(j/i\right)x\right)}
          {s_{y}\left(j/i\right)}\right),\hspace{5mm}
          \widetilde{\kappa}_{j/i}\left(y\right)=
          \ln \left( \frac{X_{1}\left(s_{y}\left(j/i\right)y\right)}
          {s_{x}\left(j/i\right)}\right)+ \frac{j}{i}\ln \left(y \right).
     \end{equation}
According to Subsection~\ref{Saddle_point}, the function
$\kappa_{j/i}$ (resp.\ $\tilde{\kappa}_{j/i}$) has, for all $j/i>0$,
a critical point at $1$ where it equals $0$. Consider now the
functions $x_{j/i}(t)$ and $y_{j/i}(t)$ defined in neighborhoods
$V_{x, j/i}(0)$ and $V_{y, j/i}(0)$ of $0$ by
     \begin{equation}\label{def_steepest_parametrization}
          \kappa_{j/i}\left(x_{j/i}\left(t\right)\right)=t^{2},\hspace{10mm}
          \widetilde{\kappa}_{j/i}\left(y_{j/i}\left(t\right)\right)=t^{2},
     \end{equation}
and $\text{sign}(\text{Im}(x_{j/i}(t)))=\text{sign}(t)$,
$\text{sign}(\text{Im}(y_{j/i}(t)))=\text{sign}(t)$. These last
relationships are fixed in order to define $x_{j/i}$ and $y_{j/i}$ not
ambiguously.
By inverting the relationships~(\ref{def_steepest_parametrization}),
we obtain the explicit expression of $x_{j/i}$ and $y_{j/i}$. Here,
inverting means using the so-called Bürman-Lagrange formula, see
e.g.~\cite{Chat}, that allows to write the coefficients of the
Taylor series of a reciprocal function as integrals in terms of the
direct function.
As $j/i\to 0$, then $s_{y}(j/i)\to s_{y}(0)\in ]y_{2},y_{3}[$. We
may define
     \begin{equation*}\label{def_uniform_radius}
          \widetilde{\rho}=\inf_{j/i\in \left[0,1\right]}
          \inf\left\{y_{3}/s_{y}\left(j/i\right)-1,
          1-y_{2}/s_{y}\left(j/i\right)\right\},
     \end{equation*}
which verifies $\tilde{\rho}>0$ and $\tilde{\kappa}_{j/i}$ is
holomorphic in the disc $\mathcal{D}(1,\tilde{\rho})$ for all $j/i
\in [0,1]$. Using the Bürman-Lagrange formula we see that the radius
of $V_{y, j/i}(0)$ does not vanish as $j/i \to 0$~:  $y_{j/i}(t)$ is
in fact defined and holomorphic
in $D(0,\rho)$, $\rho$ being positive and independent of $j/i \in [0,1]$. Moreover,
the functions  $x_{j/i}(t)$ and $y_{j/i}(t)$ are joined together by~:
     \begin{equation}\label{join_steepest_x_y}
          x_{j/i}\left(-t\right)= \frac{X_{1}\left( s_{y}\left(j/i\right)
          y_{j/i}\left(t\right)\right)}{s_{x}\left(j/i\right)},\hspace{5mm}
          y_{j/i}\left(-t\right)= \frac{Y_{1}\left(s_{x}\left(j/i\right)
          x_{j/i}\left(t\right) \right)}{s_{y}\left(j/i\right)}.
     \end{equation}
This is a consequence of the automorphy relationships proved
in~\cite{FIM}~: for $x$ (resp.\ $y$) exterior to some curve (what
is the case here), $X_{1}(Y_{1}(x))=x$ (resp.\ $Y_{1}(X_{1}(y))=y$). 
For this reason and the fact that $1$ is a
critical point of order one of $\kappa_{j/i}$, we obtain that
$x_{j/i}(t)\in\{X_{1}(s_{y}(j/i) y_{j/i}(-t))/s_{x}(j/i), X_{1}(
s_{y}(j/i) y_{j/i}(t))/s_{x}(j/i)\}$. Then it suffices to calculate
the sign of the imaginary part in order to identify which of the two
possibilities  happens : we have $x_{j/i}(t)=X_{1}( s_{y}(j/i)
y_{j/i}(-t))/s_{x}(j/i)$.
The first equality of~(\ref{join_steepest_x_y}) shows that
$x_{j/i}$ is holomorphic as well in $V_{x,j/i}(0)=D(0,\rho)$ for
all $j/i\in[0,1]$.

The functions $\hat{x}_{j/i}(t)=s_{x}(j/i)x_{j/i}(t)$ and
$\hat{y}_{j/i}(t)=s_{y}(j/i)y_{j/i}(t)$ determine of course
the paths of steepest descent for
$\ln(xY_{1}(x)^{j/i})$ and $\ln(X_{1}(y)y^{j/i})$ respectively.
Note that the limiting curve $\hat{x}_{0}(t)$ runs the real line
decreasing from $\hat{x}_{0}(\rho)$ to $\hat{x}_{0}(0)=x_{3}$
and then increasing from $x_{3}$ to $\hat{x}_0(\rho)$ when
$t$ runs $[-\rho, \rho]$~; indeed $x_{0}(t)=\exp(t^{2})$.
The function $\hat{y}_{0}(t)=Y_{1}(\hat{x}_{0}(-t))$ runs
the values $Y_{1}^{-}(\hat{x}_{0}(t))=\lim_{x \uparrow \hat{x}_{0}(t)}Y_{1}(x)$
from $Y_{1}^{-}(\hat{x}_0(\rho))$ up to $Y_{1}(x_{3})$ and then the values
$Y_1^{+}(\hat{x}_{0}(t))=\lim_{x \downarrow \hat{x}_{0}(t)}Y_{1}(x)$
to $Y_1^{+}(\hat{x}_{0}(\rho))$.

For any $\rho$ small enough
we will now define two closed contours $\mathcal{C}_{\rho,j/i,x}
=x_{j/i}([-\rho, \rho])\cup \mathcal{A}_{\rho, j/i,x}$ and $\mathcal{C}_{\rho,j/i,y}
=y_{j/i}([-\rho, \rho])\cup \mathcal{A}_{\rho, j/i,y}$,
where $\mathcal{A}_{\rho, j/i, x}$, $\mathcal{A}_{\rho,j/i,y}$
verify the following three properties.

(i)   There exists a constant $c(\rho)>0$ such that
      $|\kappa_{j/i}(x)|>c(\rho)$ for any $x \in
      \mathcal{A}_{\rho, j/i,x}$ and any $j/i$ small enough, and such that
      $|\tilde{\kappa}_{j/i}(x)|>c(\rho)$ for any $y \in
      \mathcal{A}_{\rho, j/i,y}$ and any $j/i$ small enough.

(ii)  The first integrand in (\ref{G_i_j_after_residue}) does not have any
      singularities in the domain bounded by $|x|=1$ and
      the contour $s_{x}(j/i)\mathcal{C}_{\rho,j/i,x}$,
      the second integrand in (\ref{G_i_j_after_residue}) does not have any
      singularities  in the domain bounded by $|y|=1$ and
      the contour $s_{y}(j/i)\mathcal{C}_{\rho,j/i,y}$.

(iii) There exists a constant $L(\rho)$
      such that the lengths of the contours $\mathcal{C}_{\rho,j/i,x}$
      and $\mathcal{C}_{\rho,j/i,y}$ are bounded by $L(\rho)$
      for all $j/i$ small enough.

Let us construct such $\mathcal{A}_{\rho, j/i,x}$. We may take for
$\mathcal{A}_{\rho, 0,x}$ the circle of radius $x_{0}(\rho)>1$. Then
$|\kappa_{0}(x)|=|\ln(x)|= \ln(x_{0}(\rho))>(\ln(x_{0}(\rho)))/2$
for any $x\in \mathcal{A}_{\rho,0,x}$. Let us then take for
$\mathcal{A}_{\rho, j/i,x}$ the arc $\{|x_{j/i}(\rho)|\exp(\imath
\theta) : \theta\in
]\arg(x_{j/i}(\rho)),2\pi-\arg(x_{j/i}(\rho))[\}$. Since
$\hat{x}_{j/i}(\rho) \to \hat{x}_{0}(\rho)=x_{3}\exp(\rho^{2})$ as
$j/i \to 0$ and since $Y_{1}$ has no zero on $\mathbb{C}$ (it can be
easily shown that under (H4), $Y_{1}$ does not vanish on
$\mathbb{C}$), the property (i) remains valid for $\kappa_{j/i}(x)$
with $c(\rho)=(\ln(x_{0}(\rho)))/4$ for all $j/i$ small enough.
Furthermore, the singularities of the first integrand
in~(\ref{G_i_j_after_residue}) are the zeros of $d$, i.e.\ the
branch points $x_{i}$, $i\in\{1,\ldots ,4\}$. But with
Lemma~\ref{lemma_branched_points}, $x_{1}$ and $x_{2}$ are inside
the unit circle $|x|=1$~; as for $x_{3}$ and $x_{4}$, they are
outside $s_{x}(j/i)\mathcal{C}_{\rho,j/i,x}$~: $x_{3}$ is outside
this contour by construction and $x_{4}$ also lies outside
$s_{x}(j/i)\mathcal{C}_{\rho,j/i,x}$ for $\rho$ small enough,
because $|x_{4}|>x_{3}$, since we have supposed
$p_{11}+p_{-1-1}+p_{1-1}+p_{-11}<1$, see once again
Lemma~\ref{lemma_branched_points}. So (ii) is verified and (iii) is
also and obviously verified.
We can also construct $\mathcal{A}_{\rho,j/i,y}$ starting by $\mathcal{A}_{\rho,0,y}$.
Since $s_{y}(0)$ is a critical point of $X_{1}$, the level line
$\{y\in \mathbb{C}  : |X_{1}(y)|=x_{3}\}$ has a double point at
$s_{y}(0)$ and $\{y\in \mathbb{C}  : |X_{1}(y)|=x_{3}\} \setminus
\{s_{y}(0)\}$ has two connected components. Moreover, thanks to
Lemma~\ref{further_properties_X_Y}, the circle
$\mathcal{C}(0,s_{y}(0))\setminus \{s_{y}(0)\}$ lies in the domain
$\{y\in \mathbb{C}  : |X_{1}(y)|>x_{3}\}$. For this reason and since
$-s_{y}(0)\in]-y_{3},-y_{2}[\subset ]-|y_{4}|,-|y_{1}|[$, one can
clearly construct a contour $\mathcal{A}_{\rho,0,y}$ that verifies
(i), (ii) and (iii). Then, by continuity of the different quantities
w.r.t.\ $j/i$, one  can  build contours $\mathcal{A}_{\rho,j/i,y}$
verifying (i), (ii) and (iii) for all $j/i$ small enough.


Let us go back to~(\ref{G_i_j_after_residue}). Before passing from
double integrals in~(\ref{three_one}), (\ref{three_two}),
(\ref{three_three}) to the simple ones
in~(\ref{G_i_j_after_residue}), it will be convenient to subtract
the constant $h(x_{3})$ from the numerator of~(\ref{three_one}) and
add it to the numerator of~(\ref{three_two}). Next, we move the
contours $|x|=1$ and $|y|=1$ to $s_{x}(j/i)\mathcal{C}_{\rho,j/i,x}$
and $s_{y}(j/i)\mathcal{C}_{\rho,j/i,y}$~; thanks to Cauchy theorem
and since by construction the contours avoid the singularities of
the integrands, the value of the integrals is not changed. After the
change of variables $x \mapsto x s_{x}(j/i)$ and $y \mapsto y
s_{y}(j/i)$ in~(\ref{G_i_j_after_residue}), we get
$G_{i,j}^{n_{0},m_{0}}=-K_{i,j}^{n_{0},m_{0}}/(s_{x}(j/i)^{i}s_{y}(j/i)^{j})$,
where $K_{i,j}^{n_{0},m_{0}}$ is defined by~:
     \begin{eqnarray}\label{def_K}
          K_{i,j}^{n_{0},m_{0}}=\frac{s_{x}\left(j/i\right)}{2\pi \imath}
          \int_{\mathcal{C}_{\rho,j/i,x}}
          \frac{ h\left(s_{x}\left(j/i\right)x\right)-h\left(x_{3}\right)}
          {d\left(s_{x}\left(j/i\right)x\right)^{1/2}}
          e^{-i\kappa_{j/i}\left(x\right)}\text{d}x\hspace{50mm}\\
          +\frac{s_{y}\left(j/i\right)}{2\pi \imath}
          \int_{\mathcal{C}_{\rho,j/i,y}}
          \frac{\widetilde{h}\left(s_{y}\left(j/i\right)y\right)+h_{00}
          +h\left(x_{3}\right)-X_{1}\left(s_{y}\left(j/i\right)y\right)^{n_{0}}
          \left(s_{y}\left(j/i\right)y\right)^{m_{0}}}
          {\widetilde{d}\left(s_{y}\left(j/i\right)y\right)^{1/2}}
          e^{-i\widetilde{\kappa}_{j/i}\left(y\right)}\text{d}y.\nonumber
     \end{eqnarray}
We split $K_{i,j}^{n_{0},m_{0}}=K_{i,j,1}^{n_{0},m_{0}}+K_{i,j,2}^{n_{0},m_{0}}$,
where $K_{i,j,1}^{n_{0},m_{0}}$ (resp.\ $K_{i,j,2}^{n_{0},m_{0}}$)
is obtained from $K_{i,j}^{n_{0},m_{0}}$ by integrating only on the
contours $x_{j/i}([-\rho,\rho])$ and $y_{j/i}([-\rho,\rho])$ (resp.\
$\mathcal{A}_{\rho,j/i,x}$ and $\mathcal{A}_{\rho,j/i,y}$). We will
prove that the asymptotic $K_{i,j,1}^{n_{0},m_{0}}$ will lead to the
result announced in
Theorem~\ref{proposition_asymptotic_Green_functions_angle_zero} and
that $K_{i,j,2}^{n_{0},m_{0}}$ will be exponentially negligible, see~(\ref{k2}).

We start by studying $K_{i,j,2}^{n_{0},m_{0}}$. Consider the
following two quantities $S_{1,\delta}$ and $S_{2,\delta}$ and prove
that for $\delta$ sufficiently small they are finite.
     \begin{eqnarray*}
          S_{1,\delta}&=&\sup_{j/i\in\left[0,\delta\right]}
                \sup_{x\in\mathcal{C}_{\rho,j/i,x}}
                \left|\frac{ h\left(s_{x}\left(j/i\right)x\right)-h\left(x_{3}\right)}
                {d\left(s_{x}\left(j/i\right)x\right)^{1/2}}\right|,\\
          S_{2,\delta}&=&\sup_{j/i\in\left[0,\delta\right]}
                \sup_{y\in\mathcal{C}_{\rho,j/i,y}}
                \left|\frac{\widetilde{h}\left(s_{y}\left(j/i\right)y\right)+h_{00}
                +h\left(x_{3}\right)-X_{1}\left(s_{y}\left(j/i\right)y\right)^{n_{0}}
                \left(s_{y}\left(j/i\right)y\right)^{m_{0}}}
                {\widetilde{d}\left(s_{y}\left(j/i\right)y\right)^{1/2}}\right|.
     \end{eqnarray*}
The fact that for $\delta$ small enough $S_{1,\delta}$ is finite
comes from the three following properties. First, we recall from
Corollary~\ref{continuation_h_h_tilde} that $h$ is holomorphic on
$\mathbb{C}\setminus [x_{3},x_{4}]$ so that in particular $h$ is
bounded on every compact of $\mathbb{C}\setminus [x_{3},x_{4}]$.
This is why $(h(x)-h(x_{3}))/d(x)^{1/2}$ is bounded away from the
branch points $x_{i}$, $i\in\{1,\ldots,4\}$. Secondly, it can be
easily deduced from the proofs of Lemmas~\ref{lemma_residue_term}
and~\ref{lemma_int_phi} that in the neighborhood of $x_{3}$,
$h(x)=h(x_{3})+c(x_{3}-x)^{1/2}+\mathcal{O}(x_{3}-x)$, where $c$ is
some non zero constant. So the quantity $(h(x)-h(x_{3}))/d(x)^{1/2}$
is bounded in the neighborhood of $x_{3}$. Thirdly, the contours
$s_{x}(j/i)\mathcal{C}_{\rho,j/i,x}$ avoid, by construction, the
branch points $x_{1}$, $x_{2}$ and $x_{4}$. The fact that
$S_{2,\delta}$ is finite follows similarly~: first, by construction, the
contours $s_{y}(j/i)\mathcal{C}_{\rho,j/i,y}$ avoid the branch
points $y_{i}$, $i\in \{1,\ldots ,4\}$ and the function $\tilde{h}$
is bounded on every compact of $\mathbb{C}\setminus [y_{3},y_{4}]$~;
also, the poles of $X_{1}$ being isolated (they are at most two),
the contours $s_{y}(j/i)\mathcal{C}_{\rho,j/i,y}$ can be chosen
such that they remain away from these poles.
Recalling the properties (i) and (iii)  of $A_{\rho, j/i,x}$ and $A_{\rho, j/i,y}$
we deduce that for any
$i>0$  and any $j/i$ small enough,
     \begin{equation}
          \label{k2}
          \left|K_{i,j,2}^{n_{0},m_{0}}\right| \leq
          \left(L\left(\rho\right)/\left(2\pi\right)\right) 
          \left(x_{3}S_{1,\delta}+y_{3}S_{2,\delta}\right) 
          \exp\left(-ic\left(\rho\right)\right). 
     \end{equation}

Let us now turn to $K_{i,j,1}^{n_{0},m_{0}}$. Making
in~(\ref{def_K}) the change of variable $x=x_{j/i}(t)$ and
$y=y_{j/i}(t)$, we represent $K_{i,j,1}^{n_{0},m_{0}}$ as an
integral on the segment $[-\rho,\rho]$. Moreover,
using~(\ref{G_i_j_after_residue}) and the following equality -- that comes
from~(\ref{functional_equation}) and~(\ref{join_steepest_x_y}) --
     \begin{equation}\label{join_s_x_'_s_y_'}
          s_{x}\left(j/i\right)x_{j/i}'\left(-t\right)\widetilde{d}\left(s_{y}\left(j/i\right)
          y_{j/i}\left(-t\right)\right)^{1/2}=-s_{y}\left(j/i\right)y_{j/i}'\left(t\right)
          d\left(s_{x}\left(j/i\right)
          x_{j/i}\left(-t\right)\right)^{1/2},
     \end{equation}
we obtain that
$K_{i,j,1}^{n_{0},m_{0}}=\int_{-\rho}^{\rho} f_{j/i}(t)\exp(-i
t^{2})\text{d}t$, where
     \begin{equation*}
          f_{j/i}\left(t\right)=\left[
          h\left(X_{1}\left(\widehat{y}_{j/i}\left(t\right)\right)
          \right)+\widetilde{h}\left(\widehat{y}_{j/i}\left(t\right)\right)+
          h_{00}-X_{1}\left(\widehat{y}_{j/i}\left(t\right)
          \right)^{n_{0}}\widehat{y}_{j/i}\left(t\right)^{m_{0}}\right]
          {\widehat{y}_{j/i}}'\left(t\right)
          \widetilde{d}\left(\widehat{y}_{j/i}\left(t\right)\right)^{-1/2},
     \end{equation*}
and $\hat{y}_{j/i}(t)=s_{y}(j/i)y_{j/i}(t)$. In addition,
the formula~(\ref{relation_h_h_tilde}) yields that $f_{j/i}$ is equal to~:
     \begin{eqnarray}\label{def_f_j/i}
          f_{j/i}\left(t\right)&=&\left[\widetilde{h}
          \left(\widehat{y}_{j/i}\left(t\right)\right)-
          \widetilde{h}\left(\frac{c\left(X_{1}
          \left(\widehat{y}_{j/i}\left(t\right)\right)\right)}
          {a\left(X_{1}\left(\widehat{y}_{j/i}\left(t\right)\right)
          \right)\widehat{y}_{j/i}\left(t\right)}
          \right)
          \right.
          \\&-&X_{1}\left(\widehat{y}_{j/i}\left(t\right)
          \right)^{n_{0}}\left.\left(\widehat{y}_{j/i}\left(t\right)^{m_{0}}-
          \left(\frac{c\left(X_{1}\left(\widehat{y}_{j/i}\left(t\right)
          \right)\right) }{a\left(X_{1}\left(\widehat{y}_{j/i}\left(t\right)
          \right)\right)\widehat{y}_{j/i}\left(t\right)}
          \right)^{m_0}\right)\right]{\widehat{y}_{j/i}}'\left(t\right)
          \widetilde{d}\left(\widehat{y}_{j/i}\left(t\right)\right)^{-1/2}.
          \nonumber
     \end{eqnarray}
In particular, this representation~(\ref{def_f_j/i}), added to the
-- already noticed -- holomorphy of $y_{j/i}$ in
$\mathcal{D}(0,\rho)$, $\rho$ being independent of $j/i\in [0,1]$,
implies that $f_{j/i}$ is holomorphic in a disc of center $0$ and of
radius positive and independent of $j/i$ for $j/i$ small enough.
Therefore $f_{j/i}$ can be expanded in its Taylor series in the
neighborhood of $0$~: $f_{j/i}(t)=f_{j/i}(0)+t
f_{j/i}'(0)+t^{2}f_{j/i}''(0)/2+t^{3}f_{j/i}'''(0)/6+t^{4}g_{j/i}(t)$,
where $g_{j/i}$ is also holomorphic in some centered disc of radius
positive and independent of $j/i$ for $j/i$ sufficiently small.
Reducing eventually $\rho$ and $\delta$, we have
that $G=\sup_{j/i\in [0,\delta]} \sup_{t\in[-\rho,\rho]} |g_{j/i}(t)|$ is finite.
Then, applying Laplace's method, we obtain the bound
     \begin{equation}\label{def_K_1}
          \left|K_{i,j,1}^{n_{0},m_{0}}-\frac{1}{2\pi\imath}
          \left(\frac{\pi^{1/2}f_{j/i}\left(0\right)}{i^{1/2}}+
          \frac{\pi^{1/2}f_{j/i}''\left(0\right)}{4i^{3/2}}\right)\right|\leq
          \frac{3\pi^{1/2}G}{4i^{5/2}}
      + C \exp\left(-i\rho^{2}\right).
     \end{equation}
with some constant $C>0$  for any $i, j>0$, $j/i \in [0, \delta]$.

To conclude the analysis of $K_{i,j,1}^{n_{0},m_{0}}$, 
it remains to evaluate the asymptotic expansions
of $f_{j/i}(0)$ and $f''_{j/i}(0)$ as $j/i\to 0$.
Taking $t=0$ in~(\ref{def_f_j/i}) we derive
that $f_{j/i}(0)$ is equal to~:
     \begin{equation}\label{value_c_0}
          \frac{\widetilde{h}\left(s_{y}\left(j/i\right)\right)
          -\widetilde{h}\left(\frac{c\left(s_{x}\left(j/i\right)\right)}
          {a\left(s_{x}\left(j/i\right)\right)s_{y}\left(j/i\right)}\right)-
          s_{x}\left(j/i\right)^{n_{0}}\left(s_{y}\left(j/i\right)^{m_{0}}-
          \left(\frac{c\left(s_{x}\left(j/i\right)\right)}
          {a\left(s_{x}\left(j/i\right)\right)s_{y}\left(j/i\right)}\right)^{m_{0}}\right)}
          {\widetilde{d}\left(s_{y}\left(j/i\right)\right)^{1/2}}
          s_{y}\left(j/i\right)y_{j/i}'\left(0\right).
     \end{equation}
Note that since $s_{y}(0)^{2}=c(s_{x}(0))/a(s_{x}(0))$ (indeed,
$s_{y}(0)=Y_{1}(s_{x}(0))$ and $s_{x}(0)=x_{3}$), then for any suitable function $F$,
$F(s_{y}(j/i))-F(c(s_{x}(j/i))/(a(s_{x}(j/i)))s_{y}(j/i))=
2F'(s_{y}(0))s_{y}'(0)j/i+o(j/i)$, $s_{y}'(0)$ being
obtained from Lemma~\ref{rc}. We use this fact  successively with
$F(y)=\tilde{h}(y)$ and $F(y)=y^{m_{0}}$ to expand the numerator of
(\ref{value_c_0}) as $2\big(m_{0}s_{x}(0)^{n_{0}}s_{y}(0)^{m_{0}-1}-
\tilde{h}'(s_{y}(0))\big)s'_{y}(0)j/i+o(j/i)$, $j/i \to 0$.
The functions $s_x(j/i)$ and $s_y(j/i)$ being continuous
on $[0, +\infty]$,
the Taylor coefficients of $y_{j/i}(t)$ depend continuously
 on $j/i$, so that $y_{j/i}'(0) \to y'_0(0)$ as $j/i \to 0$.
   Let us compute the value of $y'_{0}(0)$. To get it, we differentiate
twice~(\ref{def_steepest_parametrization})~; this yields
     \begin{equation}\label{value_y_'_0_first}
          y_{j/i}'\left(0\right)^{2}\left.\frac{\text{d}^{2}}{\textnormal{d}y^{2}}
          \left\{ \frac{X_{1}\left(s_{y}\left(j/i\right)y\right)}{s_{x}\left(j/i\right)}
          y^{j/i}\right\} \right|_{y=1}=2.
     \end{equation}
In addition,  an explicit calculation gives~:
     \begin{equation}\label{value_y_'_0_second}
          \lim_{j/i\to 0}\left.\frac{\text{d}^{2}}{\textnormal{d}y^{2}}
          \left\{ \frac{X_{1}\left(s_{y}\left(j/i\right)y\right)}
          {s_{x}\left(j/i\right)}
          y^{j/i}\right\} \right|_{y=1}=
          \frac{s_{y}\left(0\right)^{2}}{s_{x}\left(0\right)}
          X_{1}''\left(s_{y}\left(0\right)\right).
     \end{equation}
(\ref{value_y_'_0_first}) and~(\ref{value_y_'_0_second}) imply
$y_{0}'(0)=\imath \sqrt{2}
s_{x}(0)^{1/2}(-X_{1}''(s_{y}(0)))^{-1/2}s_{y}(0)^{-1}$,
the $\imath$ coming from the fact that $X_1''(s_y(0))$ is negative.
Hence, we obtain
that $f_{j/i}(0)=l_{1} j/i(1+o(1))$, where
     \begin{equation}\label{asymptotic_c_0}
          l_{1}=-\imath
          \frac{2 \sqrt{2 } s_{y}'\left(0\right)s_{x}\left(0\right)^{1/2}
          \left(m_{0}s_{x}\left(0\right)^{n_{0}}s_{y}\left(0\right)^{m_{0}-1}-
          \widetilde{h}'\left(s_{y}\left(0\right)\right)\right)}
          {\left(-\widetilde{d}\left(s_{y}\left(0\right)\right)
          X_{1}''\left(s_{y}\left(0\right)\right)\right)^{1/2}}.
     \end{equation}
The Taylor coefficients of $y_{j/i}(t)$ depending continuously
 on $j/i$,  so do those of $f_{j/i}(t)$.
    Then $f''_{j/i}(0) \to f''_0(0)$ as $j/i \to 0$.
But $f(t)$ is an odd function on $[-\rho, \rho]$. To see this,
first, remark that~(\ref{def_Q_alternative}) yields
$c(\hat{x}_{0}(t))/(a(\hat{x}_{0}(t))Y_{1}(\hat{x}_{0}(t)))
=Y_{0}(\hat{x}_{0}(t))$. Moreover,
$Y_{1}(\hat{x}_{0}(-t))=Y_{0}(\hat{x}_{0}(t))$, so that
$\hat{x}_{0}$ being even ($\hat{x}_{0}(t)$ is equal to
$x_{3}\exp(t^{2})$), the function within the brackets
in~(\ref{def_f_j/i}) is odd. In addition,
using~(\ref{join_s_x_'_s_y_'}) we obtain that
$\hat{y}_{0}'(t)\tilde{d}(\hat{y}_{0}(t))^{-1/2}=
-\hat{x}_{0}'(-t)d(\hat{x}_{0}(-t))^{-1/2}$. Being the product of
two odd functions, $\hat{x}_{0}'d(\hat{x}_{0})^{-1/2}$ is even, so
that $\hat{y}_{0}'\tilde{d}(\hat{y}_{0})^{-1/2}$ is also even. This
implies that $f_{0}$ is odd and as an immediate consequence
$f_{0}''(0)=0$ and $f_{j/i}''(0)=o(1)$ as $j/i \to 0$.
Bringing together  (\ref{k2}), (\ref{def_K_1}), the expansions $f_{j/i}(0)=
l_{1}j/i(1+o(1))$ and $f''_{j/i}(0)=o(1)$
with $l_{1}$ defined in (\ref{asymptotic_c_0}) we obtain :
    \begin{equation}
         \label{zhj}
         G_{i,j}^{n_{0},m_{0}}=\frac{-1}{s_{x}\left(j/i\right)^{i} s_{y}\left(j/i\right)^{j}}
         \left(\frac{1}{2\pi \imath} \frac{\pi^{1/2} l_{1} j/i 
         \left(1+o\left(1\right)\right)}{i^{1/2}}
         +o\left(i^{-3/2}\right)\right) + 
         \mathcal{O}\big(\exp\left(-i c\left(\rho\right)\right)\big), 
         \ \ j/i \to 0.
     \end{equation}
This concludes  the proof of the theorem in the case
$p_{11}+p_{-1-1}+p_{1-1}+p_{-11}<1$.

We now briefly explain the notable differences
in the case $p_{11}+p_{-1-1}+p_{1-1}+p_{-11}=1$.
In this case, all the functions considered are odd or even~:
for instance $Y_{i}$ and $X_{i}$, $i=0,1$ are odd,
$d$ and $\tilde{d}$ are even (see
Lemma~\ref{further_properties_X_Y}), $h$ and $\tilde{h}$ have the parity
of $n_{0}+m_{0}$. In particular, it is immediate
from~(\ref{G_i_j_after_residue}) that if
$i+j$ and $n_{0}+m_{0}$ don't have the same parity, then
$G_{i,j}^{n_{0},m_{0}}=0$.
If they have the same parity, then
we can obtain the asymptotic of the Green functions
with essentially the same analysis as in the
case $p_{11}+p_{-1-1}+p_{1-1}+p_{-11}<1$, the only significant change being that
we have now to take under account the contribution of
two critical points~: $(s_{x}(j/i),s_{y}(j/i))$, as before, but now also
$(-s_{x}(j/i),-s_{y}(j/i))$.
In particular, the new contour of integration $\mathcal{C}_{\rho,j/i,x}$
(resp.\ $\mathcal{C}_{\rho,j/i,x}$) have
to go at once through $s_{x}(j/i)$ and through $-s_{x}(j/i)$
(resp.\ $s_{y}(j/i)$ and through $-s_{y}(j/i)$), for instance
they can be taken symmetrical w.r.t.\ the imaginary axis.
This fact underlies that the asymptotic of the
$G_{i,j}^{n_{0},m_{0}}$ is, in this case,
twice, in accordance with the conclusions
of Theorem~\ref{proposition_asymptotic_Green_functions_angle_zero}.
\end{proof}

The following result has been used in the proof of
Theorem~\ref{proposition_asymptotic_Green_functions_angle_zero}.

\begin{lem}\label{further_properties_X_Y}
     If $p_{11}+p_{-1-1}+p_{-11}+p_{1-1}<1$,
     for all $y\in\mathcal{C}(0,s_{y}(0))\setminus \{s_{y}(0)\}$,
     $|X_{1}(y)|>x_{3}$.
     If $p_{11}+p_{-1-1}+p_{-11}+p_{1-1}=1$,
     for all $y\in\mathcal{C}(0,s_{y}(0))\setminus \{\pm s_{y}(0)\}$,
     $|X_{1}(y)|>x_{3}$.
\end{lem}

\begin{proof}
As a direct consequence of Lemma~\ref{lemma_branched_points},
we obtain that $X_{1}$ is meromorphic in the 
neighborhood of every point of $]-y_{3},-y_{2}[$,
since $]-y_{3},-y_{2}[\subset ]-|y_{4}|,-|y_{1}|[$.
Let us now show that 

(i)        if $p_{11}+p_{-1-1}+p_{-11}+p_{1-1}<1$, then
           for all $y\in [y_{2},y_{3}]$,
           $|X_{1}(-y)|>|X_{1}(y)|$,

(ii)       if $p_{11}+p_{-1-1}+p_{-11}+p_{1-1}=1$, then
           for all $y\in [y_{2},y_{3}]=[-y_{1},-y_{4}]$,
           $X_{1}(-y)=-X_{1}(y)$.

Start by remarking that
$X_{1}(y)\leq 0$ (resp.\ $X_{1}(y)\geq 0$) on
$[y_{4},y_{1}]$ (resp.\ on $[y_{2},y_{3}]$), this is proved in~\cite{FIM}.
Thus, on $[y_{2},y_{3}]$, the unique possibility to have
$|X_{1}(-y)|=|X_{1}(y)|$ is that $X_{1}(y)=-X_{1}(-y)$.
After calculation, we obtain that this is equivalent to
$(p_{01}y^{2}+p_{0-1})^{2}( p_{11}y^{2}+p_{1-1} ) (p_{-11}y^{2}+p_{-1-1}) +
\{(p_{-10}p_{1-1}-p_{10}p_{-1-1})^{2}+p_{0-1} (p_{-10}p_{1-1}+p_{10}p_{-1-1}) \} y^{2}+
\{p_{-10}p_{10}(1-2(p_{1-1}p_{-11}+p_{-1-1}p_{11} ))+p_{-10}p_{1-1}
(p_{01}+p_{-10}p_{11})+p_{-10}p_{11}(p_{0-1}+p_{-10}p_{1-1}) +
p_{10}p_{-1-1}(p_{01}+p_{10}p_{-11})+p_{10}p_{-11} (p_{0-1}+p_{10}p_{-1-1}) \}y^{4}+
\{(p_{10}p_{-11}-p_{-10}p_{11})^{2}+p_{01}(p_{10}p_{-11}+p_{-10}p_{11})\}y^{6}=0$.
If $p_{11}+p_{-1-1}+p_{1-1}+p_{-11}<1$ 
(resp.\ $p_{11}+p_{-1-1}+p_{1-1}+p_{-11}=1$), then the previous equality
holds for none (resp.\ any) $y\in [y_{2},y_{3}]$. Therefore (ii) is proved.
To prove (i), we remark that an explicit calculation leads to
$|X_{1}(-1)|>1=X_{1}(1)$ so that by continuity, for all $y \in [y_{2},y_{3}]$,
$|X_{1}(-y)|>|X_{1}(y)|$.

We prove now Lemma~\ref{further_properties_X_Y} in the case
$p_{11}+p_{-1-1}+p_{1-1}+p_{-11}<1$. We will show that
$\mathcal{C}(0,s_{y}(0))\cap \{y\in \mathbb{C} :
|X_{1}(y)|=x_{3}\}=\{s_{y}(0)\}$. This suffices since one hand, this
implies that either for all $y\in \mathcal{C}(0,s_{y}(0))\setminus
\{s_{y}(0)\}$, $|X_{1}(y)|>x_{3}$ or for all $y\in
\mathcal{C}(0,s_{y}(0))\setminus \{s_{y}(0)\}$, $|X_{1}(y)|<x_{3}$~;
but on the other hand, thanks to (i),
$|X_{1}(-s_{y}(0))|>X_{1}(s_{y}(0))=x_{3}$, so that by continuity we
will conclude. Let $y^{*}\in \mathcal{C}(0,s_{y}(0))$ be such that
$|X_{1}(y^{*})|=x_{3}$. Setting
$\hat{x}=X_{1}(y^{*})/x_{3}$, $\hat{y}=y^{*}/s_{y}(0)$ and using
$Q(x_{3},s_{y}(0))=0$, $Q(X_{1}(y^{*}),y^{*})=0$, we obtain
$\hat{Q}(\hat{x},\hat{y})=0$, where
$\hat{Q}(\hat{x},\hat{y})=(\sum_{i,j}\hat{p}_{i j}
\hat{x}^{i}\hat{y}^{j}-1)\hat{x}\hat{y}$
and $\hat{p}_{i j}=p_{i j}x_{3}^{i}s_{y}(0)^{j}$. In particular,
for all $i$ and $j$, $\hat{p}_{i j}>0$ and $\sum_{i,j}\hat{p}_{i
j}=1$, since $Q(x_{3},s_{y}(0))=0$. But from elementary
considerations about sums of complex numbers, having simultaneously
$\sum_{i,j}\hat{p}_{i j}=1$, $\sum_{i,j}\hat{p}_{i
j}\hat{x}^{i}\hat{y}^{j}=1$ and $|\hat{x}|=|\hat{y}|=1$
leads necessarily to $\hat{x}=\hat{y}=1$, so that
$y^{*}=s_{y}(0)$. 

We conclude the proof in the case
$p_{11}+p_{-1-1}+p_{1-1}+p_{-11}=1$ by using similar arguments and
the fact that $x_{4}=-x_{3}$.
\end{proof}

\begin{rem}
Thanks to Proposition~\ref{continuity_Martin_boundary}, it is
immediate now that for any $n_{0},m_{0},n_{1},m_{1}>0$ ~:
     \begin{equation}\label{first_equality_continuity_Martin}
        \lim_{i,j>0, j/i \to 0}
        \frac{G_{i,j}^{n_{0},m_{0}}}{G_{i,j}^{n_1,m_1}}
          =  \lim_{i\to \infty} \frac{h_{i}^{n_{0},m_{0}}}{h_{i}^{n_{1},m_{1}}}=
          \frac{m_{0}s_{x}\left(0\right)^{n_{0}}s_{y}\left(0\right)^{m_{0}-1}-
          {\widetilde{h}^{n_{0},m_{0}\hspace{0.2mm}'}}\left(s_{y}\left(0\right)\right)}
          {m_{1}s_{x}\left(0\right)^{n_{1}}s_{y}\left(0\right)^{m_{1}-1}-
          {\widetilde{h}^{n_{1},m_{1}\hspace{0.2mm}'}}\left(s_{y}\left(0\right)\right)},
     \end{equation}
the probabilities of absorption $h_{i}^{n_{0},m_{0}}$ being
defined in~(\ref{absorption_probabilities}). In addition,
using~(\ref{relation_h_h_tilde}) and Lemma~\ref{rc}, we obtain that
as $\gamma$ goes to zero, the limit  $\lim_{i,j>0, j/i \to \tan( \gamma)}
G_{i,j}^{n_{0},m_{0}}/G_{i,j}^{n_{1},m_{1}}$
converges also to the right member
of~(\ref{first_equality_continuity_Martin}). In other words, the Martin kernel
is continuous at $0$ and likewise, we verify that it is also
continuous at $\pi/2$.
\end{rem}

\newpage
\footnotesize

\bibliographystyle{alpha}
\bibliography{RW_IK_KR}

\end{document}